\documentclass[a4paper, 11pt]{article}

\usepackage[utf8]{inputenc}

\usepackage[shortlabels]{enumitem}
\usepackage{amsmath, appendix, ulem}
\usepackage{algorithm}
\usepackage{algorithmic}
\usepackage[nobysame]{amsrefs}
\usepackage{amssymb, color}
\usepackage[margin=1in]{geometry}
\usepackage{mathrsfs}
\usepackage{amssymb}
\usepackage{amsfonts}
\usepackage{graphicx}
\usepackage{subfig}
\usepackage{float}
\usepackage{authblk}
\usepackage{amsthm}
\usepackage{color}
\usepackage{hyperref}
\usepackage{titletoc}
\usepackage{grffile}
\usepackage{relsize}
\usepackage{empheq}

\graphicspath{{../Figures/}}

\numberwithin{equation}{section}

\newtheorem{lem}{Lemma}[section]
\newtheorem{thm}{Theorem}[section]
\newtheorem{proposition}[thm]{Proposition}
\newtheorem{cor}[thm]{Corollary}
\newtheorem{rmk}{Remark}[section]
\newtheorem{definition}[thm]{Definition}

\renewcommand{\tilde}{\widetilde}
\renewcommand{\hat}{\widehat}
\renewcommand{\bar}{\overline}

\newtheorem{assum}{Assumption}

\newcommand{\nn}{\nonumber}

\newcommand{\dt}{ \, {\rm d} t}
\newcommand{\dx}{ \, {\rm d} x}
\newcommand{\dy}{ \, {\rm d} y}

\newcommand{\ds}{\, {\rm d} s}

\newcommand{\MF}{{\mathcal{F}}}

\newcommand{\mc}[1]{\mathcal{#1}}

\newcommand{\MC}{\mathcal{C}}
\newcommand{\EE}{\mathbb{E}}
\newcommand{\RR}{\mathbb{R}}
\newcommand{\NN}{\mathbb{N}}

\newcommand{\ud}{\,\mathrm{d}}
\newcommand{\rd}{\mathrm{d}}

\newcommand{\la}{\langle}
\newcommand{\ra}{\rangle}

\newcommand{\abs}[1]{\left\lvert#1\right\rvert}
\newcommand{\norm}[1]{\left\lVert#1 \, \right\rVert}

\newcommand{\Hbb}{\mathbb{H}}
\newcommand{\R}{\mathbb{R}}

\newcommand{\J}{\mathbb{J}}

\newcommand{\Lcal}{\mathcal{L}}

\newcommand{\Pcal}{\mathcal{P}}
\newcommand{\Ccal}{\mathcal{C}}

\newcommand{\Fcal}{\mathcal{F}}

\newcommand{\Rcal}{\mathcal{R}}

\newcommand{\supp}{\textnormal{supp}}

\newcommand{\Lip}{\textnormal{Lip}}

\newcommand{\BPhi}{\boldsymbol{\Phi}}

\newcommand{\INTDom}[3]{\int_{#2} #1 \textnormal{d} #3}
\newcommand{\INTSeg}[4]{\int_{#3}^{#4} #1 \textnormal{d} #2}

\newcommand{\derv}[2]{\frac{\textnormal{d} #1}{ \textnormal{d} #2}}

\DeclareUnicodeCharacter{202A}{ } 

\usepackage[dvipsnames]{xcolor}

\title{A Measure Theoretical Approach to the Mean-field Maximum Principle for Training NeurODEs}

\date{}
\author[1]{Beno\^it Bonnet\thanks{Email: \texttt{benoit.a.bonnet@inria.fr}}}
\affil[1]{Inria Paris and Laboratoire Jacques-Louis Lions, Sorbonne Universit\'e, Universit\'e Paris-Diderot SPC, CNRS, Inria, 75005 Paris, France}

\author[2]{Cristina Cipriani\thanks{Email: \texttt{cristina.cipriani@ma.tum.de}}}
\author[3]{Massimo Fornasier\thanks{Email: \texttt{massimo.fornasier@ma.tum.de}} }
\affil[2,3]{Technical University Munich, Department of Mathematics, Munich,
Germany}
\affil[2,3]{Munich Data Science Institute, Munich, Germany}

\author[4]{Hui Huang\thanks{Email: \texttt{hui.huang1@ucalgary.ca}}}
\affil[4]{University of Calgary, Department of Mathematics and Statistics, Calgary,
Canada}

\begin{document}

\maketitle
\begin{abstract}
In this paper we consider a measure-theoretical formulation of the training of NeurODEs in the form of a mean-field optimal control with $L^2$-regularization of the control. We derive first order optimality conditions for the NeurODE training problem in the form of a mean-field maximum principle, and show that it admits a unique control solution, which is Lipschitz continuous in time. As a consequence of this uniqueness property, the mean-field maximum principle also provides a strong quantitative generalization error for finite sample approximations, yielding a rigorous justification of the double descent phenomenon. Our derivation of the mean-field maximum principle is much simpler than the ones currently available in the literature for mean-field optimal control problems, and is based on a generalized Lagrange multiplier theorem on convex sets of spaces of measures. The latter is also new, and can be considered as a result of independent interest.
\end{abstract}
\tableofcontents


\section{Introduction}

\subsection{Deep learning}

Deep learning is an established computational approach that performs state-of-the-art on various relevant real-life applications such as speech \cite{hannun2014deep} and image  \cite{NIPS2012_4824,7780459} recognition, language translation \cite{NIPS2017_7181}, and which also serves as a basis for novel scientific computing methods \cite{Berner2020,elbraechter2020dnn}. In unsupervised machine learning, deep neural networks have shown great success as well, for instance in image and speech generation \cite{pmlr-v48-oord16,oord2016wavenet}, and in reinforcement learning for solving control problems, such as mastering Atari games \cite{nature15} or beating human champions at playing Go \cite{silver2017mastering}. Deep learning is about realizing complex tasks as the ones mentioned above, by means of highly parametrized functions, called deep artificial neural networks $\mathcal N :\mathbb R^{d_0} \to \mathbb R^{d_L}$. A classical architecture is the one of feed-forward artificial neural networks of the type
\begin{equation}
\label{nndef}
\mathcal N(x)=\rho \big(W_L^\top \rho \big(W_{L-1}^\top \dots \rho\big(W_1^\top x + \tau_1\big)\dots \big)+\tau_L\big),
\end{equation} 
where the matrices $W_{\ell} \in \R^{ d_{\ell-1} \times d_{\ell}}$ represent collections of weights, the vectors $\tau{_\ell} \in \R^{d_{\ell}}$ are shifts/biases for each layer $\ell = 1,\ldots,L$ and $\rho$ is a scalar activation function acting component-wisely on vectors. Below, we shall denote by $\mathcal  F(X) := \rho (W^\top X + \tau)$ a generic layer of the network. In practical applications, the number $L \geq 1$ of layers -- determining the depth of the network -- and the dimensions $d_{\ell-1}\times d_{\ell}$ of the weight matrices $W_{\ell}$  are typically determined by means of heuristic considerations, whereas the weight matrices and the shifts are free parameters which are tuned in various possible ways by using a given training dataset.

Practical evidences towards certified benchmarks confirm that deep-learning algorithms are able to outperform many of the previously existing methods. Also, recent mathematical investigations \cite{mhaskar2020function,shaham2018provable,Berner2020,elbraechter2020dnn,elbrachter2020deep,daubechies2019,PETERSEN2018296,cloninger2020relu,devore2020neural,mhaskar2016deep} have proven that deep artificial networks can approximate  high dimensional functions without incurring in the curse of dimensionality, i.e. without needing a number of parameters (here the weights and shifts of the network) that is exponential with respect to the input dimension in order to approximate  high-dimensional functions. While the approximation properties -- also called the \textit{expressivity} -- of neural networks are becoming more and more understood and transparent \cite{gühring2020expressivity}, the training phase itself, based on suitable optimization processes, remains a (black-)box with some levels of opacity. 
In fact, the latter procedure features a surprising and yet mostly unexplained phenomenon, which is in stark contrast with conventional statistics wisdom: in addition to providing a finer empirical data fitting, increasing the number of modelling parameters beyond that of training examples also tends to improve the {\it generalization error}, namely the prediction error on unseen data. This simultaneous decrease of both empirical and generalization errors is called the {\it double descent} phenomenon.
Instead, from classical statistical learning theory \cite{books/daglib/0033642}, one would expect that overfitting should lead to a blow-up of the generalization error, owing to the wealth of complextiy of the underlying model \cite{zhang2016understanding}. Hence the prediction of the generalization error from data remains at large a fundamental open problem in deep learning. As one of the main results of this paper, we show that for certain classes of neural networks based on dynamical systems, whose training is reformulated as an optimal control problem, the double descent phenomenon can be rigorously explained.

%


\subsection{Training of deep nets and residual blocks}

In order to understand the context of our results, let us mention how the neural networks considered in this paper arise. We start by recalling how training of neural networks is performed and how it is facilitated by appropriate network architectures.
The method that is most frequently used to train deep neural networks is the so-called \textit{backpropagation of error} \cite{werbos74,10.5555/104279.104293,7fa6b6a5cde14bcfbd7ab3a8f19d0d56}, which is justified by its tremendous empirical success. Inherently, all the practical advances recalled above are due to the efficacy of this method. The term backpropagation\footnote{In fact, ``backpropagation'' refers more precisely to a recursive way of applying the chain rule needed to compute the gradient of the loss with respect to weights, but it is often used also to describe any algorithmic optimization procedure resorting to such gradients. In many cases, these latter are computed using symbolic calculus.} usually refers to the use of stochastic gradient descent or some of its variants \cite{sun2019optimization} to minimize a given loss function (e.g. mean-squared distances, Kullback-Leibler divergences, or Wasserstein distances) over the parameters of the network (the weights and biases), usually measuring the misfit of input-output information over a finite number of labeled training samples. On the one hand, the practical efficiency of deep learning is currently ensured in the so-called overparametrized regime by fitting a large amount of data with a larger amount of parameters. On the other hand, solving learning problems with very large numbers of layers gets increasingly harder with the total depth of the network, as the resulting non-convex optimization problems become in turn very high-dimensional. 

In their groundbreaking work \cite{7780459}, He et al. showed that the training error of the $56$-layer CNN network remains worse than that of a $20$-layer network for the same problem, highlighting an issue which could be blamed either on the optimization function, on initialization of the network, or on the vanishing/exploding gradient phenomenon. The problem of training very deep networks has been alleviated with the introduction of a new neural network layer called the ``Residual Block'', see Figure \ref{fig:resblock}. 
\begin{figure}
\begin{center}
\includegraphics[width=0.3\textwidth]{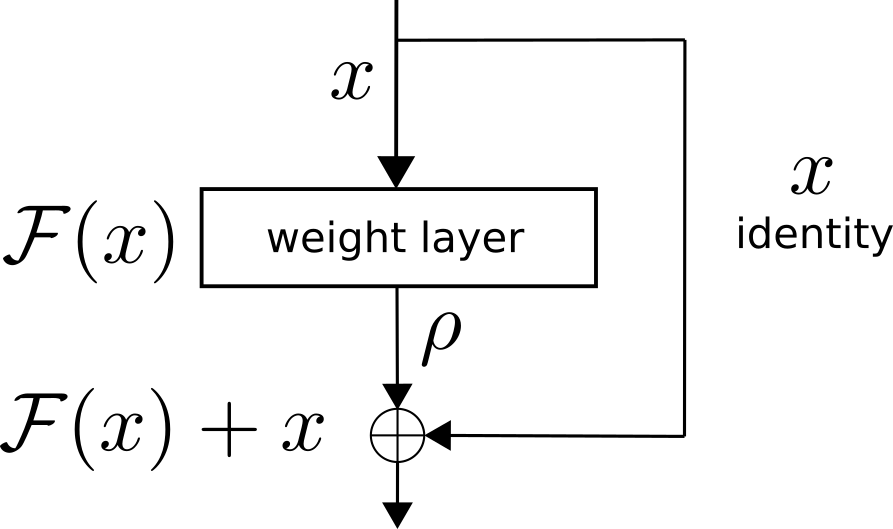}
\end{center}
\caption{The layer update reads: $X^{n+1} = X^n + \MF(X^n)$, see \cite{7780459}.}\label{fig:resblock}
\end{figure}
According to the analysis conveyed in \cite{he2016identity}, the use of identity  mappings as  skip connections and after-addition activations of the form
\begin{equation} \label{euler}
X^{n+1} = X^n +  \mathcal F(X^n)
\end{equation}
turns out to be beneficial to promote the smoothness of the information propagation. Therein, the authors present several $1000$-layer  deep networks that can be easily trained and achieves improved accuracy. The use of such skip connections  with identity mappings presupposes a rectangular shape of the network for which the depths $d_{\ell+1} = d_\ell$ of the layers are all identical.

 
\subsection{NeurODEs and stochastic optimal control}

While originally the arguments in \cite{he2016identity} that support the use of residual blocks are based on empirical considerations, a recent line of research has been devoted to a more mathematical and rigorous formulation of deep neural networks with residual blocks in terms of dynamical systems. In this context, the training of the network can be interpreted as a large optimal control problem, an insight that was proposed independently by E Weinan  \cite{eweinan17} and Haber-Ruthotto \cite{Haber_2017}. Later on, this dynamical approach has been greatly popularized in the machine learning community under the name of \textit{NeurODE} by Chen et al. \cite{10.5555/3327757.3327764}, see also \cite{DBLP:journals/corr/abs-1908-10920}. The formulation starts by reinterpreting the iteration \eqref{euler} as a step of the discrete-time Euler approximation \cite{avelin2020neural} of the following dynamical system
\begin{equation}\label{SDE}
\dot{X}_t=\mathcal F(t,X_t,\theta_t)\,,
\end{equation}
with initial condition $X_0\in \RR^d$. Here, the map $\mathcal F:\RR_+\times\RR^d\times \RR^m\rightarrow \RR^d$ represents the feed-forwarding dynamics, the parameter $\theta_t \in \mathbb R^m$ is a general control variable, which encodes the weights and shifts of the network, i.e. $\theta_t := (W_t,\tau_t)$. A prototypical example is given by
\begin{equation}\label{classical}
\mathcal F(t,X_t,\theta_t) = \rho ( W_t X_t + \tau_t),
\end{equation}
for instance with an activation function $\rho := \tanh$ acting componentwisely on its entries. In \cite{eweinan17,eweinan19}, the authors proposed a {\it stochastic control formulation} of the training of this nonlinear process, with a detailed analysis of the related optimality conditions. Therein, both the the Hamilton-Jacobi-Bellman equations \cite{CannarsaS2004} -- based on the well-known dynamic programming principle -- and the Pontryagin Maximum Principle \cite{Pontryagin} were studied in great generality. From another perspective, several recent works \cite{Agrachev2020,Agrachev2021,Tabuada2021} in geometric control theory have aimed at explaining the efficiency of NeurODEs in approximating large classes of mappings in terms of controllability properties of such systems in the group of diffeomorphisms.  
 
In this paper, we focus on a particular measure theoretical reformulation of the general approach developed by E Weinan et al. \cite{eweinan19}, which allows us to derive more specific properties of the control problem, such as the existence, uniqueness, and smoothness of solutions to the Pontryagin Maximum Principle, and a strong form of generalization error estimates. Most importantly, our approach encompasses the prototypical model \eqref{classical} as a possible application. Consider two random variables $X_0$ and $Y_0$ which are jointly distributed according to a law $\mu_0 \in \Pcal(\R^{2d})$, and let us fix the depth $T>0$ of the time-continuous neural network \eqref{SDE}. Training this network then amounts to learning the control signals $\theta \in L^2([0,T];\R^m)$ in such a way that the terminal output $X_T$ of \eqref{SDE} is close to $Y_0$, with respect to some distortion measure $\ell(\cdot,\cdot)\in \MC^2$. A typical choice is $\ell(x,y)=:|x-y|^2$, which is often called the {\it squared loss function} in the machine learning literature. The stochastic optimal control problem can hence be posed as
\begin{equation}
\label{Ecost}
\inf\limits_{\theta \in L^2([0,T];\R^m)} J(\theta) = \left\{ 
\begin{aligned}
& \inf \limits_{\theta \in L^2([0,T];\R^m)} \EE_{\mu_0} \Big[ \ell(X_T,Y_0) \Big] +\lambda\int_0^T|\theta_t|^2\dt , \\
& \hspace{0.2cm} \text{s.t.} ~ \left\{
\begin{aligned}
& \dot X_t = \Fcal (t,X_t,\theta_t), \\
& (X_t,Y_0)_{|t =0} \sim \mu_0.
\end{aligned}
\right.
\end{aligned}
\right.
\end{equation}
The use of a regularization term of the type $\lambda\int_0^T|\theta_t|^2\dt$ is very standard in machine learning, see e.g. \cite[Chapter 7]{Goodfellow-et-al-2016} or \cite[Section 6]{Kukacka-et-al-2017}. In the absence of regularization, the resulting trained networks may have huge Lipschitz constants, rendering them extremely unstable and susceptible to adversarial attacks \cite{Goodfellow2015ExplainingAH}. Additionally, the regularization may significantly help the usual training processes, by making the loss $J$ increasingly more convex. As we shall see more in details below, such a standard regularization will allow us to establish  the  existence and uniqueness of solutions for \eqref{Ecost}, as well as their continuity with respect to the data, which provides a rigorous explaination to the stability of trained networks and the double descent phenomenon. Conversely, we shall also demonstrate numerically in Section \ref{sec:numres} that the lack of a sufficient regularization causes significant instabilities in the numerical solution of the optimal control problem \eqref{Ecost}, see Figure \ref{fig:unimodal_problems}, rendering the latter absolutely essential from a practical standpoint. Other and more general regularizations are of course possible \cite{Kukacka-et-al-2017}, but for the sake of simplicity and clarity in the exposition, we shall restrict our attention to this specific one.

\subsection{Measure-theoretical approach to mean-field optimal control}

In this paper, we develop a new point of view that is equivalent to that of \cite{eweinan19}, but which is not based on stochastic control considerations. We start by providing a measure-theoretic reformulation of \eqref{Ecost}, which can be interpreted as a generalized optimal transport problem or mean-field optimal control problem. 
To the best of our knowledge, the present paper is the first in the literature to make such a connection. To this end,  let us define a new stochastic process $Z_t :=(X_t,Y_t)$ satisfying
\begin{align}
\dot{X}_t=\mathcal F(t,X_t,\theta_t) \qquad \text{and} \qquad \dot{Y}_t=0,
\end{align}
with initial data $(X_0,Y_0)$ distributed according to $\mu_0$, and denote the law of $(X_t,Y_t)$ by $\mu_t(x,y)$. It is well-known that $\mu_t$ satisfies the following partial differential equation 
\begin{equation}
\label{PDE}
\partial_t\mu_t+\nabla_x\cdot (\mathcal F(t,x,\theta_t)\mu_t)=0, \qquad \mu_t|_{t=0}=\mu_0\, ,
\end{equation}
understood in the sense of distributions as in Definition \ref{defweak} below. With this transport equation at hand, we can recast the stochastic optimal control problem \eqref{Ecost} as
\begin{equation}
\label{Pcost}
\inf_{\theta \in L^2([0,T];\R^m)} J(\theta) = 
\left\{
\begin{aligned}
&\inf_{\theta \in L^2([0,T];\R^m)} \int_{\RR^{2d}}\ell(x,y)\rd\mu_T(x,y)+\lambda\int_0^T|\theta_t|^2\dt\, ,  \\
& \hspace{0.2cm} \text{s.t.} ~ \left\{
\begin{aligned}
& \partial_t\mu_t+\nabla_x\cdot (\mathcal F(t,x,\theta_t)\mu_t)=0\, , \\
& \mu_t|_{t=0}=\mu_0\,.
\end{aligned}
\right.
\end{aligned}
\right.
\end{equation}
Therein, the goal is again is to find the control signal $\theta$ for which  $J(\theta)$ is minimal when $\mu$ satisfies the PDE constraint \eqref{PDE}. Observe that when the initial measure $\mu_0$ is empirical, i.e.
\begin{equation*}
\mu_0 := \mu_0^N = \frac{1}{N} \sum_{i=1}^N \delta_{(X_0^i,Y_0^i)}
\end{equation*}
then the optimal control problem \eqref{Pcost} reduces to a classical finite particle optimal control problem with ODE constraints. 

Optimal control problems over spaces of probability measures of the form \eqref{Pcost} have been recently explored, mostly in the absence of final-point constraints and in the context of multi-agent interactions. The first contributions on this topic \cite{fornasier2014mean,MR3268059} were concerned with the rigorous convergence of classical finite particle optimal controls towards their mean-field counterparts, see also the more recent work \cite{LipReg,cavagnari2020lagrangian,fornasier_lisini_orrieri_savaré_2019}. The derivation of first order optimality conditions, i.e., the so-called Pontryagin Maximum Principle (PMP), has been proposed for the first time in \cite{BFRS15} based on the leader-follower model studied in \cite{MR3268059}. In this work, the mean-field Pontryagin Maximum Principle is derived as limit of its classical finite-particle counterpart. The first general derivation of the PMP for mean-field optimal control problems was obtained in \cite{PMPWass}, and is based on a careful adaptation of the strategy of needle-variations to the abstract geometric structure of Wasserstein spaces. These results were further extended in \cite{PMPWassConst} to problems with general final-point and running state constraints. In the latter contribution, the proof strategy combines a finite-dimensional non-smooth multipliers rule and outer-approximations of optimal trajectories by countable families of curves generated using needle-variations. Very recently, a simpler approach has been proposed in \cite{SetValuedPMP}, by adapting to the notion of multivalued dynamics in Wasserstein space introduced in \cite{ContInc} a methodology originally developed in \cite{Frankowska1990}, which relies on suitable linearisations of set-valued maps that produce admissible inner-perturbed trajectories. From a different standpoint, we also mention \cite{burger} in which a KKT approach is developed in Wasserstein spaces for rather general mean-field optimal control problems with $H^1$-controls. Therein, both the first order optimality conditions and their relationships with finite particle approximations are derived, along with the corresponding rates of convergence. We finally point out that a completely different approach to the mean-field PMP was  formulated for stochastic optimal control problems in \cite{10.1214/14-AOP946} inspired by the theory of mean-field games \cite{lali07} (see also \cite{befrph13,ACFK16}).
Similar methods, based on needle-variations in the space of measures are also leveraged in \cite{eweinan19} and \cite{Jabir2021} for the derivation of the PMP for stochastic control problems of the form \eqref{Ecost}. 

\subsection{Contributions and organization of the paper}

The contributions of this paper can be summarized as follows. From a global standpoint, we start by establishing existence and stability results for \eqref{Pcost}, based on compactness and $\Gamma$-convergence arguments. We then proceed by deriving general first-order optimality conditions for the measure-theoretic formulation of the optimal control of NeurODEs. Our modeling assumptions include the typical forward mappings \eqref{classical} that appear throughout the literature related to neural networks, with for instance $\rho:= \tanh$. As a matter of fact, most of the results available in the literature do not fully encompass this simple model, as they often require global Lipschitz bounds on the transport velocity field.

Let us now describe with more details the fundamental results of the paper. In Section \ref{sec:Existence}, we start by showing that the mean-field optimal control problem \eqref{Pcost} has solution when the regularization parameter $\lambda > 0$ is sufficiently large, and that the latter is in fact unique. By leveraging compactness arguments akin to that classically appearing in the theory of $\Gamma$-convergence, we also establish non-quantitative stability results for the training problem with respect to finite-samples, both at the level of the cost and of the controls. We then proceed by investigating first-order optimality conditions in Section \ref{sec:4}. We initiate the discussion by providing in Section \ref{sec:FormalLag} a heuristic derivation of the  following \textit{mean-field Pontryagin Maximum Principle} (``PMP'' in the sequel) 
\begin{equation}
\label{mfPMP}
\left\{
\begin{aligned}
\partial_t\mu_t & + \nabla_x\cdot (\mathcal F(t,x,\theta_t)\mu_t)=0, \hspace{0.5cm} \mu_t|_{t=0}=\mu_0\, ,\\
\partial_t\psi & + \nabla_x\psi\cdot \mathcal F(t,x,\theta_t)=0, \hspace{0.5cm} \psi_t|_{t=T}=\ell \, , \\
\theta^{\top}_t & = -\frac{1}{2\lambda}\int_{\RR^{2d}}\nabla_x\psi\cdot \nabla_\theta \mathcal F(t,x,\theta_t)\rd\mu_t(x,y)\,, 
\end{aligned} 
\right.
\end{equation}
which characterizes optimal trajectory-control pairs $(\mu,\theta)$ for \eqref{Pcost}. In Section \ref{sec:WellPosed}, we show that the above optimality system is well-posed, and prove in Theorem \ref{thmPMP} that it admits a unique control solution $\theta^* \in \operatorname{Lip}([0,T];\mathbb R^m)$. 
Consequently, we are able to show that the function $\mu_0 \to \theta^*$ which maps initial data distributions to the optimal parameters is single-valued, and to prove that it is also Lipschitz continuous with respect to the Wasserstein distance. Such a precise description of how data are encoded in the parameters of the network is a quite remarkable feature of our results. In particular, it allows us to establish a quantitative \textit{generalization error} for finite samples in Corollary \ref{thmrate}, which writes
\begin{equation}
\label{Pcost2}
\bigg| \int_{\mathbb R^{2d}} \ell(x,y) \ud \mu_T(x,y) - \frac{1}{N}\sum_{i=1}^N \ell(X_T^i,Y_T^i) \bigg| \leq C W_1(\mu_0^N,\mu_0).
\end{equation}
In particular, \eqref{Pcost2} provides a rate of convergence that depends exclusively on the approximability of $\mu_0$ by empirical measures $\mu_0^N$. We should stress at this point the relevance of  \eqref{Pcost2} as it is one of the few results in the literature that rigorously explains the {\it double descent} of both empirical and generalization error in the training of deep neural networks. In Section \ref{sec:numres} we present numerical experiments fully confirming the double descent phenomenon as theoretically predicted by \eqref{Pcost2}, see Figure \ref{fig:statistical}.

\begin{rmk}[Comparison with the existing literature on generalization errors]
We point out that while the generalization errors established in \cite{Jabir2021} are sharper than those of the present paper (in the sense that they express a rate of convergence in $N$ which is dimension-independent), this improved stability comes at the price of considering \emph{relaxed controls} -- i.e. probability measures over $\R^m$ --, that are forced to be non-deterministic by means of entropic regularization terms (see also \cite{cavagnari2020lagrangian}). On the contrary, the generalization errors that we obtain here relate to \emph{deterministic} optimal controls with values in $\R^m$. A similar bound, yielding \eqref{Pcost2}, also appears in a completely different context in  \cite[Theorem 5.1]{burger}, under the constraint that the control is in a ball of $H^1((0,T),\R^m)$, which is a quite restrictive a priori assumption. 
\end{rmk}

After establishing the general form of the optimality system along with some of its interesting properties and applications, we move on to the rigorous derivation of the mean-field PMP in Section \ref{sec:RigorDev}. 
At this stage, let it be noted that while part of our results may be derived by due adaptations from other approaches developed, e.g., in \cite{burger,eweinan19} or \cite{PMPWassConst,PMPWass,SetValuedPMP}, we are able to obtain a few stronger properties on the solutions of the optimal control problem than those generally presented in the literature. Whereas in \cite{PMPWassConst,PMPWass,SetValuedPMP} the first order optimality conditions are established in greater generality -- but also with significant technical effort --, we propose in this paper a new and alternative derivation (very much inspired by the previous work \cite{ACFK16} of the third author), which is significantly simpler and hopefully more accessible to non-specialists. The latter can be heuristically explained as follows: under the technical assumption that the optimal control is continuous in time -- which is motivated by the well-posedness of \eqref{mfPMP} in $\operatorname{Lip}([0,T];\mathbb R^m)$ discussed in Theorem \ref{thmPMP} --, we prove in Theorem \ref{thmcondi} that the mean-field PMP \eqref{mfPMP} can be obtained by means of a generalized Lagrange Multiplier Theorem on the convex subset of Radon measures with unit mass. To this end, we use a new form of calculus recently introduced in \cite{ambrosio2018spatially}, which is simpler than the calculus in Wasserstein spaces used in \cite{burger}. In contrast to this latter work, our approach is applied in a slightly simpler setting, as the forward and backward equations in \eqref{mfPMP} are linear and decoupled, while therein the authors consider models for which they are non-linear and coupled. 
This novel interpretation of the mean-field PMP as result of a Lagrange Multiplier Theorem in spaces of measures is in our view quite powerful, because it can be applied in other mean-field optimal control problems and be more easily understood by a broader community in optimization.

The main theoretical results of the paper can then be summarized as follows.

\begin{thm}[Main contributions of the article]
\label{MAINthm} 
Let $T>0$ be given, consider a map $\MF$ satisfying Assumptions \ref{asum1} and \ref{asum2} of Section \ref{sec:Existence}, fix an initial data distribution $\mu_0\in \mc{P}_c(\RR^{2d})$, and suppose that the regularization parameter $\lambda>0$ is sufficiently large. 

Then, the mean-field optimal control problem \eqref{Pcost} admits solutions, and an admissible control $\theta^* \in L^2([0,T],\mathbb R^m)$ fulfills the mean-field PMP \eqref{mfPMP} \emph{if and only if} it is optimal. In addition, the optimal control $\theta^*$ is uniquely determined, Lipschitz continuous in time, and depends continuously on the initial data distribution $\mu_0$.
\end{thm}

We then close the article by presenting numerical experiments to test the novel mean-field Pontryagin maximum principle that we propose, in which we show the training of simple classification models in $\mathbb R^2$. The reason for working on simple two-dimensional examples is to provide full understanding of the properties of the resulting algorithm and a relatively easy reading and visualization of the results. 

\bigskip

The paper is organized as follows. In Section \ref{sec:2} we introduce  notations and recall a series of preliminary results. In Section \ref{sec:Existence}, we derive a general semiconvexity estimate for the reduced cost functional, and provide sufficient conditions ensuring the existence and stability of its minimizers. In Section \ref{sec:4} we address investigate the mean-field maximum principle by first studying its well-posedness and deriving the generalization error estimate \eqref{Pcost2}, and then showing rigorously how it can be derived either by using a Lagrange multiplier theorem, or via a reduction of the Hamiltonian form. We finally present instructive numerical experiments in Section \ref{sec:5}, where solution of the mean-field maximum principle are computed by means of a shooting method. The Appendix contains proofs of auxiliary results, including the proof of a generalized Lagrange multiplier theorem, Theorem \ref{thmla}, for constrained problems defined over convex subsets of Banach spaces.


\section{Preliminaries and notations}\label{sec:2}
In this section we list some preliminary notations and results from \cite[Section 2.1 and Appendix A.1]{ambrosio2018spatially}, which will be useful throughout the paper.


\subsection{Analysis in measure spaces and optimal transport}

We denote by $\mc{M}(\RR^d)$ the space of signed  Borel measures in $\RR^d$ with finite total variation. 
Note that the space $\mc{M}(\RR^d)$ endowed with the total variation norm
\begin{equation}
\norm{\mu}_{TV}:=\sup\left\{\int_{\RR^d} \varphi \rd\mu~\big|~ \varphi\in \mc{C}_0(\RR^d), ~ \norm{\varphi}_\infty\leq 1\right\}\, ,
\end{equation}
is a Banach space, where $\mc{C}_0(\RR^d)$ represents the set of continuous functions on $\RR^d$ which vanish at infinity. By the Riesz-Markov theorem, it is known that $\mc{M}(\RR^d)=(\MC_0(\RR^d))'$ can be identified with the topological dual of $\MC_0(\RR^d)$ \cite[Theorem 1.54]{AmbrosioFuscoPallara}. 
We further denote $\mc{M}^+(\RR^d)$ the space of positive measures
and by $\mc{P}(\RR^d) \subset \mc{M}^+(\RR^d)$ the subset of probability measures. Furthermore, $\mc{P}_c(\RR^d)\subset \mc{P}(\RR^d)$ represents the set of probability measures with compact support, while $\mc{P}^N_c(\RR^d)\subset\mc{P}_c(\RR^d)$ denotes the subset of empirical or atomic probability measures.  We will also use the following representation formulas for the subset of measures with zero mass 
\begin{equation}
\mc{M}_0(\RR^d) :=\left\{\mu\in (\MC_0(\RR^d))' ~\big|~ \mu(\RR^d)=\int_{\RR^d}1 \rd\mu=0\right\}=: (\MC_0(\RR^d))'_0\,,
\end{equation}
and the subset of measures with unit mass 
\begin{equation}
\mc{M}_1(\RR^d) :=\left\{\mu\in (\MC_0(\RR^d))' ~\big|~ \mu(\RR^d)=\int_{\RR^d}1 \rd\mu=1\right\}=: (\MC_0(\RR^d))'_1\,.
\end{equation}
Moreover, we shall denote by $\mc{M}_{0,c}(\RR^d),\mc{M}_{1,c}(\RR^d)$ the corresponding subsets of measures whose supports are compact. One can also note that given $\mu\in \mc{M}(\RR^d)$, the Jordan decomposition theorem tells us that $\mu=\mu^+-\mu^-$ and $\|\mu\|_{TV}=\mu^+(\RR^d)+\mu^-(\RR^d)$, where $\mu^+,\mu^-\in \mc{M}^+(\RR^d)$.

For the convenience of the reader, we briefly recall the definition of the Wasserstein metrics of optimal transport in the following definition, and refer to \cite[Chapter 7]{AGS} for more details.

\begin{definition}
Let $1\leq p < \infty$ and $\mc{P}_p(\RR^{d})$ be the space of Borel probability measures on $\RR^{d}$ with finite $p$-moment. In the sequel, we endow the latter with the $p$-\emph{Wasserstein metric}
\begin{equation}\label{wassdis}
W_p^{p}(\mu, \nu):=\inf\left\{\int_{\RR^{2d}} |z-\hat{z}|^{p}\ d\pi(z,\hat{z})\ \big| \ \pi \in \Pi(\mu, \nu)\right\}
\end{equation}
where $\Pi(\mu, \nu)$ denotes the set of \emph{transport plan} between $\mu$ and $\nu$, that is the collection of all Borel probability measures on $\RR^d\times \RR^d$ with marginals $\mu$ and $\nu$ in the first and second component respectively. The Wasserstein distance can also be expressed as
\begin{equation}
W_p^{p}(\mu, \nu) = \inf \left\{\mathbb{E} \big[ |Z-\hat{Z}|^{p} \big]\right\}
\end{equation}
where the infimum is taken over all possible joint distributions of random variables $(Z,\hat{Z})$ which laws are given by $\mu$ and $\nu$ respectively.
\end{definition} 

It is a well-known result in optimal transport theory that when $p =1$ and $\mu,\nu \in \Pcal_c(\R^d)$, the following alternative representation holds for the Wasserstein distance
\begin{equation}\label{Kanto}
W_1(\mu,\nu)=\sup\left\{\int_{\RR^d}\varphi(x)\rd (\mu-\nu)(x) ~\big|~ \varphi\in \mbox{Lip}(\RR^d),~ \mbox{Lip}(\varphi)\leq 1\right\}\,,
\end{equation}
by Kantorovich's duality \cite[Chapter 6]{AGS}. Here, $\mbox{Lip}(\RR^d)$ stands for the space of real-valued Lipschitz continuous functions on $\RR^d$, and $\mbox{Lip}(\varphi)$ is the Lipschitz constant of a mapping $\varphi$. In the sequel, we shall also use the signed generalized Wasserstein distance $\mathbb{W}_1^{1,1}$ introduced in \cite{piccoli2019wasserstein}, which coincides with the bounded Lipschitz distance. Given $\mu,\nu\in \mc{M}(\RR^d)$, we set 
\begin{equation}
\mathbb{W}_1^{1,1}(\mu,\nu)=\sup\left\{\int_{\RR^d}\varphi(x)\rd (\mu-\nu)(x) ~\big|~ \varphi\in \mbox{Lip}_b(\RR^d),~ \norm{\varphi}_{\mbox{Lip}_b}\leq 1\right\}\,,
\end{equation}
where
\begin{equation}
\norm{\varphi}_{\mbox{Lip}_b}:=\sup\limits_{x\in\RR^d}|\varphi(x)|+\mbox{Lip}(\varphi)\,.
\end{equation}
In this context, we also define the bounded Lipschitz norm of a signed measure as
\begin{equation}
\|\mu\|_{BL}:=\mathbb{W}_1^{1,1}(\mu,0)\,.
\end{equation}


\subsection{Continuity equations in the space of measures}

In what follows, we recollect some basic facts about continuity equations in the space of measures, following \cite[Section 8.1]{AGS}.
\begin{definition}
\label{defweak}
For any given $T>0$ and $\theta\in L^2([0,T];\RR^m)$, we say that $\mu\in\mc{C}([0,T];\mc{P}_c(\RR^{2d}))$ is a weak solution of \eqref{PDE} on the time interval $[0,T]$ if   
\begin{equation}
\label{eqweak}
\int_0^T\int_{\RR^{2d}} \Big( \partial_t \psi(t,x,y) + \nabla_x \psi(t,x,y) \cdot \mathcal{F}(t,x,\theta_t) \Big)\rd\mu_t(x,y)\dt = 0,
\end{equation}
for every $\psi\in \MC_c^1((0,T)\times \RR^{2d})$.
\end{definition}

\begin{rmk}
First, note that \eqref{eqweak} is equivalent to
\begin{equation}
\label{eqweak1}
\int_{\RR^{2d}}\psi(x,y)\rd \mu_{t_2}(x,y) - \int_{\RR^{2d}}\psi(x,y)\rd \mu_{t_1}(x,y) = \int_{t_1}^{t_2} \int_{\RR^{2d}} \nabla_x \psi(x,y) \cdot \mathcal F(s,x,\theta_s) \rd\mu_s(x,y)\ds
\end{equation}
for all $\psi\in \MC_b^1(\RR^{2d})$ and every $t_1,t_2 \in [0,T]$. This follows from the fact that the linear span of functions of the form $\psi(t,x,y) := \eta(t)\xi(x,y)$ with $\eta\in \MC_c^1((0,T))$ and $\xi\in \MC_c^1(\RR^{2d})$ is dense in $\MC_c^1((0,T)\times \RR^{2d})$ (see e.g.  \cite[Remark 8.1.1]{AGS}). Also, observe that since $\mu$ is a curve of compactly supported probability measures, we can use the simpler testing space $\MC_b^1(\RR^{2d})$ instead of $\mc{C}_c^1(\RR^{2d})$  or $\mc{C}_0^1(\RR^{2d})$ in \eqref{eqweak1}.
\end{rmk}

Classical well-posedness result for \eqref{PDE} for arbitrary initial measures is usually established under the following type of standard Cauchy-Lipschitz assumptions (or minimal variations thereof).  

\begin{assum}
\label{asum1}
For any given $T>0$, the vector field $\MF$ satisfies the following.
\begin{enumerate}
\item[$(i)$] For any fixed $\theta \in \R^m$, the map $(t,x) \mapsto \MF(t,x,\theta) \in \R^d$ is continuous.
\item[$(ii)$] There exists a constant  $C_{\MF}>0$ that may depend on $d,m$ such that for every $\theta \in \R^m$, it holds
\begin{equation*}
|\mathcal F(t,x,\theta)|\leq C_{\MF}(1+|x|),\quad \mbox{ for a.e. } t\in [0,T] \mbox{ and every }x\in\RR^{d} \,.
\end{equation*}
\item[$(iii)$] There exists a constant  $L_{\MF} >0$  independent of $d,m$ such that for every $\theta \in \R^m$, it holds
\begin{equation*}
|\mathcal F(t,x_1,\theta)-\mathcal F(t,x_2,\theta)|\leq L_{\MF}(1+|\theta|) |x_1-x_2| ,\quad \mbox{ for a.e. } t\in [0,T] \mbox{ and every }x_1, x_2 \in \RR^d  \, ,
\end{equation*}
and we denote $L_{\mc{F},T,\|\theta\|_1} := L_{\MF} \INTSeg{(1+|\theta_t|)}{t}{0}{T}$
\item[$(iv)$] For all $(t,x) \in [0,T] \times \R^d$, the map $\theta \mapsto \MF(t,x,\theta)$ is twice differentiable. Moreover for each $R > 0$, there exists a constant $C(d,m,R)>0$  such that
\begin{equation*}
\|\nabla_\theta \MF\|_{\mc{C}([0,T]\times B(R) \times \R^m;\R^{d\times m})} \, + \, \| \nabla_{\theta}^2 \mathcal F \|_{\mc{C }([0,T] \times B(R)\times \RR^m;\R^{d\times m\times m})} \leq  C(d,m,R) \,.
\end{equation*}
\end{enumerate}	
\end{assum}

Under the set of assumptions listed above, we can prove the well-posedness of \eqref{PDE} as stated in the following theorem. The proof of the latter is standard and deferred to Appendix \ref{subsection:Flows}.

\begin{thm}[Classical well-posedness for continuity equation]
\label{thm:Wellposed}
Consider a measure $\mu_0\in \mc{P}_c(\RR^{2d})$ with $\supp(\mu_0) \subset B(R)$ for some $R>0$, and suppose that $\MF$ satisfies Assumption \ref{asum1}. 

Then for any given $T>0$ and $\theta\in L^2([0,T];\RR^{m})$, there exists a unique solution $\mu\in\mc{C}([0,T];\mc{P}_c(\RR^{2d}))$ to \eqref{PDE} in the sense of Definition \ref{defweak}. Moreover, there exists a radius $R_T > 0$ depending only on $R$ and $C_{\MF}$ such that
\begin{equation}
\label{suppt}
\supp(\mu_t) \subset B(R_T),
\end{equation}
for all times $t \in [0,T]$, and additionally it holds for any $s,t\in[0,T]$ that
\begin{equation}
\label{eq:LipEst}
W_1(\mu_t,\mu_s)\leq C(R,T,C_{\MF})|t-s|\,.
\end{equation}
Denoting by $\mu^i$ for $i=1,2$ two solutions of \eqref{PDE} with initial data $\mu_0^i$ satisfying the above assumptions, the following stability estimate
\begin{equation}\label{stablity}
W_1(\mu_t^1,\mu_t^2)\leq e^{L_{\MF,T,\| \theta \|_1}}W_1(\mu_0^1,\mu_0^2), 
\end{equation}
holds for all times $t\in[0,T]$, where $C_{\MF}$ and $L_{\MF,T,\| \theta \|_1}$ are defined as in Assumption \ref{asum1}.
\end{thm}


\subsection{Differential calculus over convex subsets of Banach spaces}
\label{subsection:Convex}

We end this series of preliminaries by introducing a notion of multi-valued Fr\'{e}chet differential for functions defined on convex sets. To this end, given a convex subset $E$ of a normed vector space $X$, we define
\begin{equation*}
X_E := \RR(E-E) = \Big\{ x \in X ~\big|~ x = \alpha (e_1 - e_2) ~\text{with $\alpha \in \R$ and $e_1,e_2 \in E$} \Big\},
\end{equation*}
and given $e\in E$, we denote by $X_e := \RR_+(E-e)$ the convex cone of directions at $e$.
\begin{definition}
\label{def:Fdiff}
Let $X$, $Y$ be normed vector spaces, $E\subset X$ be a convex set and $f:E\to Y$. Then, $f$ is \emph{$F$-differentiable} at $e\in E$ if there exists $L\in\mc{L}(X_E,Y)$ such that
\begin{equation}\label{A1}
\lim\limits_{\substack{e'\to e \\ e' \in E}}\frac{\norm{f(e')-f(e)-L(e'-e)}_Y}{\norm{e'-e}_X}=0\,,
\end{equation}
where $\mc{L}(X_E,Y)$ denotes the space of bounded linear operators from $X_E$ into $Y$.
\end{definition}

Following the previous definition, we define the \emph{$F$-differential} of $f$ at $e\in E$ by
\begin{equation}
Df(e):=\Big\{ L\in\mc{L}(X_E,Y) ~\big|~  L \mbox{ satisfies $\eqref{A1}$} \Big\}\,.
\end{equation}
It can be checked that if $X_e$ is not dense in $X_E$, then the mapping $D$ is set-valued (similarly to classical convex subdifferentials). However if $v\in \overline X_e$, then the evaluation $Df(e)(v)$ is uniquely determined, namely it does not depend on the choice of $L$ in $Df(e)$, and in this case we will slightly abuse the notation and write $Df(e)(v)$ to mean $L(v)$ for any $L\in Df(e)$. By a density argument, each $L\in Df(e)$ can be uniquely extended to an operator $\overline L$ in $\mc{L}(\overline {X}_E,Y)$. We will then say that $f\in \MC^1(E;Y)$ if $f$ is $F$-differentiable at each $e\in E$, and there exists a selection $e \in E \mapsto L_e\in Df(e)$ such that
\begin{equation}\label{A2}
e\mapsto L_e \quad \mbox{ is continuous from } E \mbox{ into }\mc{L}(X_E,Y)\,,
\end{equation}
where $\mc{L}(X_E,Y)$ is endowed with the distance induced by the standard operator norm. 
 
\begin{definition}
Let $X$, $Y$ be normed vector spaces, $E\subset X$ be a convex set, and $f: E\to Y$. Then, $f$ is \emph{$G$-differentiable} at $e\in E$ if the directional right derivatives
\begin{equation}\label{Gdiff}
df(e,v):=\lim\limits_{h\to 0^+}\frac{f(e+hv)-f(e)}{h}, 
\end{equation}
exist in $Y$ for all $v\in X_e$.
\end{definition}

\begin{rmk}
Obviously if $f$ is $F$-differentiable at some $e \in E$, then it is $G$-differentiable as well with $df(e,v)=Df(e)(v)$ for all $v \in X_e$.
\end{rmk}

We shall also use the following lemma as a criterion for $\MC^1$ regularity, see \cite[Lemma A.4]{ambrosio2018spatially}.

\begin{lem}\label{lmC1}
Let $f:~E\to F$ be a continuous map and suppose that there exists a continuous application 
\begin{equation}
e \in E \mapsto L_e\in  \mc{L}(X_E,Y), 
\end{equation}
such that $df(e,v)=L_ev$ for all $e\in E$ and any $v\in X_e$. Then $f\in \MC^1(E;Y)$ and $e \mapsto L_e\in Df(e)$ is an admissible selection.
\end{lem}


\section{Existence of minimizers and stability of solutions}
\label{sec:Existence}

In this section, we investigate sufficient conditions ensuring the existence of optimal solutions to the mean-field optimal control problem \eqref{Pcost}, as well as stability properties for the minimizers and costs stemming from large finite-sample training. Throughout the remainder of this article, we will use Assumption \ref{asum1} and the  following additional hypotheses to establish most of our results.

\begin{assum}\label{asum2} For any given $T > 0$ and $R > 0$, the vector field $\MF$ satisfies the following. 
\begin{enumerate}
\item[$(i)$]  The map $x \in \R^d \mapsto \MF(t,x,\theta)$ is of class $\mc{C}^2$ all times $t\in[0,T]$ and any $\theta \in \RR^m$, and for each $x\in B(R)$, it holds 
\begin{equation}
|\nabla_x\cdot \nabla_\theta \MF(t,x,\theta)|+	|\nabla_x\MF(t,x,\theta)|+	|\nabla_x^2\MF(t,x,\theta)|\leq C(d,m,R,|\theta|)\,;
\end{equation}
\item[$(ii)$]  For any $\theta^1,\theta^2 \in \RR^m$, every $s,t \in[0,T]$ and all $x\in B(R)$, it holds
\begin{equation}\label{time}
|\mathcal F(t,x,\theta^1) - \mathcal F(s,x,\theta^2)|\leq C(d,m,R) \big(|t-s| + |\theta^1-\theta^2| \big)\,;
\end{equation}
\item[$(iii)$] For all fixed $\theta$ and $t \in[0,T]$, it holds
\begin{equation}
|\nabla_\theta \mathcal F(t,x,\theta) - \nabla_\theta \mathcal F(t,y,\theta)|\leq C(d,m,R,|\theta|)|x-y|, 
\end{equation}
for every $x,y\in B(R)$.
\end{enumerate}
\end{assum}

Before moving on to the discussion pertaining to the existence and stability properties for solutions of \eqref{Pcost}, we highlight the adequacy of our working hypotheses in connection with classical machine learning models.

\begin{rmk}[Adequacy of smooth sigmoidal activations]
\label{rmtanh}
Assumptions \ref{asum1} and \ref{asum2} require smooth activation functions that exhibit also some boundedness properties with respect to the parameter $\theta$, e.g. as in Assumption \ref{asum1}-$(ii)$. These latter are needed both to express the PMP and to establish its well-posedness, as will become apparent in Section \ref{sec:4}. Hence, some popular network models which use for instance ReLu activations are not covered by our results. However, we check here that the sets of hypotheses listed in Assumptions \ref{asum1} and \ref{asum2} include the popular subclass of feed-forwarding  dynamics \eqref{classical} involving sigmoidal-type activation functions, such as
\begin{equation*}
\MF(t,x,\theta) = \MF(x,\theta) := \tanh(\theta x)\in\RR^d\,,
\end{equation*}
where $\theta\in \RR^{m}=\RR^{d\times d}$ and $x\in\RR^{d}$. In that case, Assumption \ref{asum1}-$(i)$ obviously holds, and since
\begin{equation*}
\MF_k(x,\theta) =\tanh \Big(\mathsmaller{\sum}\limits_{l=1}^d \theta_{k,l}  x_l \Big)
\end{equation*}
for each $k \in \{1,\dots,d\}$ and $|\tanh(r)|\leq 1$ for all $r \in \R$, we have that $|\MF(x,\theta)|\leq \sqrt{d}$ for all $(x,\theta) \in \R^d \times \R^{d \times d}$, and Assumption \ref{asum1}-$(ii)$ also holds. This uniform boundedness property of the driving field implies in particular that the radius $R_T > 0$ given by Theorem \ref{thm:Wellposed} and controlling the support sizes of the solutions of \eqref{PDE} will scale polynomially and not exponentially on $d \geq 1$, along with all the relevant constants depending polynomially thereon. Moreover, observe that
\begin{equation*}
\partial_{x_i} \MF_k(x,\theta) =\tanh' \Big( \mathsmaller{\sum}\limits_{l=1}^d \theta_{k, l} \, x_l \Big) \theta_{k,i}
\end{equation*}
for each $i,k \in \{1,\dots,d\}$, which implies in particular that $|\nabla_x \mathcal F(x,\theta)| \leq |\theta|$ for all $(x,\theta) \in \R^d \times \R^{d \times d}$ by using the fact that $|\tanh'(r)|=|1-\tanh(r)^2|\leq 1$ for each $r \in \R$. By the mean-value theorem, this latter fact directly implies that
\begin{equation*}
|\MF(t,x_1,\theta)-\MF(t,x_2,\theta)|\leq |\theta||x_1-x_2|
\end{equation*}
for all $\theta \in \R^{d \times d}$ and $x_1,x_2 \in \R^d$, which verifies Assumption \ref{asum1}-$(iii)$. Concerning Assumption \ref{asum1}-$(iv)$, one has that 
\begin{equation*}
\partial_{\theta_{ij}} \MF_k(x,\theta) = \delta_{k,i} \tanh' \Big( \mathsmaller{\sum}\limits_{l=1}^d \theta_{k,l} x_l \Big) x_j 
\end{equation*}
for each $i,j,k \in \{,1\dots,d\}$ -- where $\delta_{k,i}$ refers here to the Kronecker symbol --, which implies that $|\nabla_\theta \MF(x,\theta)|\leq \sqrt{d}|x|$ for all $\theta\in \RR^m$ and $x\in\RR^d$.  Furthermore, one can easily see that 
\begin{equation*}
\partial^2_{\theta_{i,j},\theta_{m,n}} \MF_k(x,\theta) = \delta_{k,m}\delta_{k,i}\tanh'' \Big( \mathsmaller{\sum}\limits_{l=1}^d \theta_{k,l} x_l \Big)x_j x_n
\end{equation*}
for each $i,j,k,m,n \in \{1,\dots,d\}$, which then yields $|\nabla_\theta^2 \MF(x,\theta)|\leq 4\sqrt{d} \,|x|^2$ for all $(x,\theta) \in \R^d \times \R^{d \times d}$ since $|\tanh''(r)|=|2\tanh(r)\left(\tanh(r)-1\right)|\leq 4$ for every  $r \in \R^d$. Thence, it holds 
\begin{equation}
\label{3.5}
\max\limits_{(x,\theta) \in B(R) \times \RR^{m}}|\nabla_\theta \MF(x,\theta)|\leq \sqrt{d}R \qquad \text{and} \qquad \max\limits_{(x,\theta) \in B(R) \times \RR^{m}}|\nabla^2_\theta \MF(x,\theta)| \leq 4\sqrt{d}R^2,
\end{equation}
which completes the verification of Assumption \ref{asum1}.

We now shift our attention to the verification of Assumption \ref{asum2}. First of all, one has that 
\begin{equation*}
\partial^2_{x_i,x_j} \MF_k(x,\theta) = \tanh'' \Big( \mathsmaller{\sum}_{l=1}^d  \theta_{k,l} x_l \Big) \theta_{k,i} \theta_{k,j}
\end{equation*}
for each $i,j,k \in \{1,\dots,d\}$, which yields the  estimate $|\nabla_x^2 \mathcal F(x,\theta)|\leq 4|\theta|^2$ for all $(x,\theta) \in \R^d \times \R^{d \times d}$. Moreover, one can check that 
\begin{equation*}
\partial_{x_n} \partial_{\theta_{i,j}} \MF_k(x,\theta) = \delta_{k,i} \tanh'' \Big( \mathsmaller{\sum}\limits_{l=1}^d \theta_{k,l} x_l \Big) x_j\theta_{k,n} + \delta_{k,i}\delta_{j,n} \tanh' \Big(\mathsmaller{\sum}\limits_{l=1}^d \theta_{k,l} x_l \Big)
\end{equation*} 
for each $i,j,k,n \in \{1,\dots,d\}$. Thus, we obtain the estimates 
\begin{equation*}
|\nabla_x\cdot \nabla_\theta \MF(x,\theta) |\leq \sqrt{d}|\nabla_x \nabla_\theta \MF(x,\theta)| \leq \sqrt{2} \sqrt{d}\big ( 4 |x||\theta| + d\big) 
\end{equation*}
for all $(x,\theta) \in \R^d \times \R^{d \times d}$, which leads to Assumption \ref{asum2}-$(i)$ being fulfilled. Moreover, we can also deduce from the previous estimate that Assumption \ref{asum2}-$(iii)$ holds, since 
\begin{equation*}
|\nabla_\theta \mathcal F(t,x,\theta)-\nabla_\theta \mathcal F(t,y,\theta)|\leq  \sqrt{2} (4 R|\theta| + d)|x-y|
\end{equation*}
for all $\theta \in \R^{d \times d}$ and $x,y \in B(R)$. Lastly, it follows from \eqref{3.5} that  
\begin{equation*}
|\mathcal F(t,x,\theta^1) - \mathcal F(s,x,\theta^2)| \leq  \sqrt{d}R|\theta^1-\theta^2|
\end{equation*}
for all $\theta^1,\theta^2 \in \R^{d \times d}$ and $x \in \R^d$, which equivalently means that Assumption \ref{asum2}-$(ii)$ is satisfied and completes the verification of Assumption \ref{asum2}.
\end{rmk}


\subsection{Convexity of the reduced cost functional and existence of minimizers}

As already recalled in the introduction, $L^2$-regularization of network parameters is a standard practice in machine learning which helps stabilizing the training procedure, while promoting the generalization capacities of networks \cite{Goodfellow-et-al-2016,Kukacka-et-al-2017}. In this section, we show that for regularization parameters $\lambda > 0$ that are sufficiently large, the reduced cost of the problem is actually strictly convex, which in particular implies the existence and uniqueness of an optimal control $\theta^* \in L^2([0,T];\R^m)$ for the mean-field optimal control problem \eqref{Pcost}. Given the smoothness of the forward map $\mathcal F$ the convexity of $J$ is perhaps not surprising, but it has never been noticed before in the literature in connection to mean-field optimal control problems, and appears to have far-reaching practical implications that we shall explore in the remainder of the paper.

For any fixed $\theta\in L^2([0,T],\RR^m)$, we denote by $(\Phi^{\theta}_{(\tau,t)}(\cdot))_{\tau,t \in [0,T]}$ the \textit{characteristic flow} generated by the controlled velocity field $(t,x) \in [0,T] \times \R^d \mapsto \MF(t,x,\theta_t) \in \R^d$, defined by 
\begin{equation}
\label{eq:ThetaFlow}
\left\{
\begin{aligned}
\partial_t \Phi^{\theta}_{(\tau,t)}(x) & = \MF \big(t,\Phi^{\theta}_{(\tau,t)}(x),\theta_t\big), \\
\Phi^{\theta}_{(\tau,\tau)}(x) & = x, 
\end{aligned}
\right.
\end{equation}
for every $x \in \R^d$. It is a well-known result in the theory of non-linear dynamical systems (see e.g. \cite[Theorem 2.3.2]{BressanPiccoli}) that under Assumption \ref{asum1}, the flow maps $\Phi^{\theta}_{(\tau,t)} : \R^d \rightarrow \R^d$ are continuously differentiable for every  $\tau,t \in [0,T]$, and the application $t \in [0,T] \mapsto \nabla_x \Phi^{\theta}_{(\tau,t)}(x) \in \R^{d \times d}$ is the unique solution of the forward linearized Cauchy problem
\begin{equation}
\label{eq:linearizedFlow}
\left\{
\begin{aligned}
\partial_t w(t,x) & = \nabla_x \MF\big( t , \Phi^{\theta}_{(\tau,t)}(x), \theta_t \big)w(t,x) \\
w(\tau,x) & =\mbox{Id}.
\end{aligned}
\right.
\end{equation}
This allows us to establish the following semiconvexity result for the reduced cost of \eqref{Pcost}.

\begin{proposition}[Semiconvexity of the reduced cost functional]
\label{prop:Semiconvexity}
Let $T,R >0$ and $\mu_0 \in \Pcal_c(\R^d)$ be such that $\supp(\mu_0) \subset B(R)$, and suppose that Assumptions \ref{asum1} and \ref{asum2} hold. Then, for every ball $\Gamma \subset L^2([0,T];\R^m)$, there exists a constant $\Lcal(T,R,\Gamma) > 0$ such that the reduced cost functional
\begin{equation}
\label{eq:ReducedCost}
J : \theta \in L^2([0,T];\R^m) \mapsto
\left\{
\begin{aligned}
& \INTDom{\ell(x,y)}{\R^{2d}}{\mu_T^{\theta}(x,y)} + \lambda \INTSeg{|\theta_t|^2}{t}{0}{T},  \\
& \;\, \textnormal{s.t.} \,\,
\left\{ 
\begin{aligned}
& \partial_t \mu_t^{\theta} + \nabla_x \big( \Fcal(t,x,\theta_t) \mu_t^{\theta} \big) = 0, \\
& \mu_0^{\theta} = \mu_0, 
\end{aligned}
\right.
\end{aligned}
\right.
\end{equation}
satisfies the semiconvexity estimate 
\begin{equation}
\label{eq:Semiconvexity}
J \big((1-\zeta)\theta^1 + \zeta \theta^2 \big) \leq (1-\zeta) J(\theta^1) + \zeta J(\theta^2) - (2\lambda - \Lcal(T,R,\Gamma)) \tfrac{\zeta(1-\zeta)}{2} \| \theta^1 - \theta^2\|_2^2 
\end{equation}
for any $\theta^1,\theta^2 \in \Gamma$ and all $\zeta \in [0,1]$. In particular if $\lambda > \tfrac{1}{2} \Lcal(T,R,\Gamma)$, the reduced cost functional is then strictly convex over $\Gamma$.
\end{proposition}

The proof of this convexity estimate is almost entirely contained in the following regularity result, which itself relies on a series of technical properties for characteristic flows which are exposed in Appendix \ref{subsection:ControlFlow}.

\begin{lem}[Regularity of the reduced final cost]
\label{lem:FinalCostReg}
Let $T,R > 0$ and $\mu_0 \in \Pcal_c(\R^{2d})$ be such that $\supp(\mu_0) \subset B(R)$, and suppose that Assumptions \ref{asum1} and \ref{asum2} hold. Then, the reduced final cost 
\begin{equation}
\label{eq:FinalReduced}
J_{\ell} : \theta \in L^2([0,T];\R^m) \mapsto  
\left\{
\begin{aligned}
& \INTDom{\ell(x,y)}{\R^{2d}}{\mu_T^{\theta}(x,y)},  \\
& \; \, \textnormal{s.t.} \,\,
\left\{ 
\begin{aligned}
& \partial_t \mu_t^{\theta} + \nabla_x \big( \Fcal(t,x,\theta_t) \mu_t^{\theta} \big) = 0, \\
& \mu_0^{\theta} = \mu_0, 
\end{aligned}
\right.
\end{aligned}
\right.
\end{equation}
is Fr\'echet-differentiable. Moreover, denoting its gradient by $\nabla_{\theta} J_{\ell}(\theta) \in L^2([0,T];\R^m)$ and choosing $\theta^1,\theta^2 \in L^2([0,T];\R^m)$, there exists a constant $\Lcal(T,R,\|\theta^1\|_1, \| \theta^2\|_1) > 0$ such that 
\begin{equation*}
\big\| \nabla_{\theta} J_{\ell}(\theta^1) - \nabla_{\theta} J_{\ell}(\theta^2) \big\|_2 \leq \Lcal(T,R,\|\theta^1\|_1, \| \theta^2\|_1) \, \big\| \theta^1 - \theta^2 \big\|_2.
\end{equation*}
\end{lem}

\begin{proof}
We start by fixing a control signal $\theta \in L^2([0,T];\R^m)$. Following the discussion in Appendix \ref{subsection:Flows} below, the unique solution $\mu^{\theta} \in \Ccal([0,T];\Pcal_c(\R^{2d}))$ of the controlled continuity equation can be expressed as $\mu_t^{\theta} = \BPhi^{\theta}_{(0,t)}\sharp \mu_0$, where
\begin{equation*}
\BPhi_{(0,t)}^{\theta}(x,y) = \big( \Phi^{\theta}_{(0,t)}(x),y \big)
\end{equation*}
for all $(x,y) \in \R^{2d}$, with $(\Phi^{\theta}_{(0,t)}(\cdot))_{t \in [0,T]}$ being the characteristic flow defined in \eqref{eq:ThetaFlow}. In particular, this allows us to rewrite the reduced final cost as 
\begin{equation*}
J_{\ell}(\theta) = \INTDom{\ell \Big( \Phi^{\theta}_{(0,T)}(x),y \Big)}{\R^{2d}}{\mu_0(x,y)}. 
\end{equation*}
Given another control signal $\vartheta \in L^2([0,T];\R^m)$ and some $\varepsilon > 0$, we know by Proposition \ref{prop:DiffFlow} that the following Taylor expansion 
\begin{equation}
\label{eq:FlowTaylorCost1}
\Phi^{\theta+\epsilon \vartheta}_{(0,T)}(x) = \Phi^{\theta}_{(0,T)}(x) + \varepsilon \INTSeg{\Rcal^{\theta}_{(t,T)}(x) \nabla_{\theta} \Fcal \big( t , \Phi^{\theta}_{(0,t)}(x) , \theta_t \big) \vartheta_t}{t}{0}{T} + o_{\theta}(\varepsilon)
\end{equation}
holds for all $(t,x) \in [0,T] \times B(R)$, where $(\Rcal^{\theta}_{(\tau,t)}(\cdot))_{t \in [0,T]} \subset \Ccal^1(\R^d;\R^{d \times d})$ are the resolvent maps of the linearized Cauchy problem defined as in \eqref{eq:Resolvent}. Since the small-o in \eqref{eq:FlowTaylorCost1} is uniform in $x \in B(R)$, it holds by  Lebesgue's dominated convergence and Fubini's theorems that 
\begin{equation}
\label{eq:FlowTaylorCost2}
\begin{aligned}
& \INTDom{\ell \Big( \Phi^{\theta+\varepsilon \vartheta}_{(0,T)}(x),y \Big)}{\R^{2d}}{\mu_0(x,y)} \\
& = \INTDom{\ell \Big( \Phi^{\theta}_{(0,T)}(x),y \Big)}{\R^{2d}}{\mu_0(x,y)} \\
& \hspace{0.45cm} + \varepsilon \INTSeg{\bigg\langle \INTDom{\Big( \Rcal^{\theta}_{(t,T)}(x) \nabla_{\theta} \Fcal \big( t , \Phi^{\theta}_{(0,t)}(x),\theta_t \big) \Big)^{\hspace{-0.1cm} \top} \nabla_x \ell \big(\Phi^{\theta}_{(0,T)}(x), y \big)}{\R^{2d}}{\mu_0(x,y)} , \vartheta_t \bigg\rangle}{t}{0}{T} + o_{\theta}(\varepsilon),
\end{aligned}
\end{equation}
for every $\varepsilon > 0$ small enough. From the regularity estimates of Assumption \ref{asum1}, Proposition \ref{prop:BoundFlow} and Proposition \ref{prop:DiffFlow}, we may infer that the Gateaux derivative expressed in \eqref{eq:FlowTaylorCost2} is continuous with respect to $\theta \in L^2([0,T];\R^m)$, so that the reduced final cost is Fr\'echet-differentiable, with
\begin{equation}
\label{eq:ReducedCostGrad}
\nabla_{\theta} J_{\ell}(\theta) : t \in [0,T] \mapsto \INTDom{\Big( \Rcal^{\theta}_{(t,T)}(x) \nabla_{\theta} \Fcal \big( t , \Phi^{\theta}_{(0,t)}(x),\theta_t \big) \Big)^{\hspace{-0.1cm} \top} \nabla_x \ell \big(\Phi^{\theta}_{(0,T)}(x), y \big)}{\R^{2d}}{\mu_0(x,y)}.
\end{equation}
At this stage, by resorting again to Assumptions \ref{asum1} and \ref{asum2}, Proposition \ref{prop:BoundFlow} and Proposition \ref{prop:DiffFlow}, one can check that the previous expression is a (formal) product of quantities which are bounded and Lipschitz with respect to $\theta$ on bounded subsets of $L^1([0,T];\R^m)$. Whence, for every pair $\theta^1,\theta^2 \in L^2([0,T];\R^m)$, there exists a constant $\Lcal(T,R,\|\theta^1\|_1,\|\theta^2\|_1)$ such that 
\begin{equation*}
\big\| \nabla_{\theta} J_{\ell}(\theta^1) - \nabla_{\theta} J_{\ell}(\theta^2) \big\|_2 \leq \Lcal(T,R,\|\theta^1\|_1, \| \theta^2\|_1) \, \big\| \theta^1 - \theta^2 \big\|_2,
\end{equation*}
which ends the proof of our claim. 
\end{proof}

We are now ready to move on to the proof of Proposition \ref{prop:Semiconvexity}. 

\begin{proof}[Proof of Proposition \ref{prop:Semiconvexity}]
First, observe that the reduced cost of the problem can be written as 
\begin{equation*}
J(\theta) = J_{\ell}(\theta) + \lambda \| \theta \|_2^2
\end{equation*}
for all $\theta \in L^2([0,T];\R^m)$, where $J_{\ell}(\mu_0,\theta)$ stands for the reduced final cost defined in \eqref{eq:FinalReduced}. Whence, it can be easily checked as a consequence of Lemma \ref{lem:FinalCostReg} that the reduced cost is Fr\'echet-differentiable, with 
\begin{equation}
\nabla_{\theta} J(\theta) = \nabla_{\theta} J_{\ell}(\theta) + 2\lambda \theta.
\end{equation}
Let $\Gamma \subset L^2([0,T];\R^m)$ be a closed ball and $\theta^1,\theta^2 \in \Gamma$. By performing routine computations based on the integral version of Taylor's theorem (see e.g. \cite[Lemma 6]{LipReg} for a detailed proof in the finite-dimensional case), one can show that 
\begin{equation*}
\begin{aligned}
J_{\ell} \big( (1-\zeta)\theta^1 + \zeta \theta^2) \big) & \leq (1-\zeta) J_{\ell}(\theta^1) + \zeta J_{\ell}(\theta^2) + \Lip( \nabla_{\theta} J_{\ell} \, ; \Gamma) \tfrac{\zeta(1-\zeta)}{2} \big\| \theta^1-\theta^2 \big\|_2^2 \\
& \leq (1-\zeta) J_{\ell}(\theta^1) + \zeta J_{\ell}(\theta^2) + \Lcal(T,R,\Gamma) \tfrac{\zeta(1-\zeta)}{2} \big\| \theta^1-\theta^2 \big\|_2^2, 
\end{aligned}
\end{equation*}
for all $\zeta \in [0,1]$, where the constant $\Lcal(T,R,\Gamma) := \Lcal(T,R,\|\theta^1\|_1 , \|\theta^2\|_1)$ is given as in Lemma \ref{lem:FinalCostReg}. This, together with the standard fact of convex analysis in Hilbert spaces stating that 
\begin{equation*}
\big\| (1-\zeta) \theta^1 + \zeta \theta^2 \big\|_2^2 \leq (1-\zeta)\|\theta^1\|_2^2 + \zeta\|\theta^2\|_2^2 - \tfrac{\zeta(1-\zeta)}{2} \big\| \theta^1 - \theta^2 \big\|_2^2
\end{equation*}
allows us to conclude that the reduced cost functional satisfies the semiconvexity estimate \eqref{eq:Semiconvexity} over $\Gamma$. 
\end{proof}

By leveraging the semiconvexity result of Proposition \ref{prop:Semiconvexity}, we are able to derive sufficient conditions for the existence of mean-field optimal controls.

\begin{thm}[Existence of minimizers]
\label{thm:Exist}
Let $T,R > 0$, $\mu_0 \in \Pcal_c(\R^{2d})$ be such that $\supp(\mu_0) \subset B(R)$, and $\Gamma \subset L^2([0,T];\R^m)$ be the closed ball of radius $C_{\Gamma}^{1/2} := \| \ell \|_{\Ccal(B(R))}+1$. If the regularization parameter is such that $\lambda > \tfrac{1}{2}\Lcal(T,R,\Gamma)$ where the latter constant is given as in Lemma \ref{lem:FinalCostReg}, then there exists a unique optimal control $\theta^* \in \Gamma$ for \eqref{Pcost}.
\end{thm}

\begin{proof}
The result follows from a standard application of the direct method of the calculus of variations. Given a minimizing sequence $(\theta^n) \subset L^2([0,T];\R^m)$ for which
\begin{equation}
\label{eq:minimizing}
J(\theta^n) ~\underset{n \to +\infty}{\longrightarrow}~ \inf_{\theta \in L^2([0,T];\R^m)} J(\theta), 
\end{equation}
it necessarily holds for $n \geq 1$ sufficiently large that 
\begin{equation*}
J(\theta^n) \leq J(0) + 1 \leq \| \ell \|_{\Ccal(B(R))} + 1.
\end{equation*}
Recalling the expression \eqref{eq:ReducedCost} of the reduced cost, this implies in particular that $\| \theta^n \|_2 \leq C_{\Gamma}^{1/2}$ for each $n \geq 1$, or equivalently $(\theta^n) \subset \Gamma$. Remark now that $\Gamma \subset L^2([0,T];\R^m)$ is weakly compact since it is a closed ball in a Hilbert space (see e.g. \cite[Theorem 3.17]{Brezis}), so that there exists an element $\theta^* \in \Gamma$ for which
\begin{equation*}
\qquad \theta^{n_k} ~\underset{k \to +\infty}{\rightharpoonup}~ \theta^* \qquad \text{in $L^2([0,T];\R^m)$}, 
\end{equation*}
along an adequate subsequence. Moreover, it easily follows from Lemma \ref{lem:FinalCostReg} that $\theta \mapsto J(\theta) \in \R$ is continuous in the strong $L^2$-topology, as well as convex since we assumed that $\lambda > \tfrac{1}{2} \Lcal(T,R,\Gamma)$. As such, it is weakly lower-semicontinuous (see e.g. \cite[Corollary 3.9]{Brezis}), which together with \eqref{eq:minimizing} implies that 
\begin{equation*}
J(\theta^*) \, \leq \, \liminf_{n \to +\infty} J(\theta^n) \, = \inf_{\theta \in L^2([0,T];\R^m)} J(\theta).
\end{equation*}
Hence, we have shown that $\theta^* \in \Gamma$ is a solution of the mean-field optimal control problem \eqref{Pcost}, and its uniqueness follows straightforwardly from the strict convexity of the reduced cost. 
\end{proof}


\subsection{Stability of finitely-sampled costs and controls}

In this section, we establish a general stability property for solutions of the mean-field optimal control problem \eqref{Pcost} with respect to finite-samples. More precisely, assume that we are given a sample $\{(X_0^i,Y_0^i)\}_{i=1}^N$ of size $N \geq 1$ independently and identically distributed according to $\mu_0 \in \Pcal_c(\R^{2d})$, and let us consider the empirical loss minimization problem
\begin{equation}
\label{PNCost}
\inf_{\theta \in L^2([0,T];\R^m)} J^N(\theta) := 
\left\{
\begin{aligned}
& \inf_{\theta \in L^2([0,T];\R^m)} \frac{1}{N} \sum_{i=1}^N \ell(X_T^i,Y_T^i) + \lambda \INTSeg{|\theta_t|^2}{t}{0}{T} \\
& \,\; \textnormal{s.t.} \,\,
\left\{
\begin{aligned}
& \dot X_t^i = \Fcal(t,X_t^i,\theta_t), \hspace{2.45cm} \dot Y_t^i =0, \\
& (X_t^i,Y_t^i)_{|t=0} = (X_0^i,Y_0^i), ~~ i \in \{1,\dots,N\}.
\end{aligned}
\right.
\end{aligned}
\right.
\end{equation}
By introducing the empirical measure $\mu_0^N \in \Pcal_c^N(\R^{2d})$, defined by 
\begin{equation}
\label{eq:EmpiricalDef}
\mu_0^N := \frac{1}{N} \sum_{i=1}^N \delta_{(X_0^i,Y_0^i)}, 
\end{equation}
the latter can be rewritten as the mean-field optimal control problem \eqref{Pcost} with initial datum $\mu^N_0$. In the following theorem, we show that when the regularization parameter $\lambda > 0$ is sufficiently large and the empirical samples satisfy
\begin{equation}
\label{eq:EmpiricalConv}
W_1(\mu_0^N,\mu_0) ~\underset{N \to +\infty}{\longrightarrow}~ 0, 
\end{equation}
then the minimizers and optimal values of the problems \eqref{PNCost} converge in a suitable sense towards those of \eqref{Pcost}. Even though we do not resort explicitly to this terminology in the sequel, this stability result amounts to showing that the sequence $(J^N)$ is $\Gamma$-converging towards $J$ for the weak topology of $L^2([0,T];\R^m)$ in the sense e.g. of \cite{DM}. Although it bears some interest and provides insights on the finite data consistency of the problem, the result that follows  is non-quantitative and purely based on compactness arguments. In order to obtain a quantitative version of this stability property, it is necessary to establish a smooth relation between optimal controls the $\theta^*$ and the data distributions $\mu_0$. Such a connection will be realized through the fundamental formula  \eqref{eq3} below, by leveraging the mean-field PMP studied in Section \ref{sec:4}.

\begin{thm}[Stability of finitely sampled costs and controls]
\label{thm:Stability}
Let $T,R > 0$ be given, $\mu_0 \in \Pcal_c(\R^{2d})$ be such that $\supp(\mu_0) \subset B(R)$, and assume that Assumptions \ref{asum1} and \ref{asum2} hold. Moreover, suppose that $\lambda > 0$ is sufficiently large in the sense of Theorem \ref{thm:Exist}. 

Then for every empirical approximating sequence $(\mu_0^N)$ satisfying \eqref{eq:EmpiricalDef}-\eqref{eq:EmpiricalConv}, the corresponding sequence of optimal controls $(\theta^N) \subset L^2([0,T];\R^m)$ is such that 
\begin{equation}
\label{eq:GammaConvCont}
\theta^N ~\underset{N \to +\infty}{\rightharpoonup}~ \theta^* \qquad \text{in $L^2([0,T];\R^m)$}, 
\end{equation}
where $\theta^* \in L^2([0,T];\R^m)$ is the unique solution of \eqref{Pcost}. Moreover, the optimal values converge as well, in the sense that
\begin{equation}
\label{eq:GammaConvCost}
J^N(\theta^N) ~\underset{N \to + \infty}{\longrightarrow}~ J(\theta^*) ~ = \min_{\theta \in L^2([0,T];\R^m)} J(\theta).
\end{equation}
\end{thm}

Before proving Theorem \ref{thm:Stability}, we state a useful auxiliary lemma exhibiting the dependence of the reduced empirical cost with respect to the sample size $N \geq 1$. 

\begin{lem}[Dependence of the reduced cost with respect to $N$]
\label{lem:CostDepN}
For every $\theta \in L^2([0,T];\R^m)$, there exists a constant $C(T,R,\|\theta\|_1) > 0$ such that  
\begin{equation}
\label{eq:CostDepN1}
\big| J(\theta) - J^N(\theta) \big| \leq C(T,R,\|\theta\|_1) W_1(\mu_0^N,\mu_0) 
\end{equation}
and 
\begin{equation}
\label{eq:CostDepN2}
\big\| \nabla J(\theta) - \nabla J^N(\theta) \big\|_2 \leq C(T,R,\|\theta\|_1) W_1(\mu_0^N,\mu_0)
\end{equation}
for each $N \geq 1$.
\end{lem}

\begin{proof}
Let us denote by $\mu,\mu^N \in \Ccal([0,T];\Pcal_c(\R^{2d})$ the solutions of \eqref{PDE} with control $\theta$ and initial data $\mu_0^N,\mu_0 \in \Pcal_c(\R^{2d})$ respectively. Under Assumptions \ref{asum1}, it follows from Theorem \ref{thm:Wellposed} that 
\begin{equation*}
\sup_{t \in [0,T]} W_1(\mu_t^N,\mu_t) \leq e^{L_{\Fcal,T,\|\theta\|_1}} W_1(\mu_0^N,\mu_0)
\end{equation*}
for some $L_{\Fcal,T,\|\theta\|_1} > 0$. This combined with Kantorovich's duality formula \eqref{Kanto} implies that
\begin{equation*}
\big| J(\theta) - J^N(\theta) \big| = \bigg| \INTDom{\ell(x,y)}{\R^{2d}}{\big( \mu_T - \mu_T^N \big)(x,y)} \bigg| \leq \Lip(\ell \, ; B(R)) \, e^{L_{\Fcal,T,\|\theta\|_1}} W_1(\mu_0^N,\mu_0), 
\end{equation*}
for each $N \geq 1$. Analogously by leveraging the analytical expression \eqref{eq:ReducedCostGrad} of the gradient of the reduced final cost , one also has that 
\begin{equation}
\label{eq:NablaEst1}
\begin{aligned}
& \big\| \nabla J(\theta) - \nabla J^N(\theta) \big\|_2^2 \\
& \hspace{1.25cm} \leq \INTSeg{ \bigg( \INTDom{\Big( \Rcal^{\theta}_{(t,T)}(x) \nabla_{\theta} \Fcal \big( t , \Phi^{\theta}_{(0,t)}(x),\theta_t \big) \Big)^{\hspace{-0.1cm} \top} \nabla_x \ell \big(\Phi^{\theta}_{(0,T)}(x), y \big)}{\R^{2d}}{\big( \mu_0 - \mu_0^N \big)(x,y)} \bigg)^2}{t}{0}{T}. 
\end{aligned}
\end{equation}
At this stage, one can check that as a consequence of Assumptions \ref{asum1} and \ref{asum2} along with the definition \eqref{eq:Resolvent} of the resolvent maps that there exists a constant $C'(T,R,\|\theta\|_1) > 0$ such that 
\begin{equation}
\label{eq:NablaEst2}
\INTSeg{\Big\| \Big( \Rcal^{\theta}_{(t,T)}(\cdot) \nabla_{\theta} \Fcal \big( t , \Phi^{\theta}_{(0,t)}(\cdot),\theta_t \big) \Big)^{\hspace{-0.1cm} \top} \nabla_x \ell \big(\Phi^{\theta}_{(0,T)}(\cdot), \cdot \big) \Big\|_{\Ccal^1(B(R))}^2}{t}{0}{T} \leq C'(T,R,\|\theta\|_1)^2.
\end{equation}
By combining \eqref{eq:NablaEst1} and \eqref{eq:NablaEst2} with an application of Kantorovich's duality formula \eqref{Kanto}, we finally obtain that 
\begin{equation*}
\big\| \nabla J(\theta) - \nabla J^N(\theta) \big\|_2 \leq C'(T,R,\|\theta\|_1) W_1(\mu_0^N,\mu_0)
\end{equation*}
for each $N \geq 1$, which concludes the proof of Lemma \ref{lem:CostDepN} by simply setting $C(T,R,\|\theta\|_1) := \max \big\{\Lip(\ell \, ; B(R)) \, e^{L_{\Fcal,T,\|\theta\|_1}} \, , \, C'(T,R,\|\theta\|_1) \big\}$.
\end{proof}

Building on these a priori estimates, we can move on to the proof of Theorem \ref{thm:Stability}.

\begin{proof}[Proof of Theorem \ref{thm:Stability}]
Observe first that and because $\supp(\mu_0^N) \subset B(R)$ for each $N \geq 1$ and we assumed $\lambda > 0$ to be sufficiently large, there exists a unique optimal control $\theta^N \in L^2([0,T];\R^m)$ solution of \eqref{PNCost} as a consequence of Theorem \ref{thm:Exist}. Noticing again that 
\begin{equation*}
J^N(\theta^N) \leq J^N(0) \leq \| \ell \|_{\Ccal^1(B(R))} +1 
\end{equation*}
for each $N \geq 1$, the sequence $(\theta^N)$ is uniformly contained in the closed ball $\Gamma \subset L^2([0,T];\R^m)$ whose radius is defined in Theorem \ref{thm:Exist}, and as such it admits a subsequence (that we do not relabel) which converges weakly to some $\theta^* \in L^2([0,T];\R^m)$. 

Our goal is to show that $\theta^*$ is the unique minimizer of $J$ and that the optimal values $(J^N(\theta^N))$ converge towards $J(\theta^*)$. To this end observe first that by Mazur's lemma (see e.g. \cite[Corollary 3.8]{Brezis}), there exists a sequence $(\tilde{\theta}^N)$ made of convex combinations of the elements of $(\theta^N)$ such that 
\begin{equation*}
\tilde{\theta}^N ~\underset{N \to +\infty}{\longrightarrow}~ \theta^* \qquad \text{in $L^2([0,T];\R^m)$}.
\end{equation*}
Recalling that $\theta^N$ are minimizers of $J^N$ and that these latter are uniformly equi-Lipschitz over $\Gamma$ as a consequence of Lemma \ref{lem:FinalCostReg}, it further holds that 
\begin{equation*}
\begin{aligned}
J^N(\theta^N) & \leq J^N(\tilde{\theta}^N) \\
& \leq J^N(\theta^*) + \big( \Lcal(T,R,\Gamma) + 2 \lambda \big) \big\| \theta^* - \tilde{\theta}^N \big\|_2,
\end{aligned}
\end{equation*}
for each $N \geq 1$. Using the stability estimate \eqref{eq:CostDepN1} of Lemma \ref{lem:CostDepN}, we can pass to the limit in the previous expression and obtain that 
\begin{equation}
\label{eq:GammaConvCost1}
\limsup_{N \to +\infty} J^N(\theta^N) \leq J(\theta^*).
\end{equation}
In order to recover a similar inequality for the liminf, notice that the reduced costs $J^N$ are convex by Proposition \ref{prop:Semiconvexity}, which implies that 
\begin{equation}
\label{eq:GammaConstCost2}
\begin{aligned}
J^N(\theta^N) & \geq J^N(\theta^*) + \big\langle \nabla J^N(\theta^*) , \theta^N - \theta^* \big\rangle_{L^2([0,T];\R^m)} \\
& \geq J^N(\theta^*) + \big\langle \nabla J(\theta^*) , \theta^N - \theta^* \big\rangle_{L^2([0,T];\R^m)} + \big\langle \nabla J(\theta^*) - \nabla J^N(\theta^*) , \theta^N - \theta^* \big\rangle_{L^2([0,T];\R^m)}
\end{aligned}
\end{equation}
for each $N \geq 1$. Observe now that by \eqref{eq:CostDepN2} in Lemma \ref{lem:CostDepN}, one has that
\begin{equation*}
\big\| \nabla J(\theta^*) - \nabla J^N(\theta^*) \big\|_2 ~\underset{N \to +\infty}{\longrightarrow}~ 0,
\end{equation*}
which together with the fact that $(\theta^N) \subset \Gamma$ is converging weakly towards $\theta^*$ then yields
\begin{equation*}
\big\langle \nabla J(\theta^*) - \nabla J^N(\theta^*) , \theta^* - \theta^N \big\rangle_{L^2([0,T];\R^m)} ~\underset{N \to +\infty}{\longrightarrow}~ 0,
\end{equation*}
by standard results on weak-strong convergence (see e.g. \cite[Proposition 3.5]{Brezis}).
Thus, by passing to the limit as $N \to +\infty$ in \eqref{eq:GammaConstCost2} while using \eqref{eq:CostDepN1} of Lemma \ref{lem:CostDepN}, we recover
\begin{equation}
\label{eq:GammaConvCost3}
J(\theta^*) \leq \liminf_{N \to + \infty} J^N(\theta^N), 
\end{equation}
which together with \eqref{eq:GammaConvCost1} finally implies that 
\begin{equation}
\label{eq:GammaConvCost4}
J^N(\theta^N) ~\underset{N \to +\infty}{\longrightarrow}~ J(\theta^*). 
\end{equation}
In order to conclude that $\theta^*$ is a minimizer of $J$, it is sufficient to consider a minimizing sequence $(\theta^n) \subset \Gamma$ for \eqref{Pcost} and to observe that by Lemma \ref{lem:CostDepN} and \eqref{eq:GammaConvCost4}, it holds that 
\begin{equation*}
\begin{aligned}
J(\theta^n) = \lim_{N \to +\infty} J^N(\theta^n) \geq \lim_{N \to +\infty} J^N(\theta^N) = J(\theta^*)
\end{aligned}
\end{equation*}
and to let $n \to +\infty$. The strict convexity of $J$ in turn provides the uniqueness of $\theta^*$, from whence we can deduce that it is the weak limit of the whole sequence $(\theta^N)$. 
\end{proof}


\section{Mean-Field Maximum Principle}
\label{sec:4}

In this section, we investigate first-order optimality conditions for the mean-field optimal control problem \eqref{Pcost}, which take the form of a mean-field Pontryagin Maximum Principle (``PMP'' for short). Their derivation -- which is based on a Lagrange multiplier rule for the convex calculus introduced in Section \ref{sec:2} -- is heuristically presented in Section \ref{sec:FormalLag}. After studying the well-posedness of the optimality system in Section \ref{sec:WellPosed}, we proceed to rigorously establish the PMP throughout Section \ref{sec:RigorDev}.


\subsection{Formal derivation of the Lagrangian maximum principle}
\label{sec:FormalLag}

We start this section by providing a formal derivation of the mean-field PMP. To this end, we first introduce the Lagrangian of the mean-field optimal control problem \eqref{Pcost}, defined by 
\begin{align}
\mc{L}(\mu,\theta,\psi) & =\int_{\RR^{2d}}\ell(x,y)\rd\mu_T(x,y)+\lambda\int_0^T|\theta_t|^2\dt\nn\\
& \hspace{0.41cm} +\int_{\RR^{2d}}\psi(0,x,y)\rd\mu_0(x,y)-\int_{\RR^{2d}}\psi(T,x,y)\rd\mu_T(x,y)\nn\\
&\hspace{0.41cm} +\int_0^T\int_{\RR^{2d}}\Big( \partial_t \psi(t,x,y) + \nabla_x \psi(t,x,y) \cdot \mathcal F(t,x,\theta_t) \Big)\rd\mu_t(x,y)\dt\,.
\end{align}
Next, we compute its functional derivatives with respect to the curves $\mu$ and $\theta$, namely
\begin{align*}
\frac{\delta \mc L}{\delta \mu_t} =
\begin{cases}
0,&\quad \text{for $t=0$ (the initial condition is fixed)}\\
\partial_t \psi+\nabla_x\psi \cdot \MF,&\quad \text{for $t \in (0,T)$},\\
\ell - \psi_T,&\quad \text{for $t=T$},
\end{cases}
\end{align*}
and
\begin{equation*}
\frac{\delta \mc L}{\delta \theta_t} = 2\lambda\theta_t^\top+ \INTDom{\nabla_x\psi\cdot \nabla_\theta \MF(t,x,\theta_t)}{\RR^{2d}}{\mu_t(x,y)} \,.
\end{equation*}
for almost every $t \in [0,T]$. Then, given an optimal trajectory-control pair $(\mu^*,\theta^*)$ for the problem \eqref{Pcost}, we will show that there exists a Lagrange multiplier $\psi^*$ such that
\begin{equation}
\frac{\delta \mc L}{\delta \mu}(\mu^*,\theta^*,\psi^*) = 0 \qquad \mbox{and} \qquad \frac{\delta \mc L}{\delta \theta}(\mu^*,\theta^*,\psi^*) =0\,.
\end{equation}
These latter will in turn provide us with the following backward adjoint dynamics 
\begin{equation}
\label{P2}
\partial_t\psi^*+\nabla_x\psi^*\cdot \mathcal F(t,x,\theta_t^*)=0,
\end{equation}
subject to the terminal condition $\psi^*_T =\ell$, along with the fixed-point equation 
\begin{equation}
\label{P3}
2\lambda\theta^{*\top}_t+\int_{\RR^{2d}}\nabla_x\psi^*\cdot \nabla_\theta \mathcal F(t,x,\theta^*_t)\rd \mu_t^*(x,y)=0,
\end{equation}
characterizing the optimal controls, where the curve $\mu^*$ satisfies the native forward dynamics
\begin{equation}
\label{P1}
\partial_t\mu_t^*+\nabla_x\cdot (\mathcal F(t,x,\theta_t^*)\mu_t^*)=0,\qquad \mu_t^*|_{t=0}=\mu_0\,.
\end{equation}
We will see below that \eqref{P2} is understood in the sense of \eqref{PMP2}, and that \eqref{P3} is understood in the sense of \eqref{PMP3}. 


\subsection{Well-posedness of the maximum principle}
\label{sec:WellPosed}

This section is devoted to discussing the existence and uniqueness of a solution $(\mu^*,\theta^*,\psi^*)\in \MC([0,T];\mc{P}_c(\RR^{2d}))\times \operatorname{Lip}([0,T];\RR^m)\times \mc{C}^1([0,T];\MC_c^2(\RR^{2d})) $ to the first-order optimality system 
\begin{empheq}[left=\empheqlbrace]{align}
\partial_t\mu_t^* & + \nabla_x\cdot (\mathcal F(t,x,\theta_t^*)\mu_t^*)=0, \hspace{0.5cm} \mu_t^*|_{t=0}=\mu_0, \label{eq1}\\
\partial_t\psi^* &+ \nabla_x\psi^*\cdot \mathcal F(t,x,\theta_t^\ast)=0, \hspace{0.5cm} \psi_t^*|_{t=T}=\ell , \label{eq2}\\
\theta^{*\top}_t & = -\frac{1}{2\lambda}\int_{\RR^{2d}}\nabla_x\psi^*\cdot \nabla_\theta \mathcal F(t,x,\theta^*_t)\rd\mu_t^*(x,y). \label{eq3}
\end{empheq}
To do so, we consider a compact and convex subset $\Gamma_{M,C}$ of the subspace $\mbox{Lip}([0,T];\RR^{m})\subset\MC([0,T];\RR^{m})$, defined by 
\begin{equation}
\label{Gm}
\Gamma_{M,C} := \Big\{ \theta \in\MC([0,T];\RR^{m})~\big|~|\theta_t-\theta_s|\leq M|t-s|,~\|\theta\|_\infty\leq C_{\Gamma} \Big\}\,.
\end{equation}
for some constants $M,C_{\Gamma} > 0$. We will also make use of the following ball in $L^2([0,T];\RR^m)$ 
\begin{equation}\label{Gm1}
\Gamma_C:=\left\{\theta\in L^2([0,T];\RR^m)~|~\|\theta\|_2\leq C_{\Gamma}T^{\frac{1}{2}}\right\}\,.
\end{equation}
One can easily notice that $\Gamma_{M,C}\subset \Gamma_C$.

\begin{thm}
\label{thmPMP} 
For any given $T>0$, take an initial data $\mu_0\in \mc{P}_c(\RR^{2d})$ and a terminal condition $\psi_T$ satisfying \eqref{tem}, let $\MF$ be a map satisfying Assumptions \ref{asum1} and \ref{asum2}, and suppose that $\lambda>0$ is large enough.

Then, there exists a triple $(\mu^*,\theta^*,\psi^*)\in \MC([0,T];\mc{P}_c(\RR^{2d}))\times \operatorname{Lip}([0,T];\RR^m)\times \mc{C}^1([0,T];\MC_c^2(\RR^{2d}))$ solution of \eqref{eq1}-\eqref{eq3}. Moreover, the control solution $\theta^*$ is unique in $\Gamma_{C}\subset L^2([0,T];\RR^m)$ defined as in \eqref{Gm1}, and $\psi^*\in \mc{C}^1([0,T];\MC_c^2(\RR^{2d}))$ is in characteristic form.
\end{thm}

\begin{rmk}
If there exists an optimal control $\theta^*\in L^2([0,T];\RR^m)$ satisfying the maximum principle \eqref{eq1}-\eqref{eq3}, then the uniqueness result in Theorem \ref{thmPMP} ensures that $\theta^*$ coincides with a Lipschitz continuous function almost everywhere. This means that in such a case there exists a smooth optimal control $\theta^*\in \operatorname{Lip}([0,T];\RR^m)$.
\end{rmk}

Using arguments that are similar to those of Theorem \ref{thm:Wellposed}, one can show the following result.

\begin{proposition}
\label{prop1}
Consider an initial data $\mu_0\in \mc{P}_c(\RR^{2d})$ with $\supp(\mu_0) \subset B(R)$ for some $R>0$, and let $\MF$ satisfy Assumption \ref{asum1}. Then for any $T>0$ and $\theta\in \Gamma_{M,C}$, there exists a unique solution $\mu^\theta\in\mc{C}([0,T];\mc{P}_c(\RR^{2d}))$ to \eqref{eq1} in the sense of Definition \ref{defweak}. Moreover, there exists some $R_T >0$ depending only on $R$ and $C_{\MF}$, such that
\begin{equation}
\label{suppt1}
\supp(\mu_t^\theta) \subset B(R_T) \qquad \mbox{for all } t\in[0,T]\,.
\end{equation}
Additionally, for any $s,t\in[0,T]$, it holds
\begin{equation}\label{conW}
W_1(\mu_t^\theta,\mu_s^\theta)\leq C(R,C_{\MF})|t-s|\,.
\end{equation}
If $\mu^{\theta,i}$, $i=1,2$ are two solutions with initial data $\mu_0^i$ satisfying the above assumptions, we have
\begin{equation}\label{stablity1}
W_1(\mu_t^{\theta,1},\mu_t^{\theta,2})\leq e^{L_{\MF,T,C_{\Gamma} }}W_1(\mu_0^1,\mu_0^2)\quad \mbox{ for all }t\in[0,T]\,.
\end{equation}
Here $C_{\MF}$ and $L_{\MF,T,C_{\Gamma} }$ are defined as in Assumption \ref{asum1} by replacing $\|\theta\|_1$ by $C_{\Gamma} T$. 
\end{proposition}

In what follows, we will only be interested  in what is happening inside the supports of $\mu^\theta$ for $\theta \in \Gamma_{M,C}$. Therefore, we shall recast the terminal condition in \eqref{eq2} as $\psi_T\in \MC_c^2(\RR^{2d})$ with
\begin{equation}
\label{tem}
\supp(\psi_T) = B(R_T) \quad \text{and} \quad \psi_T(x,y)=\ell(x,y) ~~ \mbox{ for all }x,y\in B(R_T)\,.
\end{equation}
In this context, we are able to derive the following norm estimate on $\psi^{\theta}$. 

\begin{proposition}
\label{prop2}
Suppose that $\MF$ satisfies Assumption \ref{asum1}. Then for any $T>0$ and $\theta\in \Gamma_{M,C}$, there exists a  unique characteristic solution $\psi^\theta\in\mc{C}^1([0,T];\MC_c^2(\RR^{2d}))$ to the equation \eqref{eq2} which terminal condition satisfies \eqref{tem}. Moreover it holds
\begin{equation}\label{rad}
\big\| \psi_t^\theta \big\|_{\MC_c^2(\RR^{2d})}\leq C(R',T,C_{\Gamma},C_{\MF},L_{\MF,T,C_{\Gamma}})\norm{\psi_T}_{\MC^2(B(R_T))}\,,
\end{equation}
for all times $t\in[0,T]$. Here the supports of $\psi_t^\theta$ satisfies the inclusion $\supp(\psi_t^\theta) \subset B(R_T')$ where $R' = R + (R+C_{\MF}T)e^{C_{\MF}T}$.
\end{proposition}

The results of Proposition \ref{prop2} are classical, and we postpone their proof to Appendix \ref{subsection:Flows}.

\begin{rmk}
Here, the fact that $\psi^{\theta}$ is a \emph{characteristic solution} means that it is obtained via the characteristic method, and is of the form $\psi^\theta(t,x,y)=\psi_T(\BPhi_{(T,t)}^{\theta}(x,y))$. Therein, we denoted by $(\BPhi_{(\tau,t)}^{\theta})_{\tau,t \in [0,T]}$ the flow maps defined as in \eqref{eq:Characteristics} with $\MF(t,x) := \MF(t,x,\theta_t)$. Characteristic solutions to \eqref{eq3} are unique because of the way they depends on the terminal condition \eqref{rad}.  Note here that we do not claim to have general uniqueness in $\mc{C}^1([0,T];\MC_c^2(\RR^{2d}))$ for \eqref{eq3}, i.e. there may exist $\mc{C}^1([0,T];\MC_c^2(\RR^{2d}))$ solutions that are not in the characteristic form. In what follows however, we will only consider characteristic solutions.
\end{rmk}

\begin{proof}[Proof of Theorem \ref{thmPMP}] 
The existence of optimal controls $\theta^*$ in $\Gamma_{M,C}$ is based on the Schauder fixed point theorem \cite[Theorem 11.1]{gilbarg2015elliptic}. Then, the uniqueness will be obtained by additionally showing that the underlying fixed-point map is a contraction in $\Gamma_C$. \medskip

\hspace{-0.75cm} $\bullet$ \textit{(Existence in $\Gamma_{M,C}$)} For any $\theta \in \Gamma_{M,C}$, denote by $\mu^{\theta}\in \mc{C}([0,T];\mc{P}_c(\RR^{2d}))$ the corresponding solution of \eqref{eq1} and by $\psi^\theta\in \mc{C}^1([0,T]; \MC_c^2(\RR^{2d}))$ the unique characteristic solution of \eqref{eq2}. In this context, we introduce the continuous mapping $\Lambda:~\Gamma_{M,C}\to \MC([0,T];\RR^m)$, defined by 
\begin{equation}
\Lambda(\theta)(t)^{\top} = -\frac{1}{2\lambda}\int_{\RR^{2d}}\nabla_x\psi_t^\theta\cdot \nabla_\theta \mathcal F(t,x,\theta_t)\rd\mu_t^\theta(x,y),
\end{equation}
for every $\theta \in \Gamma_{M,C}$ and all times $t \in [0,T]$. We start by checking that $\Lambda(\Gamma_{M,C})\subset \Gamma_{M,C}$ for $\lambda$ large enough. On the one hand, it follows Assumption \ref{asum1}-$(iii)$ and \eqref{rad} that
\begin{align*}
|\Lambda(\theta)(t)| & \leq \frac{1}{2\lambda} \int_{B(R_T)} \big| \nabla_x \psi_t^\theta\cdot \nabla_\theta \mathcal F(t,x,\theta_t) \big| \rd\mu_t^\theta(x,y) \\
& \leq \frac{1}{2\lambda}C(R_T,T)\sup_{t\in[0,T]} \big\| \psi_t^\theta \big \| _{\MC^1(B(R_T'))} \\
& \leq \frac{1}{2\lambda} C(R_T,T)C(R_T',T,C_{\Gamma},C_{\MF},L_{\MF,T,C_{\Gamma}}) \norm{\psi_T}_{\MC^1(B(R_T))},
\end{align*}
for all $t \in [0,T]$, with the explicit constant $R_T':= R + (R+C_{\MF}T)e^{C_{\MF}T}$. Hence, upon choosing a parameter $\lambda > 0$ that is large enough, it holds 
\begin{equation}
\norm{\Lambda(\theta)}_{L^{\infty}([0,T];\R^m)} \leq C_{\Gamma}.
\end{equation}
On the other hand, one has for any $s,t\in[0,T]$ that
\begin{align*}
|\Lambda(\theta)(t)-\Lambda(\theta)(s)|&\leq \frac{1}{2\lambda}\left|\int_{B(R_T)} \big( \nabla_x\psi_t^\theta-\nabla_x\psi_s^\theta \big) \cdot \nabla_\theta \mathcal F(t,x,\theta_t)\rd\mu_t^\theta(x,y)\right|\\
& \hspace{0.45cm} +\frac{1}{2\lambda}\left|\int_{B(R_T)}\nabla_x\psi_s^\theta\cdot (\nabla_\theta \mathcal F(t,x,\theta_t)-\nabla_\theta \mathcal F(s,x,\theta_s))\rd\mu_t^\theta(x,y)\right|\\
& \hspace{0.45cm}+\frac{1}{2\lambda}\left|\int_{B(R_T)}\nabla_x\psi_s^\theta\cdot \nabla_\theta \mathcal F(s,x,\theta_s)(\rd\mu_t^\theta-\rd\mu_s^\theta)(x,y)\right|\\
&=:I_1+I_2+I_3.
\end{align*}
Using the fact that $\psi^\theta \in \MC^1([0,T];\MC_c^2(\RR^{2d}))$ along with Assumption \ref{asum1}-$(iii)$, one can see that
\begin{equation}
\label{I1}
I_1\leq \frac{1}{2\lambda}C(R_T,T)|t-s|\, ,
\end{equation}
for all $s,t \in [0,T]$. Furthermore, it follows from assumption \eqref{time} and the estimate \eqref{rad} that
\begin{align}
\label{I2}
I_2&\leq \frac{1}{2\lambda}C(R_T)\sup_{t\in[0,T]} \big\| \psi_t^\theta \big\|_{\MC^1(B(R_T'))} \big( |t-s| + |\theta_t-\theta_s| \big) \nn\\
&\leq \frac{1}{2\lambda}C(R_T)C(R_T',T,C_{\Gamma},C_{\MF},L_{\MF,T,C_{\Gamma}})\norm{\psi_T}_{\MC^1(B(R_T))} M|t-s|,
\end{align}
with $R_T' := R + (R+C_{\MF}T)e^{C_{\MF}T}$. Lastly by Kantorovich's duality formula \eqref{Kanto}, one has
\begin{equation}
I_3\leq \frac{1}{2\lambda}\mbox{Lip} \big( \nabla_x\psi_s^\theta\cdot \nabla_\theta \mathcal F(s,\cdot,\theta_s) \, ; B(R_T) \big) W_1(\mu_t,\mu_s),
\end{equation}
and can further notice that
\begin{align*}
\mbox{Lip} \big( \nabla_x\psi_s^\theta\cdot \nabla_\theta \mathcal F(s,\cdot,\theta_s) \, ; B(R_T) \big)&  \leq C(R_T',T,C_{\Gamma},C_{\MF},L_{\MF,T,C_{\Gamma}})\norm{\psi_T}_{\MC^2(B(R_T))} \\
& \hspace{0.5cm} \times \Big( \|\nabla_\theta \MF(s,\cdot,\theta_s) \|_{L^\infty(B(R_T))}+\mbox{Lip}(\nabla_\theta \MF(s,\cdot,\theta_s) \, ; B(R_T)) \Big)\\
& \leq C(R_T',T,C_{\Gamma},C_{\MF},L_{\MF,T,C_{\Gamma}},R_T)\norm{\psi_T}_{\MC^2(B(R_T))},
\end{align*}
where we have used \eqref{rad} and Assumption \ref{asum2}-$(iii)$. This combined with \eqref{conW} thus yields 
\begin{equation}\label{I3}
I_3 \leq \frac{1}{2\lambda}C(R_T',T,C_{\Gamma},C_{\MF},L_{\MF,T,C_{\Gamma}},R_T)\norm{\psi_T}_{\MC^2(B(R_T))}|t-s|.
\end{equation}
Collecting estimates \eqref{I1}, \eqref{I2} and \eqref{I3}, we deduce that for $\lambda > 0$ large enough, it holds
\begin{equation}
|\Lambda(\theta)(t)-\Lambda(\theta)(s)|\leq M|t-s|.
\end{equation}
Thus, we have proven that $\Lambda(\Gamma_{M,C})\subset \Gamma_{M,C}$ when $\lambda > 0$ is taken to be sufficiently large. Hence by Schauder's fixed point theorem, the mapping $\Lambda$ has at least a fixed point $\theta^*$, namely
\begin{equation}
\theta^{*\top}=-\frac{1}{2\lambda}\int_{\RR^{2d}}\nabla_x\psi_t^{\theta^*}\cdot \nabla_\theta \mathcal F(t,x,\theta_t^*)\rd\mu_t^{\theta^*}(x,y).
\end{equation}
This concludes the existence part of the proof. \medskip

\hspace{-0.75cm} $\bullet$ \textit{(Uniqueness in $\Gamma_{C}$)} Our goal now is to prove that $\Lambda$ is a contraction over $\Gamma_{M,C}$ with respect to the $L^2$-norm, so that that the fixed point $\theta^* \in \Gamma_{M,C}$ is actually unique in $\Gamma_{C}$.  Indeed assuming that $\Lambda$ had two distinct fixed points $\theta^1$ and $\theta^2$, it would hold 
\begin{equation*} 
\|\theta^1 - \theta^2\|_2= \|\Lambda(\theta^1)-\Lambda(\theta^2)\|_2  \leq \kappa \|\theta^1-\theta^2\|_2,
\end{equation*}
which leads to a contradiction for contraction constants satisfying $0\leq \kappa <1$. In order to prove the contractivity of $\Lambda$, we start by fixing $t \in [0,T]$ and denote by $\mu^{\theta^1},\mu^{\theta^2}$ two solutions of \eqref{eq1} driven by $\theta^1,\theta^2$ respectively, with the same initial condition $\mu_0$. Similarly, denote by $\psi^{\theta^1},\psi^{\theta^2}$ the solutions of \eqref{eq2} generated by $\theta^1,\theta^2$ with the same terminal condition $\psi_T$. Then 
\begin{equation*}
\label{contractivity}
\begin{aligned}
& |\Lambda(\theta^1)(t)-\Lambda(\theta^2)(t)| \\
& = \frac{1}{2\lambda}\left|  \int_{\RR^{2d}}\nabla_x\psi_t^{\theta^1}\cdot \nabla_\theta \mathcal F(t,x,\theta_t^1)\rd\mu_t^{\theta^1}(x,y) -\int_{\RR^{2d}}\nabla_x\psi_t^{\theta^2}\cdot \nabla_\theta \mathcal F(t,x,\theta_t^2)\rd\mu_t^{\theta^2}(x,y) \right|
\end{aligned}
\end{equation*}
which can in turn be estimated by inserting suitable crossed terms as 
\begin{align*}
|\Lambda(\theta^1)(t)-\Lambda(\theta^2)(t)| & \leq \frac{1}{2\lambda} \left| \int_{\RR^{2d}}\nabla_x\psi_t^{\theta^1}\cdot \nabla_\theta \mathcal F(t,x,\theta_t^1)(\rd\mu_t^{\theta^1} - \rd \mu_t^{\theta^2})(x,y) \right|\\
& \quad + \frac{1}{2\lambda} \left| \int_{\RR^{2d}} \Big( \nabla_x\psi_t^{\theta^1}\cdot \nabla_\theta \mathcal F(t,x,\theta_t^1) - \nabla_x\psi_t^{\theta^2}\cdot \nabla_\theta \mathcal F(t,x,\theta_t^2) \Big) \rd\mu_t^{\theta^2}(x,y) \right| \\
& =: \frac{1}{2\lambda} (| I_1| +|I_2| ).
\end{align*}
We start by further simplifying the integral term $I_2$, which can be recast as 
\begin{align*}
|I_2 | & = \bigg| \int_{\RR^{2d}}\Big( \nabla_x\psi_t^{\theta^1}\cdot \nabla_\theta \mathcal F(t,x,\theta_t^1) - \nabla_x \psi_t^{\theta^2} \cdot \nabla_\theta \mathcal F(t,x,\theta_t^1) \\
& \qquad \qquad +\nabla_x\psi_t^{\theta^2}\cdot \nabla_\theta \mathcal F(t,x,\theta_t^1) - \nabla_x \psi_t^{\theta^2} \cdot \nabla_\theta \mathcal F(t,x,\theta_t^2) \Big) \rd\mu_t^{\theta^2}(x,y) \bigg|  \\
& \leq \left| \int_{\RR^{2d}} (\nabla_x \psi_t^{\theta^1} - \nabla_x \psi_t^{\theta^2} )\cdot \nabla_\theta \mathcal F(t,x,\theta_t^1) \rd\mu_t^{\theta^2}(x,y) \right|\\
&\quad+ \left| \int_{\RR^{2d}} \nabla_x \psi_t^{\theta^2} \cdot (\nabla_\theta \mathcal F(t,x,\theta_t^1)- \nabla_\theta \mathcal F(t,x,\theta_t^2) ) \rd\mu_t^{\theta^2}(x,y) \right| \\ 
& =: |I_3| +|I_4|.
\end{align*}
Hence, the estimate in \eqref{contractivity} is equivalent to
\begin{equation}
\label{split}
| \Lambda(\theta^1)(t)- \Lambda(\theta^2)(t) | \leq \frac{1}{2\lambda} ( |I_1| + |I_3| + |I_4|)\,.
\end{equation}
Let us focus on each term separately, starting with the integral $I_1$. Henceforth, we only consider the integrals over $B(R_T)$, in which the curves $\mu^{\theta^i}$ are supported for $i=1,2$. By using the same reasoning as in \eqref{I3}, we have that
\begin{align}
\label{I1_parts}
|I_1| &= \left| \int_{B(R_T)}\nabla_x\psi_t^{\theta^1}\cdot \nabla_\theta \mathcal F(t,x,\theta_t^1)(\rd\mu_t^{\theta^1}-\rd\mu_t^{\theta^2})(x,y) \right| \nn\\
& \leq \operatorname{Lip} \big( \nabla_x \psi_t^{\theta^1} \cdot \nabla_\theta \mathcal F(t,x,\theta_t^1) \, ; B(R_T) \big) W_1(\mu_t^{\theta^1}, \mu_t^{\theta^2}) \nn \\
& \leq C(R_T', T, C_{\Gamma}, C_{\MF}, L_{\MF,T,C_{\Gamma}},R_T) \norm{\psi_T}_{\MC^2(B(R'_T))}W_1(\mu_t^{\theta^1}, \mu_t^{\theta^2}).
\end{align}
Observe now that following Appendix \ref{subsection:Flows}, the curves $\mu_t^{\theta^1}$ and $\mu_t^{\theta^2}$ are characteristic solutions of \eqref{eq1}, in the sense that
\begin{equation}
\mu_t^{\theta^i} = \BPhi_{(0,t)}^{\theta^i} \sharp \mu_0 
\end{equation}
for all times $t \in [0,T]$, where $\BPhi_{(0,t)}^{\theta^i}(\cdot)$ are the flow maps of the underlying ODEs  
\begin{align*}
\frac{\rd X^i_t}{\dt}=\mathcal F(t,X^i_t,\theta^i_t),\qquad 	\frac{\rd Y^i_t}{\dt}=0, \qquad (X_0^i, Y_0^i) =(x_0,y_0)
\end{align*}
for $i=1,2$. Then, it follows from  Assumption \ref{asum1} that
\begin{align}
\label{esti}
\big| (X_t^1,Y_t^1)-(X_t^2,Y_t^2) \big| &= \left| \left(x_0-x_0+\int_0^t (\mathcal F(s,X_s^1,\theta_s^1)-\mathcal F(s,X_s^2,\theta_s^2)) \rd s, y_0-y_0\right)\right| \nn\\
& \leq \int_0^t \left|(\mathcal F(s,X_s^1,\theta_s^1)-\mathcal F(s,X_s^2,\theta_s^2)\right| \rd s \nn \\
& \leq \int_0^t \big|\mathcal F(t,X_s^1,\theta_s^1) -\mathcal F(t,X_s^2,\theta_s^1) \big| \rd s +\int_0^t \mathcal F(t,X_s^2,\theta_s^1) -\mathcal F(t,X_s^2,\theta_s^2) \big| \rd s \nn \\
& \leq  L_{\MF,T,C_\Gamma} \int_0^t |X_s^1-X_s^2| \rd s + C(R_T,T) \int_0^t |\theta_s^1-\theta_s^2| \rd s.
\end{align}
Then by Gronwall's lemma and the definition of Wasserstein distance, we obtain
\begin{equation}
\label{forw}
W_1(\mu_t^{\theta^1}, \mu_t^{\theta^2}) \leq W_1 \Big( \BPhi_{(0,t)}^{\theta^1} \sharp \mu_0,\BPhi_{(0,t)}^{\theta^2} \sharp \mu_0 \Big) \leq  C(R_T,T)e^{L_{\MF,T,C_\Gamma}T} \|\theta^1-\theta^2\|_2,
\end{equation}
and by using \eqref{forw} in \eqref{I1_parts}, it further holds that
\begin{equation}
\label{I_1}
|I_1|  \leq C(R_T', T, C_{\Gamma}, C_{\MF}, L_{\MF,T,C_{\Gamma}}, R_T)\norm{\psi_T}_{\MC^2(B(R_T))}\|\theta^1-\theta^2\|_2.
\end{equation}
We now shift our focus to the integral $I_3$. By Assumption \ref{asum2}-$(i)$, we have that
\begin{align}\label{I3parts}
|I_3| & = \left| \int_{B(R_T)}  (\psi_t^{\theta^1} - \psi_t^{\theta^2} )\nabla_x\cdot \nabla_\theta \mathcal F(t,x,\theta_t^1) \rd\mu_t^{\theta^2}(x,y) \right| \notag\\
& \leq C(R_T, T,C_{\Gamma}) \sup_{t \in [0,T]} \big\| \psi_t^{\theta^1} - \psi_t^{\theta^2} \big\|_{\MC(B(R_T'))}.
\end{align}
Recalling that $\psi^{\theta^1},\psi^{\theta^2}$ are characteristic solutions of \eqref{eq2} while using \eqref{psi_t_norm}, one further has
\begin{align}
\norm{\psi_t^{\theta^1} -\psi_t^{\theta^2}}_{\MC(B(R_T'))} & = \norm{\psi_T\big( \BPhi^{\theta^1}_{(t,T)} \big) - \psi_T\big( \BPhi^{\theta^2}_{(t,T)} \big)}_{\MC(B(R_T'))} \nonumber \\
& \leq \norm{\psi_T}_{\MC^1(B(R_T))} \norm{ \BPhi^{\theta^1}_{(t,T)} - \BPhi^{\theta^2}_{(t,T)}}_{C(B(R_T))}. \label{pippo}
\end{align}
Besides, it simply follows from Proposition \ref{prop:BoundFlow} that 
\begin{equation}\label{pluto}
\sup_{t\in [0,T]} \big\| \BPhi^{\theta^1}_{(t,T)} - \BPhi^{\theta^2}_{(t,T)} \big\|_{\Ccal(B(R_T))} \leq  C(R_T,T)e^{L_{\MF,T,C_\Gamma}T} \|\theta^1-\theta^2\|_2,
\end{equation}
for some given constant $C(R_T,T)e^{L_{\MF,T,C_\Gamma}T} > 0$. Therefore, the term $I_3$ can be estimated as
\begin{equation}\label{I_3}
|I_3| \leq C(T, R_T,C_{\Gamma},C_{\MF}) \norm{\psi_T}_{\MC^1((B(R_T))} \|\theta^1-\theta^2\|_2\,.
\end{equation}
Lastly, we focus on the integral quantity $I_4$. Using Assumption \eqref{asum1}-$(iv)$, we can write
\begin{align}\label{I_4}
|I_4| & \leq \int_{\RR^{2d}} \nabla_x \psi_T^{\theta^2} \cdot \left| \nabla_\theta (\mathcal F(t,x,\theta_t^1)-\mathcal F(t,x,\theta_t^2)) \right| \rd\mu_t^{\theta^2}(x,y) \nn \\
& \leq \int_{\RR^{2d}} \nabla_x \psi_T^{\theta^2} \cdot \left|\nabla_\theta^2 (\mathcal F(t,x,\theta)\right| \left|\theta_t^1-\theta_t^2\right| \rd\mu_t^{\theta^2}(x,y) \nn \\
& \leq C(R_T,T) |\theta_t^1-\theta_t^2| \sup_{t \in [0,T]} \norm{\psi_t^{\theta^2}}_{C^1(B(R_T')) } \nn\\
& \leq  C(R_T, T,R_T', C_{\Gamma}, C_{\MF}, L_{\MF,T, C_{\Gamma}}) \big\| \psi_T \big\|_{\MC^1(B(R_T))}|\theta_t^1-\theta_t^2|\,.
\end{align}
Collecting the estimates from \eqref{I_1}, \eqref{I_3} and \eqref{I_4}, we can conclude 
\begin{align*}
\|\Lambda(\theta^1) - \Lambda(\theta^2)\|_2 & \leq \frac{1}{2\lambda} C(R_T', R_T,T, C_{\MF}, C_{\Gamma}, L_{\MF,T,C_{\Gamma}}) \norm{\psi_T}_{\MC^2(B(R_T))}\|\theta^1 - \theta^2\|_2 \\
& = \kappa_{\lambda} \|\theta^1- \theta^2\| _2\,.
\end{align*}
Hence by choosing the parameter $\lambda > 0$ to be large enough, we obtain that $\kappa_{\lambda} < 1$, which means that the mapping $\Lambda : ~\Gamma_{M,C}\to ~\Gamma_{M,C}$ is a contraction and thus that its fixed point $\theta^*$ is unique in $\Gamma_{C}$.  Thus we have obtained a solution $(\mu^*,\theta^*,\psi^*)\in \MC([0,T];\mc{P}_c(\RR^{2d}))\times \Gamma_{M,C}\times \mc{C}^1([0,T];\MC_c^2(\RR^{2d})) $ to equations \eqref{eq1}-\eqref{eq3}, and it is unique in $\Gamma_{C}$.
\end{proof}

\begin{rmk}
\label{remark_lambda}
As it was shown in the proof above, the size condition imposed on $\lambda$ depends on some constant $C(|R_T'|, R_T,T, C_{\MF}, C_{\Gamma}, L_{\MF,T,C_{\Gamma}})$ and $\norm{\psi_T}_{\MC^2(R_T)}$. Especially for the case $\Fcal(t,x,\theta):=\tanh(\theta x)$, we can simplify the constant as $C(R_T,T,C_{\Gamma})$, which shows that $\lambda$ depends on the size of the support of $\mu_0$, on the final time $T>0$ and on the constant $C_{\Gamma}$.
\end{rmk}

In addition to its usefulness in characterizing and computing optimal controls, the mean-field maximum principle allows us to derive a quantitative norm rate of convergence of the latter with respect to the $L^p$-norms and a quantitative generalization error.

\begin{cor}
\label{thmrate} 
For any $T,>0$, let $\mu_0\in \mc{P}_c(\RR^{2d})$ be such that $\supp(\mu_0) \subset B(R)$ and and $\psi_T$ be a terminal condition satisfying \eqref{tem}, and suppose Assumptions \ref{asum1} and \ref{asum2} hold. Moreover, assume that for each $N \geq 1$ we are given an approximating empirical measure of the form
\begin{equation*}
\mu_0^N :=  \frac{1}{N} \sum_{i=1}^N \delta_{(X^i_0,Y^i_0)} \in \mc{P}_c^N(\RR^{2d}), 
\end{equation*}
such that 
\begin{equation*}
\lim_{N\to \infty} W_1(\mu^N_0,\mu_0)=0.
\end{equation*}
Let $\lambda>0$ be sufficiently large so that $(\mu^*,\theta^*,\psi^*)\in \MC([0,T];\mc{P}_c(\RR^{2d}))\times \operatorname{Lip}([0,T];\RR^m)\times \mc{C}^1([0,T];\MC_c^2(\RR^{2d}))$ and $(\mu^N,\theta^N,\psi^N)\in \MC([0,T];\mc{P}_c^N(\RR^{2d}))\times \operatorname{Lip}([0,T];\RR^m)\times \mc{C}^1([0,T];\MC_c^2(\RR^{2d})) $ are the unique solutions of \eqref{eq1}-\eqref{eq3} with initial conditions $\mu_0$ and $\mu_0^N$ respectively. Then
\begin{equation}
\label{rate}
\max\bigg\{ \|\theta^N- \theta^*\|_{p} \; , \, \sup_{t \in [0,T]}W_1(\mu^N_t,\mu^*_t) \, , \, \|\psi^N-\psi^*\|_{\mathcal C([0,T] \times B(R_T))} \bigg\} \leq C W_1(\mu_0^N,\mu_0),
\end{equation}
for a constant $C > 0$ which only depends on the parameters of the model and $p \in [1,\infty]$, and where $R_T > 0$ is defined as in Proposition \ref{prop1} above. In particular, we obtain the following quantitative generalization error estimate
\begin{equation}
\label{generr}
\bigg| \int_{\mathbb R^{2d}} \ell(x,y) \ud \mu_T^*(x,y) - \frac{1}{N} \sum_{i=1}^N \ell \big( X^i_T,Y^i_T \big)  \bigg| \leq C W_1(\mu_0^N,\mu_0).
\end{equation}
\end{cor}

\begin{proof}
By using similar arguments as in the proof of Theorem \ref{thmPMP}, see in particular \eqref{stablity1} and \eqref{esti}-\eqref{forw}, we can prove the stability estimate
\begin{align}
\label{estima1}
\sup_{t \in [0,T]}W_1(\mu_t^N,\mu_t^*) & \leq \sup_{t \in [0,T]}W_1(\mu_t^N,\mu_t^{\theta^N})+\sup_{t \in [0,T]}W_1(\mu_t^{\theta^N},\mu_t^{*})\notag\\
& \leq C \left (W_1(\mu_0^N,\mu_0)+\int_0^T |\theta^N_t- \theta^*_t| \dt \right) \\
& \leq C \Big( W_1(\mu_0^N,\mu_0)+\|\theta^N- \theta^*\|_{p} \Big),
\end{align}
where $\mu_t^{\theta^N}$ is the unique solution of \eqref{P1} driven by $\theta^N$ with initial datum $\mu_0$, and $C > 0$ is an overloaded constant depending on the data of the problem. Similarly, from \eqref{esti}, \eqref{pippo} and \eqref{pluto}, we have that
\begin{equation}\label{estima2}
\|\psi^N-\psi^*\|_{\mathcal C([0,T] \times B(R_T))} \leq C \int_0^T |\theta^N_t- \theta^*_t| \dt \leq C \|\theta^N- \theta^*\|_{p}, 
\end{equation}
for any $p \in [1,+\infty]$. Finally, by using the fixed point equations 
\begin{equation*}
\theta^N=\Lambda(\theta^N) \qquad \text{and} \qquad \theta^*=\Lambda(\theta^*),
\end{equation*}
and following the estimates in the proof of  Theorem \ref{thmPMP}, see in particular \eqref{split}, \eqref{I1_parts}, \eqref{I3parts} and \eqref{I_4},  we obtain
\begin{equation*}
\begin{aligned}
\|\theta^N- \theta^*\|_{p} &=\|\Lambda(\theta^N)-\Lambda(\theta^*)\|_p  \\
&\leq \frac{C}{\lambda} \bigg( \|\theta^N- \theta^*\|_{p} + \sup_{t \in [0,T]} W_1(\mu_t^N,\mu_t^*)+ \|\psi^N-\psi^*\|_{\mathcal C([0,T] \times B(R_T))} \bigg)\\
&\leq \frac{C}{\lambda}\left  (W_1(\mu_0^N,\mu_0)+\|\theta^N- \theta^*\|_{p} \right),
\end{aligned}
\end{equation*}
where we applied \eqref{estima1} and \eqref{estima2} in the last inequality. Hence for $\lambda>0$ large enough, it holds 
\begin{equation}
\label{estima3}
\|\theta^N- \theta^*\|_{p} \leq CW_1(\mu_0^N,\mu_0).
\end{equation}
Combining now \eqref{estima1}, \eqref{estima2}  and \eqref{estima3} finally yields \eqref{rate}. The generalization error displayed in \eqref{generr} follows from \eqref{estima1} and \eqref{estima3}, since 
\begin{equation*}
\begin{aligned}
\left | \int_{\mathbb R^{2d}} \ell(x,y) \ud (\mu_T^*(x,y)-\mu^N_T(x,y)) \right | &\leq \operatorname{Lip}(\ell \, ; B(R_T)) \sup_{t \in [0,T]} W_1(\mu_t^N,\mu_t^*)\\
&\leq C \left (W_1(\mu_0^N,\mu_0)+\|\theta^N- \theta^*\|_{p} \right ) \\
&\leq C W_1(\mu_0^N,\mu_0).
\end{aligned}
\end{equation*}
This completes the proof of Corollary \ref{thmrate}.
\end{proof}

\begin{rmk}[Data bounds, regularization parameters and error estimates] The estimate \eqref{generr} is in the worst case affected by the curse of dimension, although it will not be the case in practice e.g. for networks driven by sigmoid activation functions. The constant $C$ in \eqref{generr} is encoding the complexity of the NeurODE and is derived as a consequence of \eqref{estima3} as
$$
C=C_1(1- C_0/\lambda) >0.
$$
Therein, the constant $C_0 > 0$ may depend exponentially on the constants $C_{\MF}$ and $L_{\MF}$ appearing in Assumptions \ref{asum1} -- and in particular on the dimension $d \geq 1$ of the state space --, and polynomially on those of Assumptions \ref{asum2}, owing to the pessimistic nature of deterministic Gr\"onwall estimates. Thus, as long as the worst-case Gr\"onwall estimates do indeed reflect the actual stability of the PMP, the constant $C_0 > 0$ may be extremely large. Nevertheless, in the case of sigmoidal-type activation functions such as $\rho := \tanh$, we detailed in Remark \ref{rmtanh} how the uniform boundedness of the velocity field $\mathcal{F}$ implied a polynomial dependence of all the relevant constants of the problem with respect to the state space dimension. Therefore, in that particular yet relevant case, the quantity $C_0$ will in fact scale polynomially and not exponentially with $d$.

For arbitrary initial measures $\mu_0$, it is known that empirical measures $\mu_N$ supported on finite samples satisfy the estimate
$$
\mathbb E \big[ W_1(\mu_0^N,\mu_0) \big] \leq C N^{-1/d},
$$
see for instance \cite{10.1214/12-AIHP489,FG15}, which scales quite badly with the dimension $d \geq 1$ of the state space. However, if $\mu_0$ is concentrated around manifolds of lower dimension, then the factor $C > 0$ depends favorably on that intrinsic lower dimension \cite{10.3150/18-BEJ1065}. In practice, it is expected that data distributions do concentrate around such lower-dimensional structures.
\end{rmk}


\subsection{Rigorous derivation of the mean-field maximum principle}
\label{sec:RigorDev}

The previous section, we proved the well-posedness of the mean-field PMP \eqref{eq1}-\eqref{eq3} in the class of control that are Lipschitz continuous with respect to time. Under this  assumption, we rigorously derive in what follows the optimality conditions by using a generalized Lagrange multiplier theorem over convex sets. The method we present is to a certain extent a standard calculus of variations approach, and allows to bypass the more technical ones based either on the abstract differential calculus of Wasserstein as in \cite{SetValuedPMP,PMPWassConst,PMPWass}, or on the fine structural results for continuity equations leveraged in \cite{burger}.

Let it be stressed that the requirement of continuity of the control is purely technical, and stems from our use of \cite[Theorem 1]{piccoli2019wasserstein} concerning the well-posedness of transport equations with sources. Were such results available in the case where the source terms are merely measurable in time -- which seems true but is not written anywhere yet --, we could then remove the continuity assumption and prove the mean-field PMP in its full generality using the Lagrangian approach.

\subsubsection{A Lagrange Multiplier Theorem over convex sets}
\label{sec:LagMult}

Let $X$ and $Y$ be Banach spaces, $E\subset X$ be a convex set, $J: E \rightarrow \RR$ be a continuous functional and $G:E\rightarrow Y$ be a linear mapping, both continuously $F$-differentiable on $E$ in the sense of \eqref{A2}. For $x^*\in E$, we introduce the notation
\begin{equation}
DG(x^*):=\Big\{ L\in \mc L(X_E,Y) ~\big|~ L \mbox{ satisfies } \eqref{A1}\Big\}\,.
\end{equation}
It is known that every $L \in \mc{L}(X_E,Y)$ can be uniquely extended to a operator $\overline L\in \mc L(\overline{X}_E,Y)$ over the Banach space $\overline{X}_E$. In what follows, we will slightly abuse the notation $DG(x^*)$ to denote the set of operators obtained after extending the convex subgradients to $\overline{X}_E$. 

In the following theorem, we extend the Lagrange multiplier theorem for the Banach space \cite[Section 4.14]{zeidler1995applied} to the setting of the calculus for convex subsets introduced in Section \ref{sec:2}. To ease the readability of the paper, the proof of this result is reported in Appendix \ref{subsection:Lagrange}.
\begin{thm}\label{thmla}
Let $x^* \in E$ be a solution of the constrained optimization problem
\begin{equation}
\label{eqopt}
\left\{
\begin{aligned}
\inf\limits_{x\in E} & J(x), \\ 
\textnormal{s.t.} ~ & G(x)=0.
\end{aligned}
\right.
\end{equation}
Suppose moreover that the inclusion $x^*+X_E\subset E$ holds, and that there exists some $G'(x^*)\in DG(x^*)$ that is a surjective operator from $\overline {X}_E$ into $Y$. Then for any $J'(x^*)\in DJ(x^*)$, there exists a non-zero covector $ p^* \in Y'$ which satisfies 
\begin{equation}
 \la J'(x^\ast),z \ra+\la G'(x^\ast)z,p^*\ra = 0
\end{equation}
for all $z\in \overline{X}_E$.
\end{thm}

\subsubsection{Preparation and verification of assumptions}

Recall that in Theorem \ref{thm:Wellposed},  we have shown that for every $\theta\in L^2([0,T];\RR^m)$, there exists a unique solution $\mu\in \mc{C}([0,T];\mc{P}_c(\RR^{2d}))$ to the continuity equation \ref{PDE}. In the sequel, we assume that $\theta\in \MC([0,T];\RR^m)$ so that the map $t\mapsto \mathcal F(t,x,\theta_t)$ is continuous on $[0,T]$, and that 
$\mathcal F$ satisfies Assumption \ref{asum1}.

Under these working assumption we can further prove that the solution $\mu$ is such that $\partial_t\mu \in \MC([0,T];(\MC_b^1(\RR^{2d}))')$. Indeed for any $\varphi\in \MC_b^1(\RR^{2d})$, one has
\begin{align}
\label{conti}
\norm{\partial_t\mu_t}_{(\MC_b^1(\RR^{2d}))'} & =\sup\limits_{\norm{\varphi}_{\MC_b^1}\leq 1}|\la \partial_t\mu_t,\varphi\ra| \\
& = \sup\limits_{\norm{\varphi}_{\MC_b^1}\leq 1}|\la \mathcal F(t,\cdot,\theta_t)\mu_t,\nabla_x\varphi\ra|\notag \\
& \leq \norm{\mathcal F}_{L^\infty(\supp(\mu_t))}\leq C_{\MF}(1+|R_T|).
\end{align}
Additionally, it holds for any $s,t \in[0,T]$ that
\begin{align}
\label{conti1}
\norm{\partial_t\mu_t-\partial_s\mu_s}_{(\MC_b^1(\RR^{2d}))'} & = \sup\limits_{\norm{\varphi}_{\MC_b^1}\leq 1}|\la \partial_t\mu_t-\partial_s\mu_s,\varphi\ra| \\
& = \sup\limits_{\norm{\varphi}_{\MC_b^1}\leq 1}|\la \mathcal F(t,\cdot,\theta_t)\mu_t-\mathcal F(s,\cdot,\theta_s)\mu_s,\nabla_x\varphi\ra|\notag\\
& \leq \sup\limits_{\norm{\varphi}_{\MC_b^1}\leq 1} \Big| \big\la (\mathcal F(t,\cdot,\theta_t)-\mathcal F(s,\cdot,\theta_s))\mu_t,\nabla_x\varphi \big\ra \Big| \\
& \hspace{0.4cm} +\sup\limits_{\norm{\varphi}_{\MC_b^1}\leq 1}|\la \mathcal F(s,\cdot,\theta_s)(\mu_t-\mu_s),\nabla_x\varphi\ra| \\
& \leq C|t-s|+\sup\limits_{\norm{\varphi}_{\MC_b^1}\leq 1}|\la \mathcal F(s,\cdot,\theta_s)(\mu_t-\mu_s),\nabla_x\varphi\ra|,
\end{align}
Observe that by standard density results, there exists for every $\varphi\in \MC_b^1(\RR^{2d})$ a sequence $(\varphi^n) \subset \MC_b^2(\RR^{2d})$ such that $\norm{\varphi^n-\varphi}_{\MC^1_b(\R^{2d})} \rightarrow 0$ as $n\to +\infty$. Thus, one has that
\begin{align}
\label{comar1}
& \hspace{-0.15cm} \sup\limits_{\norm{\varphi}_{\MC_b^1}\leq 1}|\la \mathcal F(s,\cdot,\theta_s)(\mu_t-\mu_s),\nabla_x\varphi\ra|\notag\\
& \leq \sup\limits_{\norm{\varphi}_{\MC_b^1}\leq 1} \big| \big\la \mathcal F(s,\cdot,\theta_s)(\mu_t-\mu_s),(\nabla_x \varphi - \nabla_x\varphi^n) \big\ra  \big| + \sup\limits_{\norm{\varphi}_{\MC_b^1}\leq 1}|\la \mathcal F(s,\cdot,\theta_s)(\mu_t-\mu_s),\nabla_x\varphi^n\ra|\notag\\
& \leq C\norm{\varphi^n-\varphi}_{\MC^1_b(\R^{2d})} + \mbox{Lip} \big( \MF(t,\cdot,\theta_t) \cdot \nabla_x\varphi^n \big) W_1(\mu_t,\mu_s) \\
& \leq C\norm{\varphi^n-\varphi}_{\MC^1_b(\R^{2d})} + C_n|t-s| ,
\end{align}
where  we have used the Kantorovitch duality \eqref{Kanto} and \eqref{Was}, which further yields that 
\begin{equation}
\label{comar1bis}
\lim_{s \to t}\sup\limits_{\norm{\varphi}_{\MC_b^1}\leq 1}|\langle \mathcal F(s,\cdot,\theta_s)(\mu_t-\mu_s),\nabla_x\varphi \rangle| \leq \norm{\varphi^n-\varphi}_{\MC^1_b(\R^{2d})},
\end{equation}
for every $n \in \NN$. Therefore letting $n\to + \infty$ in \eqref{comar1bis}, we can conclude 
\begin{equation}
\label{comar2}
\sup\limits_{\norm{\varphi}_{\MC_b^1}\leq 1}|\la \mathcal F(s,\cdot,\theta_s)(\mu_t-\mu_s),\nabla_x\varphi\ra| ~\underset{s \to t}{\longrightarrow}~ 0.
\end{equation}
This combined with \eqref{conti1} and the fact that $t \mapsto \MF(t,x,\theta_t) \in \RR^d$ is continuous implies that $\partial_t\mu \in \MC([0,T];(\MC_b^1(\RR^{2d}))')$. In the sequel, we will therefore consider trajectory-control pairs $(\mu^*,\theta^*)\in \mc{C}^1([0,T];(\MC_b^1(\RR^{2d}))')\times \MC([0,T];\RR^{m})$ solution of the optimal control problem \eqref{Pcost}, where we have used the notation $\mu\in\mc{C}^1([0,T];(\MC_b^1(\RR^{2d}))')$ to represent that $\mu\in \mc{C}([0,T];(\MC_b^1(\RR^{2d}))')$ and $\partial_t\mu\in\mc{C}([0,T];(\MC_b^1(\RR^{2d}))')$.

\vspace{-0.35cm}

\paragraph*{$\circ$ The setup of spaces and sets.} Let us start by defining the spaces
\begin{equation}
V:=\tilde{\mc{C}}([0,T];\mc{M}_{1,c}(\RR^{2d})) \, \cap \, \MC^1([0,T];(\MC_b^1(\RR^{2d}))') \qquad \text{and} \qquad  Q:=\MC([0,T];\RR^{m}),
\end{equation}
where
\begin{equation}
\begin{aligned}
\tilde{\mc{C}}([0,T];\mc{M}_{1,c}(\RR^{2d})) := \bigg\{\mu\in \MC([0,T];\mc{M}_{1,c}(\RR^{2d})) ~\big|~ & \supp(\mu_t) \subset S_{\mu}~ \text{for all $t \in [0,T]$} \\
& \text{where $S_{\mu} \subset \R^d$ is a compact set} \bigg\},
\end{aligned}
\end{equation}
and fix
\begin{equation}
\label{DefE}
E := V\times Q ~=~ \tilde{\mc{C}}([0,T];\mc{M}_{1,c}(\RR^{2d})) \, \cap \, \MC^1([0,T];(\MC_b^1(\RR^{2d}))')\times \MC([0,T];\RR^{m}).
\end{equation}
Clearly, $(\mu^*,\theta^*)\in E$ since $\mc{P}_c(\RR^{2d})\subset \mc{M}_{1,c}(\RR^{2d})$. We also observe that $E$ is a convex subset of the Banach space
\begin{equation}
X:= U \times Q = \MC^1([0,T];(\MC_b^1(\RR^{2d}))')\times \MC([0,T];\RR^{m}).
\end{equation}
Due to this embedding, we shall from now on endow $\mc{M}_{1,c}(\RR^{2d})$ with the weak$-^*$ topology of $(\MC_b^1(\RR^{2d}))'$. In what follows, we use the notation $U_V := \RR(V-V)$ as well as the identity
\begin{equation}
U_V:=\tilde{\mc{C}}([0,T];\mc{M}_{0,c}(\RR^{2d}))\cap\MC^1([0,T];(\MC_b^1(\RR^{2d}))').
\end{equation}
For $\nu \in V$, we shall define $U_\nu$ as the convex cone of directions
\begin{equation}
U_{\nu}:=\RR_+(V-\nu)\subset U_V\, , 
\end{equation}
in keeping with the concepts introduced in Section \ref{sec:2}. In fact, one can easily check that $U_{\nu}=U_V$, since for any $\mu\in U_V$, one has $\mu=\mu+\nu-\nu$ with $\mu+\nu\in V$.
Next we introduce 
\begin{equation}\label{DefXE}
X_E:=U_V\times Q  ~=~ \tilde{\mc{C}}([0,T];\mc{M}_{0,c}(\RR^{2d}))\cap\MC^1([0,T];(\MC_b^1(\RR^{2d}))')\times \MC([0,T];\RR^{m}).
\end{equation}
that is seen as a convex subset of $X$. It follows from the definitions of $E$ and $X_E$ that $(\mu^*,\theta^*)+X_E\subset E$, which is compatible with the assumptions of Theorem \ref{thmla}.

\paragraph*{$\circ$ The setup of maps.} For any $(\mu,\theta)\in E$, we denote the full cost functional of \eqref{Pcost} by
\begin{equation}
J(\mu,\theta):=\int_{\RR^{2d}}\ell(x,y)\rd\mu_T(x,y)+ \lambda \int_0^T|\theta_t|^2\dt,
\end{equation}
and observe that it is a map from $E$ into $\RR_+$. We also introduce the notation
\begin{equation}
\label{Gform}
G(\mu,\theta) :=-\partial_t\mu-\nabla_x\cdot (\mathcal F(t,x,\theta)\mu)\,.
\end{equation}
Seeing $G(\mu,\theta)$ as time-dependent quantity, it is easy to check that $ G(\mu,\theta)\in \MC([0,T];(\MC_b^1(\RR^{2d}))')$ for $(\mu,\theta)\in E$, and that $\la G(\mu,\theta)_t,1 \ra=0$ for all $t\in[0,T]$.  Indeed for any $\varphi\in \MC_b^1(\RR^{2d})$, it holds
\begin{align*}
\norm{G(\mu,\theta)_t-G(\mu,\theta)_s}_{(\MC_b^1)'} &= \sup\limits_{\|\varphi\|_{\MC_b^1}\leq 1}\left|\la G(\mu,\theta)_t-G(\mu,\theta)_s,\varphi\ra\right| \\
& = \norm{\partial_t\mu_t-\partial_s\mu_s}_{(\MC_b^1)'} + \sup\limits_{\| \varphi\|_{\MC_b^1}\leq 1} \big| \big\la (\mathcal F(t,\cdot,\theta_t)-\mathcal F(s,\cdot,\theta_s))\mu_t, \nabla\varphi \big\ra \big| \\
& \hspace{0.45cm} +\sup\limits_{\|\varphi\|_{\MC_b^1}\leq 1} \big| \big\la \mathcal F(s,\cdot,\theta_s)(\mu_t-\mu_s), \nabla\varphi \big\ra \big|.
\,.
\end{align*}
By performing density arguments similar to those of \eqref{comar1}-\eqref{comar2}, one has that
\begin{equation}
\sup\limits_{\|\varphi\|_{\MC_b^1}\leq 1}\left|\la \mathcal F(s,x,\theta_s)(\mu_t-\mu_s), \nabla\varphi\ra\right|\leq C\|\mu_t-\mu_s\|_{(\MC^1)'}.
\end{equation}
This with together with the fact that $\mu \in \MC^1([0,T];(\MC_b^1(\RR^{2d}))') $ and that $t \in [0,T] \mapsto \mc F(t,\cdot,\theta_t)$ is continuous in time yields $ G(\mu,\theta)\in \MC([0,T];(\MC_b^1(\RR^{2d}))')$. Observe now that for any $\mu \in \tilde{\mc{C}}([0,T];\mc{M}_{1,c}(\RR^{2d}))$, there exists some compact set $S_{\mu} \subset \RR^d$ such that 
\begin{equation}
\supp(\mu_t) \subset S_{\mu} \quad \mbox{ for all }t\in[0,T]\, .
\end{equation} 
This implies that $G(\mu,\theta)$ is uniformly compactly supported in the sense of  distribution, namely $G : E \to Y_0$ with
\begin{align*}
Y_0 :&=\tilde{\mc{C}}([0,T];(\MC_b^1(\RR^{2d}))'_{0,c})\nn\\
& =\bigg\{ g\in \MC([0,T];(\MC_b^1(\RR^{2d}))') ~\big|~ \la g_t,1\ra=0 \mbox{ and } \supp(g_t) \subset S_{g}\Subset \RR^{2d}, ~ \forall t\in[0,T] \bigg\}.
\end{align*}
This allows us to define the Banach space
\begin{equation}
Y:=\overline{Y}_0 = \overline{\tilde{\mc{C}}([0,T];(\MC_b^1(\RR^{2d}))'_{0,c})},
\end{equation}
which is a closed subspace of the  Banach space $\MC([0,T];(\MC_b^1(\RR^{2d}))')$.

Now let us verify that $G\in \MC^1(E;Y)$ and $J\in \MC^1(E;\RR)$. For any $t\in[0,T]$, it holds that
\begin{align*}
\norm{G(\mu^1,\theta^1)_t-G(\mu^2,\theta^2)_t}_{(\MC_b^1(\RR^{2d}))'} & = \sup\limits_{\|\varphi\|_{\MC_b^1}\leq 1}\left|\la G(\mu^1,\theta^1)_t-G(\mu^2,\theta^2)_t,\varphi\ra\right| \\
& = \norm{\partial_t\mu_t^1-\partial_t\mu_t^2}_{(\MC_b^1)'}+ \sup\limits_{\|\varphi\|_{\MC_b^1}\leq 1}\left|\la \mathcal F(t,x,\theta_t^1)(\mu_t^1-\mu_t^2), \nabla\varphi\ra\right| \\
& \hspace{0.45cm} +\sup\limits_{\|\varphi\|_{\MC_b^1}\leq 1}\left|\la (\mathcal F(t,x,\theta_t^1)-\mathcal F(t,x,\theta_t^2))\mu_t^2, \nabla\varphi\ra\right| \\
& \hspace{-0.25cm} \leq \norm{\partial_t\mu_t^1-\partial_t\mu_t^2}_{(\MC_b^1)'} +C\norm{\mu_t^1-\mu_t^2}_{(\MC_b^1)'}+C(R_T,T)|\theta_t^1-\theta_t^2|\
\end{align*}
where we have again used  density arguments similar to that of \eqref{comar1}-\eqref{comar2}. Thus, we have proven that
\begin{align}
\norm{G(\mu^1,\theta^1)-G(\mu^2,\theta^2)}_{\MC([0,T];(\MC_b^1(\RR^{2d}))')}\notag 
& \leq C\norm{\mu^1-\mu^2}_{\mc{C}^1([0,T];\MC_b^1(\RR^{2d}))} \\
& \hspace{0.45cm} +C(R_T,T)\norm{\theta_1-\theta_2}_{\MC([0,T])},
\end{align} 
which implies that $G\in \MC(E;Y)$. Similarly we have
\begin{align*}
& |J(\mu^1,\theta^1)-J(\mu^2,\theta^2)| \\
& \leq \left|\int_{\RR^{2d}}\ell(x,y)\rd(\mu_T^1-\mu_T^2)(x,y)+\int_0^T(|\theta_t^1|^2-|\theta_t^2|^2)\dt\right|  \\
& \leq C\norm{\mu_T^1-\mu_T^2}_{(\MC_b^1)'} +C \big( T,\norm{\theta_1}_{\MC([0,T]},\norm{\theta_2}_{\MC([0,T])} \big) \norm{\theta_1-\theta_2}_{\MC([0,T]},
\end{align*}
where we used the fact that $\mu_T^1$ and $\mu_T^2$ are compactly supported. This in turn implies that $J\in \MC(E;\RR)$. 

Next, we use Lemma \ref{lmC1} to prove that both mappings are in fact $\MC^1$-smooth. It follows from the definition \eqref{Gdiff} of G-derivative that for all $\mu\in V$, $\nu\in U_\mu=U_V$ and $\varphi\in\MC_b^1(\RR^{2d})$, one has
\begin{align}
\la \rd_{\mu}G(\mu,\theta)(\nu),\varphi\ra & = \bigg\la \lim\limits_{\varepsilon\rightarrow 0^+} \frac{G(\mu+\varepsilon\nu,\theta)-G(\mu,\theta)}{\varepsilon}, \varphi\bigg\ra \\
&  =\lim\limits_{\varepsilon\rightarrow 0^+}\frac{\la G(\mu+\varepsilon\nu,\theta),\varphi\ra-\la G(\mu,\theta),\varphi\ra}{\varepsilon}\nn\\
& = \la -\partial_t\nu-\nabla_x\cdot (\mathcal F(t,x,\theta)\nu),\varphi\ra<+\infty.
\end{align}
Thus we have found a continuous operator $\mu \in V \mapsto L_{\theta}(\mu)\in \mc{L}(U_V,Y)$ such that $L_{\theta}(\mu)(\nu):=-\partial_t\nu-\nabla_x\cdot (\mathcal F(t,x,\theta)\nu)=\rd_{\mu}G(\mu,\theta)(\nu)$  for all $\mu\in V$ and $\nu\in U_\mu$.
 Applying Lemma \ref{lmC1} allows us to conclude that $L_{\theta}(\mu)\in D_\mu G(\mu,\theta)$  and $G(\cdot,\theta)\in \MC^1(V;Y)$. Additionally, remark that the standard Fr\'{e}chet differential $G'_{\theta}(\mu,\theta):Q\to Y$ with respect to the control curve satisfies
\begin{equation}
\la G'_{\theta}(\mu,\theta)(\alpha),\varphi\ra=\lim\limits_{\varepsilon\rightarrow 0^+} \frac{\la G(\mu,\theta+\varepsilon \alpha),\varphi\ra-\la G(\mu,\theta),\varphi\ra }{\varepsilon}=\la -\nabla_x\cdot (\nabla_\theta \mathcal F(t,x,\theta)\alpha\mu),\varphi\ra<+\infty \,.
\end{equation}
for all $\alpha\in Q$. The continuity of $\theta \in \RR^m \mapsto \nabla_\theta \MF(t,x,\theta) \in \RR^d$ implies that $G(\mu,\cdot)\in \MC^1(Q;Y)$ for every $\mu \in V$, and thus $G\in \MC^1(E;Y)$. Similarly, we have 
\begin{equation}
J'_\mu(\mu,\theta)(\nu)=\int_{\RR^{2d}}\ell(x,y)\rd\nu_T \qquad \mbox{ and } \qquad J'_\theta(\mu,\theta)(\alpha)=\int_0^T2\lambda\theta_t\cdot\alpha_t\dt\,,
\end{equation}
for all $\nu\in U_\mu=U_V$ and $\alpha\in Q$. It is then easy to check that $J\in \MC^1(E;\RR)$.


\subsubsection{The mean-field PMP for continuous controls: a Lagrangian approach}

We are now ready to present the derivation of the first order optimality condition \eqref{eq1}-\eqref{eq3} in the class of continuous controls, by means of a Lagrange multiplier rule tailored to the calculus for convex functions introduced in Section \ref{subsection:Convex}.

\begin{thm}[Abstract Lagrange multiplier theorem]
\label{thmcondi}
Let $(\mu^*,\theta^*)\in E\subset X=U\times Q$ be a solution to the optimal control problem \eqref{Pcost}. Then there exists $p^\ast\in Y'$ such that
\begin{empheq}[left=\empheqlbrace]{align}
&\la G'_\mu(\mu^\ast,\theta^\ast)(\nu),p^*\ra+ J'_\mu(\mu^\ast,\theta^\ast)(\nu)=0,\quad \mbox{ for all }\nu \in \overline U_{V}, \label{PMP2}\\
&\la G'_\theta(\mu^\ast,\theta^\ast)(\alpha),p^*\ra+ J'_\theta(\mu^\ast,\theta^\ast)(\alpha)=0,\quad \mbox{ for all }\alpha \in Q\label{PMP3}\,.
\end{empheq}
\end{thm}

\begin{rmk}The solution $\psi^*=p^*\in\mc{C}^1([0,T];\MC_c^2(\RR^{2d}))$ constructed in Proposition \ref{prop2} is in $Y'$. This comes from the fact that, for any $\eta\in Y\subset \MC([0,T];(\MC_b^1(\RR^{2d}))')$, one has $\la p^*,\eta\ra<+\infty$. 
\end{rmk}

\begin{proof}
In order to prove our set of optimality conditions, we will use Theorem \ref{thmla} which application has already been prepared above. Indeed we have shown that both the cost and constraint functionals are continuously $F$-differentiable, and it follows directly from the definitions \eqref{DefE} and \eqref{DefXE} that $(\mu^*,\theta^*)+X_E\subset E$. Thus, there remains to  prove that the linear operator $G'(\mu^\ast,\theta^\ast):  \overline {X}_E=\overline{U_{V}}\times Q\rightarrow Y$ is surjective. We split  the proof of the surjectivity  into two steps below.

\vspace{-0.3cm}
	
\paragraph*{$\bullet$ Surjectivity of the partial derivative $G'_\mu(\mu^\ast,\theta^\ast):  \overline{U_{V}}\rightarrow Y$.}  We first want to show that for any given element
\begin{equation*}
\eta\in Y :=\overline{\tilde{\mc{C}}([0,T];(\MC_b^1(\RR^{2d}))'_{0,c})}, 
\end{equation*}
there exists a $\nu\in  \overline{U}_V$ such that
\begin{equation}
G'_\mu(\mu^\ast,\theta^\ast)(\nu)=\eta\,,
\end{equation}
which is understood in the sense of 
\begin{equation}
\la G'_\mu(\mu_t^\ast,\theta_t^\ast)(\nu_t), \varphi\ra =\la\eta_t, \varphi\ra \qquad \mbox{ for all } \varphi\in \MC_b^1(\RR^{2d})\,.
\end{equation}
To this end, it suffices to show that for a given $(\mu^\ast,\theta^\ast,\eta)\in V\times Q\times Y$, there exists some $\nu\in  \overline{U}_{V}$ solution of the following  transport equation 
\begin{equation}
\label{eqsource}
\partial_t\nu_t+\nabla_x\cdot(\mathcal F(t,x,\theta^\ast_t)\nu_t)=-\eta_t\,,
\end{equation}
with source term $(-\eta)$ and initial condition $\nu_0\in U_{\mu_0}$. Notice that $(\MC_b(\RR^{2d}))'_{0,c}$ is dense in $(\mc{C}_b^1(\RR^{2d}))'_{0,c}$, namely for any $\eta\in Y=\overline{\tilde{\mc{C}}([0,T];(\MC_b^1(\RR^{2d}))'_{0,c})}$, there exists a sequence $(\eta^n)_{n \in \NN} \subset \tilde{\MC}([0,T];(\MC_b(\RR^{2d}))'_{0,c})$ such that for all $\varphi\in  \MC_b^1(\RR^{2d})$, it holds
\begin{equation}\label {eqeta}
\sup\limits_{t\in[0,T]}|\la\eta_t^n-\eta_t,\varphi\ra| ~\underset{n \to +\infty}{\longrightarrow}~ 0.
\end{equation}
In particular, observe that $\sup\limits_{t\in[0,T],n \in \NN}\|\eta_t^n\|_{(\MC_b^1)'}<+\infty $ is uniformly bounded.

Since $\eta_t^n \in (\MC_b(\RR^{2d}))'_{0,c}\subset(\MC_0(\RR^{2d}))_{0,c}'=\mc{M}_{0,c}(\RR^{2d})$, it then follows from \cite[Theorem 1]{piccoli2019wasserstein} that there exists a unique measure solution $\mu^{1,n}\in V$ to the following transport equation 
\begin{equation}
\label{eq:TransportEqSource}
\partial_t\mu_t^{1,n}+\nabla_x\cdot(\mathcal F(t,x,\theta^\ast_t)\mu_t^{1,n})=-\eta_t^n,\qquad \mu_t^{1,n}|_{t=0}=\mu_0^1\in\mc{P}_c(\RR^{2d})\, ,
\end{equation}
understood analogously to \eqref{eqweak1} in the sense of distribution, namely 
\begin{align*}
&\int_{\RR^{2d}} \varphi(x,y) \rd\mu_{t_2}^{1,n}(x,y) - \int_{\RR^{2d}}\varphi(x,y) \rd\mu_{t_1}^{1,n}(x,y) \\
& =\int_{t_1}^{t_2}\int_{\RR^{2d}} \nabla_x \varphi(x,y) \cdot \mathcal F(s,x,\theta_s^*) \, \rd\mu_s^{1,n}(x,y) \ds -\int_{t_1}^{t_2}\int_{\RR^{2d}}\varphi(x,y) \, \rd\eta_s^n(x,y)\ds
\end{align*}
for all $\varphi\in \MC_b^1(\RR^{2d})$ and every $t_1,t_2 \in [0,T]$. Indeed, we can build a solution to above as a limit of a sequence of approximated solutions satisfying the following Euler-explicit-type splitting scheme. Fix $k\in\NN$, and define $\Delta t=\frac{T}{2^k}$ and set $\mu_0^{1,n,(k)}=\mu_0$. Given $\mu_{i\Delta_t}^{1,n,(k)}$ for $i\in\{0,1,\cdots,2^k-1\}$, we denote by $\MF_{i\Delta t}=\mathcal F(i\Delta t, x,\theta_{i\Delta t}^*)$ and set
\begin{equation}\label{iteration}
\mu_t^{1,n,(k)}=\Gamma_{t-i\Delta t}^{\MF_{i\Delta t}}\sharp \mu_{i\Delta t}^{1,n,(k)}-(t-i\Delta t)\eta_{i\Delta t}^n,\quad t\in[i\Delta t,(i+1)\Delta t]\,,
\end{equation}
where $\Gamma_{t-i\Delta t}^{\MF_{i\Delta t}}\sharp \mu_{i\Delta t}^{1,n,(k)}$ is the unique solution of the linear transport equation
\begin{align}
\label{linear}
\begin{cases}
&\partial_tf_t+\nabla\cdot (\MF_{i\Delta t}f_t)=0,\qquad t\in (i\Delta t,(i+1)\Delta t],\\
&f_{i\Delta t}=\mu_{i\Delta t}^{1,n,(k)}\, ,
\end{cases}
\end{align}
which is is explicitly written as a pushforward through a characteristic flow. From \eqref{iteration}, we know the sequence $(\mu_t^{1,n,(k)})_{k\in\NN}$ has uniformly bounded support, since
\begin{equation}
\supp(\mu_t^{1,n,(k)}) \subset B(R_T)\cup S_{\eta^n}
\end{equation}
where $\supp(\eta_t^n) \subset S_{\eta^n}\Subset \RR^{2d}$ for all $t\in[0,T]$ and we denoted by $B(R_T)$ the support of solutions to the linear transport equation obtained in \eqref{suppt}. Intuitively, the support of $\mu_t^{1,n,(k)}$ is the union of the support of the solution to the linear transport equation \eqref{linear} and the support of the source term. Similarly, it holds for $t\in [i\Delta t, (i+1)\Delta t]$
\begin{align}
\|\mu_t^{1,n,(k)}\|_{(\MC_b^1)'}\leq \|\Gamma_{t-i\Delta t}^{\MF_{i\Delta t}}\sharp \mu_{i\Delta t}^{1,n,(k)}\|_{(\MC_b^1)'}+\Delta t\|\eta_{i\Delta t}^n\|_{(\MC_b^1)'}\leq \| \mu_{i\Delta t}^{1,n,(k)}\|_{(\MC_b^1)'}+\Delta t\|\eta_{i\Delta t}^n\|_{(\MC_b^1)'}\,.
\end{align}
This provides us with the following upper-bound 
\begin{equation}
\sup_{t\in[0,T]}\|\mu_t^{1,n,(k)}\|_{(\MC_b^1)'}\leq \| \mu_{0}^1\|_{(\MC_b^1)'}+T\sup_{t\in[0,T]}\|\eta_{t}^n\|_{(\MC_b^1)'}<+\infty\,,
\end{equation}
which is uniform with respect to $n,k \in \NN$. By letting $k \to +\infty$, we recover the existence of a solution $\mu^{1,n}$ to \eqref{eq:TransportEqSource} such that 
\begin{equation}
\sup_{t \in [0,T]} \mathbb{W}_1^{1,1}(\mu^{1,n},\mu_t^{1,n,(k)}) ~\underset{k \to +\infty}{\longrightarrow}~ 0.
\end{equation}
Recall that the generalized Wasserstein metric introduced in \cite{piccoli2019wasserstein} is equivalent to the bounded-Lipschitz norm $\| \cdot\|_{BL}$, so that the limit curves $(\mu^{1,n})_{n \in \NN}$ satisfy
\begin{equation}
\supp(\mu_t^{1,n}) \subset B(R_T)\cup S_{\eta^n} \quad \text{and} \quad \|\mu_t^{1,n}\|_{(\MC_b^1)'}<+\infty
\end{equation}
for all $t\in[0,T]$. This in turn implies that the sequence $(\mu_t^{1,n})_{n\in \NN}$ is uniformly equi-bounded in $\MC([0,T];(\MC_b^1(\RR^{2d}))')$. According to \cite[Theorem 1]{piccoli2019wasserstein}, it follows that each curve $t \in [0,T] \mapsto \mu^{1,n}$ is Lipschitz continuous with respect to the $\|\cdot\|_{BL}$-norm, and thus it is uniformly equi-continuous with respect to the  $(\MC_b^1)'$-norm. By a direct application of the Arzel\`a-Ascoli theorem, there exists a subsequence of $(\mu^{1,n})_{n\in \NN}$ that converges uniformly in $\MC([0,T];(\MC_b^1(\RR^{2d}))')$ to some curve $\mu^1$, which then satisfies
\begin{align}
&\int_{\RR^{2d}}\varphi(x,y) \rd\mu_{t_2}^1(x,y) - \int_{\RR^{2d}}\varphi(x,y) \rd\mu_{t_1}^{1}(x,y) \\
& = \int_{t_1}^{t_2}\int_{\RR^{2d}} \nabla_x \varphi(x,y) \cdot \mathcal F(s,x,\theta_s^*)\rd\mu_s^{1}(x,y)\ds - \int_{t_1}^{t_2}\int_{\RR^{2d}} \varphi(x,y) \rd\eta_s(x,y)\ds\,.
\end{align}
However, recall now that the optimal curve $\mu^*\in V$ satisfies
\begin{equation}
\partial_t\mu_t^*+\nabla_x\cdot(\mathcal F(t,x,\theta^\ast_t)\mu_t^*)=0,\qquad \mu_t^*|_{t=0}=\mu_0\in \mc{P}_c(\RR^{2d})\, ,
\end{equation}
Then, defining the curves $(\mu^{1,n}-\mu^*)_{n \in \NN} \subset U_V$ and letting $n \to +\infty$, we can find a solution
\begin{equation*}
\nu := \mu^1-\mu^* =\lim\limits_{n\to\infty}(\mu^{1,n}-\mu^*)\in\overline{U}_V, 
\end{equation*}
to the transport equation with source term \eqref{eqsource}, with the initial datum $\nu_0=\mu_0^1-\mu_0\in U_{\mu_0}$. This completes the proof of the surjectivity of $G'_\mu(\mu^\ast,\theta^\ast)$.

\medskip

\paragraph*{$\bullet$ Surjectivity of the full derivative $G'(\mu^\ast,\theta^\ast):  \overline{X}_E = \overline{U}_V \times Q \rightarrow Y$.}  Assume that $\nu\in \overline{U}_V$ is a curve obtained as above. Then for any $\eta\in Y$, there exists $(\nu,0)\in \overline{U}_V \times Q$ such that 
\begin{equation}
G'(\mu^\ast,\theta^\ast)(\nu,0)=G'_{\mu}(\mu^\ast,\theta^\ast)(\nu)+G'_{\theta}(\mu^\ast,\theta^\ast)(0)=\eta\,.
\end{equation}
Thus, we have proven that $G'(\mu^\ast,\theta^\ast)$ is surjective.
\end{proof}


\subsubsection{The mean-field PMP for measurable controls: an Hamiltonian approach}
\label{sec:Hamil}

The goal of this subsection is to show that solutions $(\mu^*,\theta^*) \in \mc{C}([0,T];\mc{P}_c(\R^d)) \times L^2([0,T];\R^m)$ the optimality condition \eqref{eq1}-\eqref{eq3} by using the Pontryagin Maximum Principle  in Wasserstein spaces studied in \cite{PMPWassConst,PMPWass,SetValuedPMP}.

In the sequel, we suppose that the optimal control problem \eqref{Pcost} admits an optimal trajectory-control pair $(\mu^*,\theta^*) \in \Lip([0,T];\mc{P}_c(\R^{2d}) \times L^2([0,T];\R^m)$. The \textit{Hamiltonian} function $\mathbb{H} : [0,T] \times \Pcal_c(\R^{4d}) \times L^2([0,T];\R^m) \rightarrow \R$ associated with the optimal control problem is defined by 
\begin{equation}
\label{eq:Hamiltonian}
\mathbb{H}(t,\nu,\theta) := \INTDom{\langle r , \MF(t,x,\theta) \rangle}{\R^{4d}}{\nu(x,y,r,s)} - \lambda |\theta|^2, 
\end{equation}
for almost every $t \in [0,T]$ and all $(\nu,\theta) \in \Pcal_c(\R^{4d}) \times \R^m$, and we denote by
\begin{equation*}
\J_{4d} := \begin{pmatrix}
0 && \operatorname{Id} \\ -\operatorname{Id} && 0
\end{pmatrix},
\end{equation*}
the standard symplectic matrix of $\R^{4d}$. In this context, the PMP of \cite{SetValuedPMP} was adapted to unbounded control sets in \cite{SemiconcavityCDC}, and can be written in context as follows.

\begin{thm}[Pontryagin Maximum Principle]
\label{thm:PMP}
There exists a radius  $R_T' > 0$ and a uniquely determined state-costate curve $\nu^* \in \Lip([0,T],\Pcal_c(\R^{4d}))$ with $\supp(\nu^*_t) \subset B(R_T') \times B(R_T')$ for all times $t \in [0,T]$, such that the following holds. 
\begin{enumerate}
\item[$(i)$] The curve $\nu^*$ solves the \emph{forward-backward} Hamiltonian continuity equation
\begin{equation}
\label{eq:Theorem_Dynamics}
\left\{
\begin{aligned}
& \partial_t \nu^*_t + \nabla_{(x,y,r,s)} \cdot \big( \J_{4d} \nabla_{\nu} \mathbb{H}(t,\nu^*_t,\theta^*_t) \nu^*_t \big) = 0, \\
& \pi^1_{\#} \nu^*_t = \mu^*_t \hspace{1.65cm} \text{for all times $t \in [0,T]$}, \\
& \nu^*_T = (\operatorname{Id},-\nabla_x \ell) \sharp \mu^*_T, 
\end{aligned}
\right.
\end{equation}
where the \emph{Wasserstein gradient} of the Hamiltonian is given explicitly by 
\begin{equation*}
\nabla_{\nu} \mathbb{H}(t,\nu^*_t,\theta^*_t)(x,y,r,s) = \begin{pmatrix}
\nabla_x \MF(t,x,\theta^*_t)^{\top} r \\ 0 \\ \MF(t,x,\theta^*_t) \\ 0
\end{pmatrix}, 
\end{equation*}
for almost every $t \in [0,T]$ and all $(x,y,r,s) \in B(R_T') \times B(R_T')$.
\item[$(ii)$] The \emph{maximization condition}
\begin{equation}
\label{eq:Theorem_Maximisation}
\mathbb{H}(t,\nu^*_t,\theta^*_t) = \max_{\theta \in \R^m} \, \mathbb{H}(t,\nu^*_t,\theta), 
\end{equation}
holds for almost every $t \in [0,T]$. 
\end{enumerate} 
\end{thm}

Below, we provide a representation formula for the state-costate curve $\nu^*$, based on the \emph{disintegration theorem}  (see e.g. \cite[Theorem 5.3.1]{AGS}). The sufficient implication of this statement was used as early as \cite{PMPWass} to build solutions to \eqref{eq:Theorem_Dynamics}, while the necessary part has been established more recently in \cite{SemiSensitivity}. Following the notations of Section \ref{sec:Existence} and Appendix \ref{subsection:Flows}, we denote by $(\BPhi_{(\tau,t)}^*)_{\tau,t \in [0,T]}$ the characteristic flows such that $\mu^*_t = \BPhi_{(0,t)}^* \sharp \mu_0$ for all times $t \in [0,T]$. Observe that by construction, it holds 
\begin{equation*}
\BPhi_{(\tau,t)}^*(x,y) = (\Phi_{(\tau,t)}^*(x),y), 
\end{equation*}
for all times $\tau,t \in [0,T]$ and every $(x,y) \in B(R_T')$, where $(\Phi_{(\tau,t)}^*)_{\tau,t \in [0,T]}$ is the characteristic flow defined via \eqref{eq:ThetaFlow} with $\theta_t := \theta_t^*$ being the optimal control. 

\begin{proposition}[Representation formula for state-costate curves]
\label{prop:Costates}
A state-costate curve $\nu^* \in \Lip([0,T],\Pcal_c(\R^{4d}))$ solves the forward-backward system \eqref{eq:Theorem_Dynamics} \emph{if and only if} it can be represented as $\nu^*_t = (\BPhi_{(T,t)}^* \circ \pi^1,\pi^2)\sharp \nu^T_t$, where the curve $t \in [0,T] \mapsto \nu^T_t \in \Pcal_c(\R^{4d})$ is built via the disintegration formula as 
\begin{equation*}
\nu^T_t := \INTDom{\sigma^*_{t,x,y}(t)}{\R^{2d}}{\mu^*_T(x,y)}, 
\end{equation*}
for all times $t \in [0,T]$. Therein for $\mu^*_T$-almost every $(x,y) \in \R^{2d}$, the curve $t \in [0,T] \mapsto \sigma_{t,x,y}^* \in \Pcal_c(\R^{2d})$ is chosen as the unique solution of the backward adjoint dynamics 
\begin{equation*}
\left\{
\begin{aligned}
& \partial_t \sigma_{x,y}^*(t) + \nabla_{(r,s)} \cdot (\mc{W}_{x,y}(t,r) \sigma_{x,y}^*(t)) = 0, \\
& \sigma_{x,y}^*(T) = \delta_{(-\nabla_x \ell(x,y))}, 
\end{aligned}
\right.
\end{equation*}
where 
\begin{equation*}
\mc{W}_{x,y}(t,r,s) := \begin{pmatrix}
- \nabla_x \MF \big( t,\Phi_{(T,t)}^*(x),\theta^*_t \big)^{\top}r \\ 0
\end{pmatrix}, 
\end{equation*}
for almost every $t \in [0,T]$ and all $(r,s) \in B(R_T')$. 
\end{proposition}

It is easy to see that since the second marginal of $\mu^*$ is fixed, the matching part of the costate measure is also independent of time. In the following lemma, we provide a first-order characterization of the maximization condition \eqref{eq:Theorem_Maximisation}. 

\begin{lem}[Fixed-point expression for the optimal control]
Let $(\mu^*,\theta^*)$ be an optimal pair for the problem \eqref{Pcost}, and $\nu^*$ be the corresponding state-costate curve given by Theorem \ref{thm:PMP}. Then for $\lambda > 0$ large enough, it holds that
\begin{equation}
\label{eq:OptControl}
\theta^*_t  = \frac{1}{2\lambda} \INTDom{\nabla_\theta \MF(t,x,\theta^*_t)^{\top} r \, }{\R^{4d}}{\nu^*_t(x,y,r,s)},
\end{equation}
for almost every $t \in [0,T]$.
\end{lem}

\begin{proof}
As a consequence Assumptions \ref{asum1}-$(iv)$, the map $\theta \in \R^m \mapsto \Hbb(t,\nu^*_t,\theta)$ is twice differentiable for almost every $t \in [0,T]$. Moreover since $\supp(\nu^*_t) \subset B(R_T') \times B(R_T')$, there exists a constant $C(R_T') > 0$ such that 
\begin{equation*}
\sup_{\theta \in \R^m} \bigg| \nabla^2_{\theta} \INTDom{\big\langle r , \MF(t,x,\theta) \big\rangle}{\R^{4d}}{\nu^*_t(x,y,r,s)} \bigg|  \leq C(R'_T). 
\end{equation*}
Hence for $\lambda > C(R_T')$, the Hamiltonian is a concave function of $\theta$, and the optimal control $\theta^*$ satisfies the pointwise maximization condition \eqref{eq:Theorem_Maximisation} if and only if 
\begin{equation}
\nabla_{\theta} \Hbb(t,\nu^*_t,\theta^*_t) = 0 \qquad \text{for a.e. $t \in [0,T]$}, 
\end{equation}
which is equivalent to the fixed-point equation \eqref{eq:OptControl}. 
\end{proof}

For all times $t \in [0,T]$, we shall denote by $(x,y) \in B(R_T') \mapsto \bar{\sigma}^*(t,x,y) \in \R^{d}$ the $d$ first components of the \textit{barycentric projection} (see e.g. \cite[Definition 5.4.2]{AGS}) of the measures $\nu^T_t$ onto their first marginal $\pi^1_{\#} \nu^T_t = \mu^*_T$, namely 
\begin{equation*}
\bar{\sigma}^*(t,x,y) := \INTDom{r \,}{\R^{2d}}{\sigma^*_{x,y}(t)(r,s)}. 
\end{equation*}
Using this notation, one can easily check by linearity of the integral that the fixed-point equation \eqref{eq:OptControl} can be rewritten as 
\begin{equation*}
\theta^*_t = \frac{1}{2\lambda} \INTDom{\nabla_\theta \MF(t,x,\theta^*_t)^{\top} \, \bar{\sigma}^*\big(t, \Phi_ {(t,T)}^*(x),y \big) \, }{\R^{2d}}{\mu^*_t(x,y)},
\end{equation*}
for $\mu^*_T$-almost every $(x,y) \in \R^{2d}$. Our goal now is to show that $\nabla_x \psi^*(t,\Phi_{(T,t)}^*(x),y) = -\bar{\sigma}^*(t,x,y)$ for all times $t \in [0,T]$ and $\mu^*_T$-almost every $(x,y) \in \R^{2d}$, so that the adjoint variable $\psi^*(\cdot,\cdot)$ stemming from the Lagrangian method described throughout Section \ref{sec:4} satisfies
\begin{equation*}
\theta^*_t = -\frac{1}{2\lambda} \INTDom{\nabla_\theta \MF(t,x,\theta^*_t)^{\top}\nabla_{x}\psi^*(t,x,y)}{\R^{2d}}{\mu^*_t(x,y)}\,
\end{equation*}
which is exactly \eqref{P3}. This is the object of the following proposition, whose proof relies on the explicit characterization of the adjoint of the differential of a flow that we recall in the following lemma. While it is a folklore result in the theory of non-linear ODEs, its proof is provided in very few references, and we include it in Appendix \ref{subsection:Flows} for the sake of completeness.
 
\begin{lem}
\label{lem:AdjointFlow}
For every $x \in \R^d$ and $\theta \in L^2([0,T];\R^m)$, the map $t \in [0,T] \mapsto \nabla_x \Phi^{\theta}_{(t,T)} (\Phi^{\theta}_{(T,t)}(x))^{\top}$ is the unique solution of the backward adjoint Cauchy problem
\begin{equation*}
\left\{
\begin{aligned}
\partial_t w(t,x) & = - \nabla_x \MF\big( t , \Phi^{\theta}_{(T,t)}(x),\theta_t\big)^{\top} w(t,x), \\
w(T,x) & = \operatorname{Id}.
\end{aligned}
\right.
\end{equation*}
\end{lem}

\begin{proposition}[Rigorous link between the Hamiltonian and Lagrangian adjoint states]
\label{prop:LinkAdjoint}
Let $\psi^* \in \mc{C}^1([0,T];\MC_c^2(\RR^{2d}))$ be the unique characteristic solution of the formal adjoint equation \eqref{eq2} associated with an optimal pair $(\mu^*,\theta^*) \in \mc{C}([0,T];\mc{P}_c(\R^d)) \times L^2([0,T];\R^m)$. Then, it holds that 
\begin{equation*}
\INTDom{\nabla_\theta \MF(t,x,\theta^*_t)^{\top}\nabla_{x}\psi^*(t,x,y)}{\R^{2d}}{\mu^*_t(x,y)} = -\INTDom{\nabla_\theta \MF(t,x,\theta^*_t)^{\top} \bar{\sigma}^*\big(t, \Phi_ {(t,T)}^*(x),y \big) \, }{\R^{2d}}{\mu^*_t(x,y)}, 
\end{equation*}
for $\Lcal^1$-almost every $t \in [0,T]$. In particular, the triple $(\mu^*,\theta^*,\psi^*)\in \MC([0,T];\mc{P}_c(\RR^{2d}))\times \operatorname{Lip}([0,T];\RR^m)\times Y'$ satisfies the mean-field PMP \eqref{eq1}-\eqref{eq3}. 
\end{proposition}

In the following lemma, we prove that for $\mu^*_T$-almost every $(x,y)\in \R^{2d}$, the map $t \in [0,T] \mapsto \bar{\sigma}^*(t,x,y) \in \R^d$ solves the backward linearized adjoint dynamics associated with the controlled velocity field $\MF: [0,T] \times \R^d \times \R^m \rightarrow \R^d$. 
\begin{lem}
For $\mu^*_T$-almost every $(x,y)\in \R^{2d}$, the map $t \in [0,T] \mapsto \bar{\sigma}^*(t,x,y) \in \R^d$ is the unique solution of the backward Cauchy problem
\begin{equation}
\label{eq:BarycenterCauchy}
\left\{
\begin{aligned}
\partial_t \bar{\sigma}^*(t,x,y) & = - \nabla_x \MF \big( t,\Phi_{(T,t)}^*(x),\theta^*_t \big)^{\top} \, \bar{\sigma}^*(t,x,y) \\
\bar{\sigma}^*(T,x,y) &= -\nabla_x \ell(x,y).
\end{aligned}
\right.
\end{equation}
\end{lem}

\begin{proof}
By definition of the barycentric projection, it is clear from the fact that $\sigma^*_{x,y}(T) = \delta_{(-\nabla \ell(x,y))}$ that $\bar{\sigma}^*(T,x,y) = -\nabla_x \ell(x,y)$ for $\mu^*_T$-almost every $(x,y) \in \R^{2d}$. Moreover following the construction detailed in Proposition \ref{prop:Costates}, it holds for any $\xi \in \Ccal^1_c(\R^{2d})$ that 
\begin{equation}
\label{eq:SigmaEq}
\derv{}{t} \INTDom{\xi(r,s)}{\R^{2d}}{\sigma^*_{t,x,y}(r,s)} = \INTDom{\Big\langle \nabla_r \xi(r,s) , - \nabla_x \MF \big( t,\theta^*_t ,\Phi_{(T,t)}^*(x) \big)^{\top} r \Big\rangle}{\R^{2d}}{\sigma^*_{t,x,y}(r,s)}
\end{equation}
for almost every $t \in [0,T]$. We can in particular choose test functions of the form $\xi(r,s) = \zeta(r)\phi(s)$ for some $\zeta,\phi \in \Ccal^1_c(\R^{2d})$. Then given an arbitrary $h \in \R^d$, consider $\zeta,\phi$ to be smooth functions such that 
\begin{equation*}
\zeta(r) = \left\{
\begin{aligned}
& \langle h,r \rangle \hspace{0.755cm} \text{if $|r| \leq R_T'$}, \\
& 0 \qquad \text{if $|r| \geq R_T'+1$,}
\end{aligned}
\right.
\qquad \text{and} \qquad 
\phi(s) = \left\{
\begin{aligned}
& 1 \hspace{1.5cm} \text{if $|s| \leq R_T'$}, \\
& 0 \qquad \text{if $|s| \geq R_T'+1$,}
\end{aligned}
\right.
\end{equation*}
for all $(r,s) \in \R^{2d}$. It then holds that $\nabla_r \xi(r,s) = \phi(s) \nabla \zeta(r) = h$ for every $(r,s) \in B(R_T')$, which upon recalling that $\supp(\sigma_{t,x,y}^*) \subset B(R_T')$ for all times $t \in [0,T]$ yields together with \eqref{eq:SigmaEq} that
\begin{equation*}
\derv{}{t} \langle h, \bar{\sigma}^*(t,x,y) \rangle = \Big\langle h , - \nabla_x \MF \big( t,\theta^*_t ,\Phi_{(T,t)}^*(x) \big)^{\top} \bar{\sigma}^*(t,x,y) \Big\rangle, 
\end{equation*}
for almost every $t \in [0,T]$. Since $h \in \R^d$ is arbitrary, we can indeed conclude that the map $t \in [0,T] \mapsto \bar{\sigma}^*(t,x,y) \in \R^d$ is a solution of the Cauchy problem \eqref{eq:BarycenterCauchy}. The uniqueness follows from Assumption \ref{asum1} together with classical Gr\"onwall estimates.
\end{proof}

\begin{proof}[Proof of Proposition \ref{prop:LinkAdjoint}]
Following Proposition \ref{prop2}, we recall that the adjoint variable $\psi^*$ of the Lagrangian approach is defined via the method of characteristics, namely
\begin{equation*}
\psi^*(t,x,y) := \ell \big( \BPhi_{(t,T)}^*(x,y) \big) = \ell \big( \Phi_{(t,T)}^*(x),y\big),
\end{equation*}
for all $(t,x,y) \in [0,T] \times \R^{2d}$. Differentiating with respect to $x \in \R^d$ in the previous expression, we further obtain that
\begin{equation*}
\nabla_x \psi^*(t,x,y) = \nabla_x \Phi_{(t,T)}^*(x)^{\top} \nabla_x \ell \big( \Phi_{(t,T)}^*(x),y\big).
\end{equation*}
Evaluating this expression at $\Phi_{(T,t)}^*(x)$ for some $(x,y) \in \supp(\mu^*_T)$, the previous identity reads 
\begin{equation*}
\nabla_x \psi^* \big(t, \Phi^*_{(T,t)}(x),y\big) = \nabla_x \Phi_{(t,T)}^* \big(\Phi_{(T,t)}^*(x) \big)^{\top} \nabla_x \ell(x,y),
\end{equation*}
for all times $t \in [0,T]$ and $\mu^*_T$-almost every $(x,y) \in \R^{2d}$. Observe now that by Lemma \ref{lem:AdjointFlow}, the mapping $t \in [0,T] \mapsto \nabla_x \Phi_{(t,T)}^* \big(\Phi_{(T,t)}^*(x) \big)^{\top} \nabla_x \ell(x,y) \in \R^d$ is the unique solution of the backward Cauchy problem
\begin{equation*}
\left\{
\begin{aligned}
\partial_t w(t,x,y) & = - \nabla_x \MF \big( t,\Phi_{(T,t)}^*(x) ,\theta^*_t \big)^{\top} w(t,x,y), \\
w(T,x,y) &= \nabla_x \ell(x,y).
\end{aligned}
\right.
\end{equation*}
By standard Cauchy-Lipschitz uniqueness, this allows us to conclude that $\nabla_x \psi^* \big(t, \Phi^*_{(T,t)}(x),y\big)  = -\bar{\sigma}^*(t,x,y)$ for all times $t \in [0,T]$ and $\mu^*_T$-almost every $(x,y) \in \R^{2d}$, which in particular yields 
\begin{equation*}
\begin{aligned}
\theta_t^* & = \INTDom{\nabla_{\theta} \MF(t,\Phi^*_{(T,t)}(x),\theta^*_t)^{\top} \bar{\sigma}^*(t,x,y)}{\R^{2d}}{\mu^*_T(x,y)} \\
& = -\INTDom{\nabla_{\theta} \MF(t,\Phi^*_{(T,t)}(x),\theta^*_t)^{\top} \nabla_x \psi^* \big(t, \Phi^*_{(T,t)}(x),y\big)}{\R^{2d}}{\mu^*_T(x,y)}  \\
& = -\INTDom{\nabla_{\theta} \MF(t,x,\theta^*_t)^{\top} \nabla_x \psi^* \big(t,x,y\big)}{\R^{2d}}{\mu^*_t(x,y)} 
\end{aligned}
\end{equation*}
for almost every $t \in [0,T]$, and concludes the proof of our claim. 
\end{proof}

We can now conclude this section with the following summarizing result, Theorem \ref{MAINthm}.

\begin{thm}\label{thmMAIN} For any given $T>0$, let $\MF$ satisfy the Assumption \ref{asum1} and \ref{asum2}, the initial data $\mu_0\in \mc{P}_c(\RR^{2d})$, and the terminal condition $\psi_T$ satisfy \eqref{tem}.  Assume further that $\lambda>0$ is large enough. Then, an admissible control $\theta^* \in L^2([0,T],\mathbb R^m)$ fulfills the mean-field PMP \eqref{eq1}-\eqref{eq3} \emph{if and only if} it is optimal. In addition, such an optimal control $\theta^*$ is uniquely determined and Lipschitz continuous. 
\end{thm}
\begin{proof} The result follows by combining Theorem \ref{thmPMP}, Theorems \ref{thmcondi}-\ref{thm:PMP} and Proposition \ref{prop:LinkAdjoint}.
\end{proof}


\section{Numerical experiments}
\label{sec:5}

We  conclude this paper with a few instructive numerical experiments, which highlight the features of a shooting method for the mean-field maximum principle. Extensive discussions on other numerical implementations and experiments are reported in \cite{carola,10.5555/3122009.3242022,Haber_2017, pmlr-v80-li18b}. In these works, impressive results in high dimensions have been presented and discussed, while in the present work we would like to focus more simply on understanding the mechanism of the algorithm and the interplay of its different parameters. Hence, we look at insightful examples in 1D and 2D, in order to give a simple and immediate explanation of how our method can be employed for a classification task, which is a typical application of deep learning methods. While we focus on moderate dimensions, we believe that our findings are general enough to explain the functioning of the algorithm also for higher dimensional data, such as images, and we refer to the above mentioned papers for more details.

\subsection{General setting}

Shooting techniques are often used to solve deterministic optimal control problems by reducing them locally to finite dimensional equations, which are solved repeatedly for different initial values that are iteratively updated. In our case, we start with an initial random guess of the control parameter $(\theta^0_t )_{t\in [0,T]}$, we solve the optimality conditions \eqref{eq1},\eqref{eq2} and \eqref{eq3} in order to update the control parameter to $(\theta_t^1)_{t\in[0,T]}$, and then use the latter as a datum for the second iteration of the shooting method. This process, more formally written as the update policy 
\begin{equation*}
\theta^{n+1}_t=\Lambda(\theta^n_t),
\end{equation*} 
is repeated iteratively, until the convergence of the method is achieved. The operator $\Lambda$ has been introduced in the proof of  Theorem \ref{thmPMP}, where we showed that the optimal control is its unique fixed point. In particular, we proved therein that such iterations are contractive as soon as they remain bounded, and provided that the regularization parameter $\lambda > 0$ is large enough. Therefore, by construction, the convergence of the shooting scheme is automatically guaranteed in our setting for bounded iterations. Moreover, Corollary \ref{thmrate} also ensures the convergence of the empirical solutions obtained for finite samples as $N\to \infty$. Hence, the combination of the results of Theorem \ref{thmPMP} and  Corollary \ref{thmrate} provides a theoretically guaranteed convergence for the shooting method, which is summarized in Algorithm \ref{shooting}.

\begin{algorithm}[!ht]
\caption{Shooting Technique}\label{shooting}
\begin{algorithmic}[1]
\STATE Initialize the layers $\theta^0 = (\theta^0_t)_{t \in [0,T]}$
\FOR {$k$ = 0 $\ldots$ number of iterations }
\STATE Find a curve $t \in [0,T] \mapsto \mu_t^{k}$ which solves the forward equation
\begin{equation}\label{eq1'}
\partial_t\mu^{k}_t + \nabla_x\cdot \big( \mc{F}(t,x,\theta^k_t)\mu^{k}_t \big)=0,\qquad \mu^{k}_t|_{t=0}=\mu_0\, .
\end{equation}
\STATE Find a curve $t \in [0,T] \mapsto \psi^{k}_t$ which solves the backward equation
\begin{equation}\label{eq2'}
\partial_t\psi_t^{k} + \nabla_x\psi_t^{k} \cdot \mc{F}(t,x,\theta_t^k)=0,\qquad \psi^{k}_t(x,y)|_{t=T} = \abs{x - y}^2\, .
\end{equation}
\STATE Find a new set of layers $(\theta_t^{k+1})_{t \in [0,T]}$ by solving
\begin{equation}\label{eq3'}
\theta_t^{k+1} + \frac{1}{2 \lambda} \int_{\RR^{2d}} \nabla_\theta \mc{F}(t,x,\theta^{k+1}_t)^{\top}\nabla_x\psi_t^{k}(x,y)\rd \mu_t^{k}(x,y) = 0\, .
\end{equation}
\ENDFOR
\end{algorithmic}
\end{algorithm}

\paragraph{Forward Equation.}

As already mentioned in the introduction, the dynamics \eqref{eq1'} is a linear transport equation that describes the forward pass of the initial data through the network. We investigate various ways to solve such a forward equation: our first approach, very much inspired by \cite{pmlr-v80-li18b} and the deep learning task that we aim to solve, is a particle method. Given an initial distribution $\mu_0$, we sample $N$ particles and their corresponding labels and evolve them in time for $t \in (0,T]$ according to their governing ODEs
\begin{align}\label{governing_ode}
\frac{\ud X^i_t}{\dt} = \mathcal F(t,X^i_t,\theta_t), \qquad \frac{\ud Y^i_t}{\dt}= 0\,,
\end{align}
where $X^i_t \in \RR^d$ is the position of i-th sampled particle and $Y^i_t \in \RR^d$ is its label at time -- or equivalently on the layer -- $t \in [0,T]$. 

In order to demonstrate that the convergence and contractivity of the method is independent of the number of particles/samples $N$ and to highlight the power of our mean-field result, we also employ a Monte Carlo method. The idea in this case is to ``break up'' the particles trajectories by performing density estimations and resamplings at every time step. Namely, we start by sampling $N$ particles from the initial distribution $\mu_0$, let them evolve according to the governing ODE \eqref{governing_ode} during a small time, and then perform a kernel density estimation in order to compute the distribution of the evolved particles, i.e. $\mu^m_1$. The apex $m$ indicates that this process sampling-moving-estimating is repeated for a certain number of repetitions $M$ in order to obtain a result that is independent of the initial sample of particles. Then, the distribution $\mu_1$ is computed as the mean over all the repetitions $\mu_1^m$ with $m=0,...,M$. Clearly, this process needs to be repeated for every layer $t \in [0,T]$. More rigorously, the method is summarized in Algorithm \ref{montecarlo}, for a generic iteration $k$ of the shooting method.

\begin{algorithm}[!ht]
\caption{Monte Carlo Method}\label{montecarlo}
\begin{algorithmic}[1]
\FOR {$ t \in [0,T]$}
\FOR {$m = 0 \ldots M$}
\STATE Sample $N$ particles from $\mu_t$
\STATE Evolve the $N$ particles according to the ODE \eqref{governing_ode}
\STATE Use kernel density estimation to compute $\mu^m_{t+1}$
\ENDFOR
\STATE Define $\mu_{t+1} = \frac{1}{M} \sum\limits_{m=1}^M \mu^m_{t+1}$
\ENDFOR
\end{algorithmic}
\end{algorithm}

By using the Monte Carlo method, we are not only highlighting the mean-field nature of our algorithm (since we can sample as many particles as we want), but also distinguish our method from the ODE-based algorithm in \cite{pmlr-v80-li18b}. Indeed, the main difference with their approach is that we are considering a mean-field version of the maximum principle, wherein the dynamics is written in terms of PDEs rather than ODEs, and for which the Monte Carlo method is a suitable solver.  

In the spirit of the latter issue, we also solve the forward equation with a classical finite volume method, which is a well-known numerical scheme to efficiently tackle generic conservation laws in any dimension. This approach is based on a mesh partition of the domain, and on the integration of the PDE over each control volume, i.e. each element of the mesh, in order to obtain a balance equation that is then discretized. One of the fundamental issues of this context lies in the discretization of the fluxes, which have to be conservative and consistent in order to produce an efficient method. In our case, since the flux depends on the function $\mathcal{F}$, we discretize it by means of an upwind spacial scheme. The drawback of this method is that it is highly dependent on the space and time discretization steps, which are very important parameters whose role will be discussed at the end of this section.

\paragraph{Backward Equation.}

The backward equation \eqref{eq2'} is independent of the forward evolution \eqref{eq1'} and, as such, it can be solved simultaneously. Observe that \eqref{eq2'} is also a transport equation, but it is defined backward in time since a boundary condition is prescribed at the final time $t=T$. As the terminal condition is a continuous function, we decide to use finite differences in space and an explicit time-scheme to solve this latter. As it happened for the resolution of \eqref{eq1'} with finite volumes, the upwind method has been used to perform the space discretization of the velocity of the backward equation. Not only is this method suitable for transport equations, but it is also ideal in the case where the velocity $\mathcal F$ depends on both space and time, i.e. when it can change at every point of the domain. Notice that we could solve \eqref{eq2'} using a finite volume method akin to that described for the forward equation, but this may prove to be inefficient because of the oscillations of $\psi_t$ for some choices of the algorithm parameters. Hence, we chose to focus our attention on the finite difference method, which produces very good results in the low dimensional test cases considered here.

\paragraph{Parameter Update.}
Finally, we solve \eqref{eq3'} which allows to update the set of layers. Given the primal-dual solutions $(\mu_t,\psi_t)$ of equations \eqref{eq1'}-\eqref{eq2'}, we can solve \eqref{eq3'} by computing the root of the following non-linear function 
\begin{equation}
f(\theta_t) = \theta_t +\frac{1}{2\lambda} \int_{\RR^{2d}} \nabla_\theta \mathcal F(t,x,\theta_t)^\top\nabla_x\psi_t(x,y)\rd\mu_t(x,y).
\end{equation}
for each $t \in [0,T$]. Inspired by the particle method employed to solve \eqref{eq1'}, the integral with respect to $\mu_t$ can be simply computed by means of a particle approximation as $\mu_t$ is an empirical distribution in our context. Moreover, given the discrete values of $\psi_t(x,y)$ that have been computed as a by-product of the finite difference scheme used to solve the backward equation \eqref{eq2'}, the function $\psi_t$ and its gradient $\nabla_x \psi_t$ can be interpolated, e.g. using splines, in order to be able to evaluate these latter in whatever position $X^i_t$ the particles may be located at in the domain. Ultimately, the fixed point equation \eqref{eq3'} can be therefore be approximated by
\begin{equation}\label{update_discretization}
f(\theta_t) \approx \theta_t + \frac{1}{2\lambda N} \sum_{i=0}^{N} \nabla_\theta \mathcal F(t,X_i(t),\theta_t)^\top\nabla_x \psi_t(X_i(t),Y_i(t)) ,
\end{equation}
and its roots can be computed using any classical non-linear equation solver such as Newton-Raphson, Bisection, or Brent's method, depending on the particular test case at and. Notice that here, the only source of approximation is the interpolation error of $\psi_t$. 

In the case where the forward equation has been solved with a Monte Carlo method, the approximation of the integral needs to be performed many times (as for the forward equation) in order to obtain a result which is independent of the initial particle sample, with same number of repetitions $M \geq 1$. Finally, if one chooses to solve the forward equation with a finite volume method, the result $\mu_t$ for each $t \in [0,T]$ is not obtained through particle approximations, but as a function on a spatially discretized domain and, as such, it is reasonable to approximate the integral using classical numerical quadrature methods. Unfortunately, those high-accuracy methods require a fine space discretization, which involves the introduction of a spline interpolation also for the forward function $\mu_t(x)$, as it was previously done for $\psi_t$ and its gradient, which adds a new source of error on top of that arising from the interpolations of $\psi_t$ and $\mu_t$. For this reason, we also opted for particle and Monte Carlo methods to approximate the integrals, rather than using its spatial values.

\subsection{Results}\label{sec:numres}
In this section, we will show how the three optimality conditions, namely forward, backward, and parameter update (\eqref{eq1'}, \eqref{eq2'}, \eqref{eq3'} respectively) are used to solve a classification task: we are given an initial distribution $\mu_0$ of data and labels, where any point with first coordinate of positive sign is corresponding to a label vector in the corresponding orthant, while a negative orthant label vector is assigned to all those points with first coordinate of negative sign (in 1D we have one coordinate only). Then, our goal is to find the control parameter $\theta$ that moves the particles sampled from $\mu_0$ in a way such that, at the final time $T$, all the particles with positive sign first coordinate are close to the positive orthant label and the particles with negative sign first coordinate are close to the negative orthant one. This task is performed through a neural network with $L \lfloor \frac{T}{\dt} \rfloor$ layers, where $\dt$ is the time discretization step used to solve both the forward \eqref{eq1'} and the backward \eqref{eq2'} equations. We will consider the layer forward map $\mathcal F(t, x, \theta_t) = \tanh( W_t \, x + \tau_t)$ and $\theta_t = (W_t, \tau_t)$, where $W_t \in \mathbb R^{d \times d}$ and $\tau_t \in \RR^d$. However, in some of the experiments reported below, we used also a forward map without shifts $\mathcal F(t, x, \theta_t) = \tanh( W_t x)$, so that simply $\theta_t = W_t$, where $W_t \in \mathbb R^{d \times d}$. 
 The test cases for the initial distribution are the following:
\begin{itemize}
\item \textit{Bimodal Gaussian in 1D and 2D:} in the monodimensional case, the initial distribution $\mu_0$ is a bimodal Gaussian, the particles sampled from it are concentrated around the points $1$ and $-1$ and are assigned to the label $y = 2$ if they have a positive sign, or to the label $y=-2$ if they have a negative sign. Similarly, in the bidimensional case the particles are initially concentrated around $(-1,-1)$ and $(+1,+1)$, but now their labels are assigned according to the sign of their first coordinate, i.e. if $X_i(0) = (X_i^1(0), X_i^2(0))$ is the initial position of the i-th particle, then this will have label $(-2,-2)$ if $X_i^1(0) <0$ and label $(+2,+2)$ if its coordinate $X_i^1(0)$ is positive. \medskip
\item \textit{Unimodal Gaussian in 1D and 2D}: since in the previous case the initial particles are already well-separated in the respective orthant, we also perform the classification of the particles sampled from an initial unimodal Gaussian centered in the origin that have corresponding positive label $+1$ and negative label $-1$ in the monodimensional case. Similarly as before, in the bidimensional case, the particles with positive first coordinate are assigned to a positive label $(+1,+1)$ and to a negative label $(-1,-1)$ when their first coordinate is negative.
\end{itemize}

\begin{figure}[ht]
\begin{minipage}[b]{0.58\linewidth} 
\includegraphics[scale = 0.47]{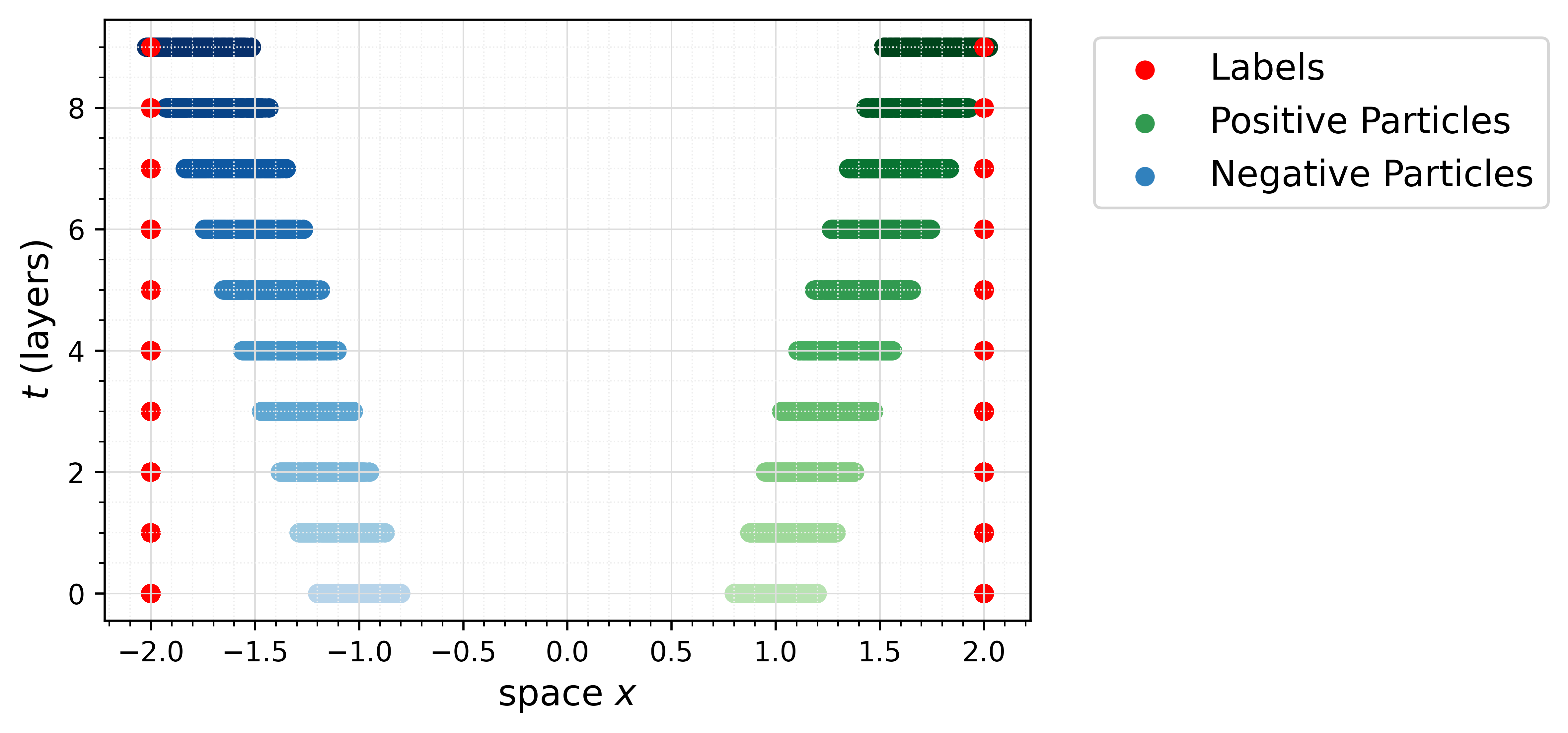}
\end{minipage}
\begin{minipage}[b]{0.58\linewidth} 
\includegraphics[scale = 0.47]{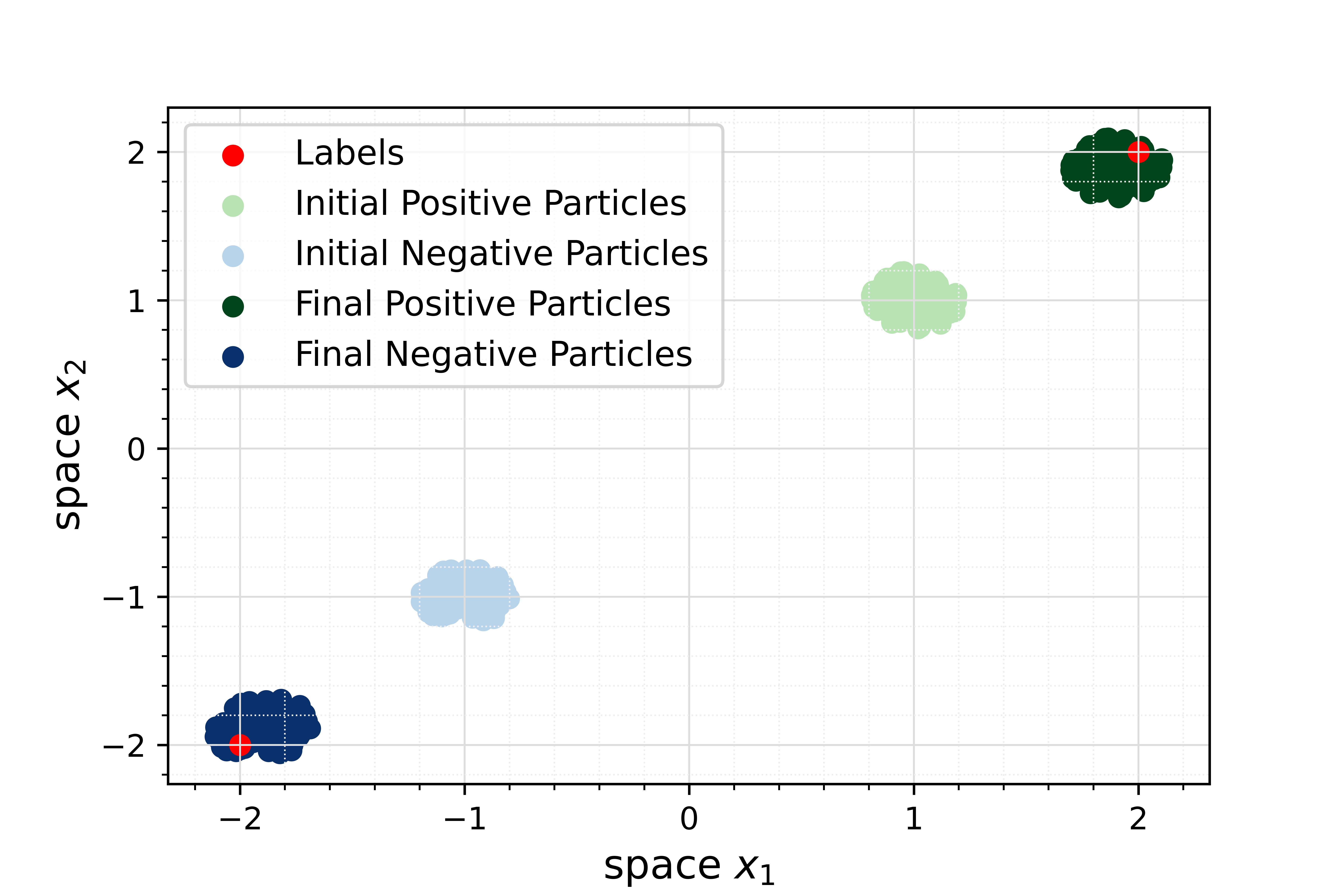}
\end{minipage}
\caption{Left: Evolution in time of the particles from the monodimensional initial bimodal distribution $\mu_0$ to $\mu_T$; Right: Plot of the initial bidimensional bimodal distribution $\mu_0$ and the final distribution $\mu_T$.}
\label{fig:bimodal_evolution}
\end{figure}

Figure \ref{fig:bimodal_evolution} shows the results obtained in the case of the bimodal distribution in 1D (on the left) and its corresponding bidimensional case (on the right). In both cases, $T = 1$ and $\dt=0.05$ which corresponds to a neural network with $L = 20$ layers, and both the layer forward maps with or without biases are used. The initial guess of $\theta^0$ is $\theta^0_t = 0$ for all $t \in [0,T]$ and the parameter $\lambda$ is set to $0.1$. The forward equation \eqref{eq1'} is solved using the particle method with $N = 200$ points, and the backward equation \eqref{eq2'} is solved in the same domain as the forward equation, namely $x \in R_T \subset \RR$ where $R_T$ is defined as in \eqref{suppt}. The $y$ variable is taken in a subset of $\RR$ as large as $R_T$ and the same space discretization in the data dimension $x$ and labels dimension $y$ is used, i.e. $dx= dy= 0.1$. The same holds for the bidimensional case, where $y \in \RR^2$ and hence the space discretization steps $dx_1 = dx_2 = dy_1 =dy_2 = 0.1$ are chosen. Finally, the root of the function in equation \eqref{update_discretization} is found using Brent's method and then the shooting method is applied for a total of $15$ (outer) iterations. \medskip

The results obtained in the case of an initial unimodal distribution in 1D and 2D are presented in Figure \ref{fig:unimodal_evolution}, respectively, left and right plots. The same parameters (namely number of layers, number of particles, space and time discretization, initial guess of $\theta^0$, and number of iterations of the shooting method) can be used in the unimodal case. The only parameter that changes is $\lambda$ which is set to $10^{-3}$ in the monodimensional case, and to $10^{-4}$ in the bidimensional one. The reason for this will be explained below when the role of $\lambda$ will be discussed. The case of unimodal Gaussian is more difficult than the bimodal one as the particles are really close to the splitting point, i.e., the origin, and it might happen that during an iteration of the shooting method some of the values of $\theta_t$ that are obtained move the particles to the other orthant, which will consequently lead these particles to be attracted to the wrong label. We notice that this behavior sometimes happens, but the particles generally learn to split nicely into two groups and move to the proper labels, as depicted in Figure \ref{fig:unimodal_evolution}. In particular, some particles appear to be a bit isolated from the others, even if they go in the direction of the labels: these are  precisely those ``confused'' particles that were first moved to the opposite orthant and then attracted to the wrong label. This is more likely to happen when the ``wrong value'' for $\lambda$ is chosen and, since it is more difficult to tune it in the bidimensional case, it is possible to see those incorrectly classified particles on the right of Figure \ref{fig:unimodal_evolution}.

\begin{figure}[ht]
\begin{minipage}[b]{0.58\linewidth} 
\includegraphics[scale = 0.47]{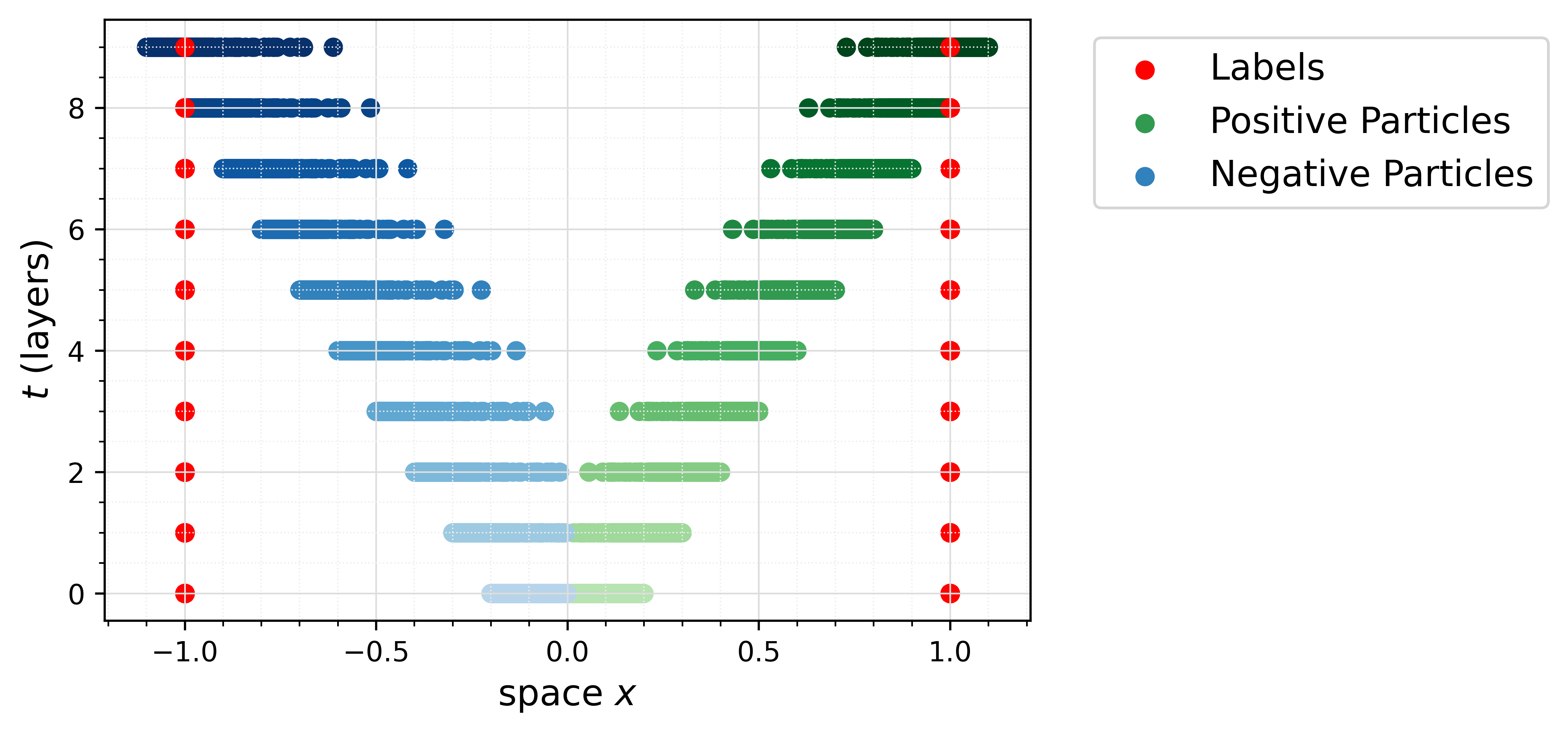}
\end{minipage}
\begin{minipage}[b]{0.58\linewidth} 
\includegraphics[scale = 0.47]{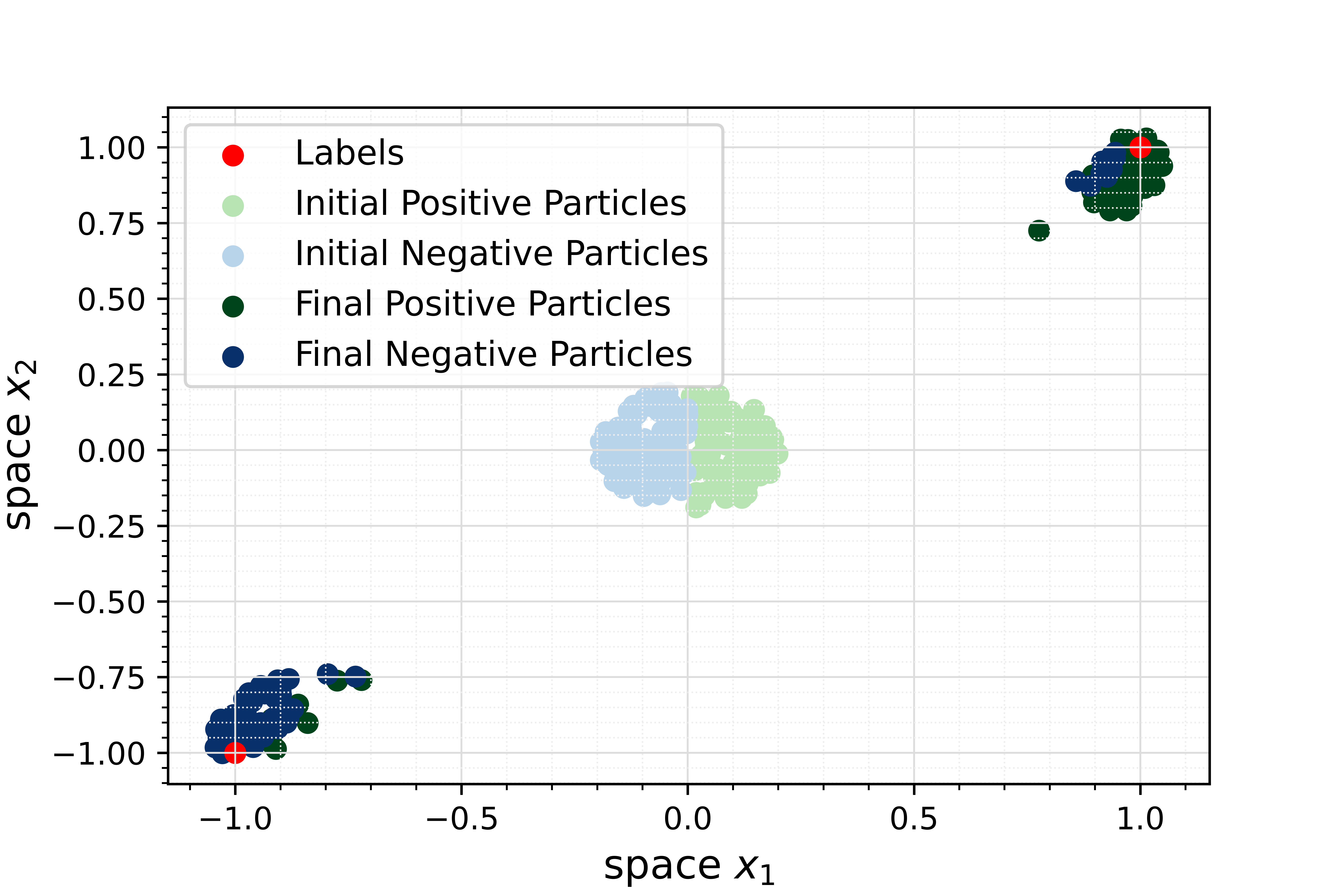}
\end{minipage}
\caption{Left: Evolution in time of the particles from the monodimensional initial unimodal distribution $\mu_0$ to $\mu_T$; Right: Plot of the initial bidimensional unimodal distribution $\mu_0$ and the final distribution $\mu_T$.}
\label{fig:unimodal_evolution}
\end{figure}

\paragraph{Comparing the resolution methods for the forward equation.}
For the monodimensional  example, it is easy to check how the various resolution methods described above perform  relative to each other in solving the forward equation. As already explained, the particle and Monte Carlo methods are more similar and based on a discrete-sampling description of the dynamics. The Monte Carlo method is more sensitive than  its particle counterpart, and needs many repetitions to  produce a result $\theta_t$  that is stable over shooting iterations. In the first row of Figure \ref{fig:resolutions_methods}, the evolution of the estimated distributions is plotted in the case of particle method, on the left, and Monte Carlo method, on the right. It is natural  to expect the Monte Carlo  scheme to be more diffusive, which stems from the high stochasticity of the algorithm. However, in both cases the final distribution is the one that we expect, i.e. both distributions are concentrated around the labels. The same happens in case of resolution with finite volume method, presented on the bottom left of Figure \ref{fig:resolutions_methods} . In this case the solution is not subject to the high stochasticity of the Monte Carlo method and hence it does not show as much diffusion, but it is not as smooth as the solution obtained with particle method. This is due to the fact that the time and space discretizations are correlated and can't be chosen freely, so a relatively big space discretization needs to be chosen to compare experiments with the same number of layers (i.e. time discretization). 
Moreover,  as illustrated by the plot on the bottom right of Figure \ref{fig:resolutions_methods}, the optimal control solution $\theta_t$ does not vary significantly  from an algorithm to the other. The solutions indicated in this graphics are the empirical expected values over multiple shooting iterations for every algorithm, and their standard deviation is also depicted around the lines representing the means. Clearly, the algorithm that has more variations in terms of shooting iterations is the Monte Carlo one, due to its stochasticity, while the particle method and the finite volume method are inherently sharper. 

Hence, in terms of computational speed and stability over iterations (especially in the more difficult case of unimodal initial distribution $\mu_0$), the particle method is the one that performs best, while being also the most suited one for a deep learning task, which in general implies very high-dimensional data.  That being said, the experiments  conducted using the Monte Carlo method  and the finite volume method do allow us to confirm numerically that the shooting method based on our mean-field optimality conditions converge on the space of probability measures as expected by the theory, independently of the number of particles. Indeed,  all the modelling parameters, and in particular the regularization constant $\lambda>0$, can be chosen independently of $N$. Moreover, the iteration does not need any batching of the data, as it is usually done in deep learning, that is, we can take a very large number of particles (as in the Monte Carlo scheme) or a small one (as in particle method), and in both cases our algorithm will return the optimal solution $\theta_t$ for every layer $t \in [0,T]$. 

\begin{figure}[ht]
\begin{minipage}[b]{0.5\linewidth} 
\includegraphics[scale = 0.47]{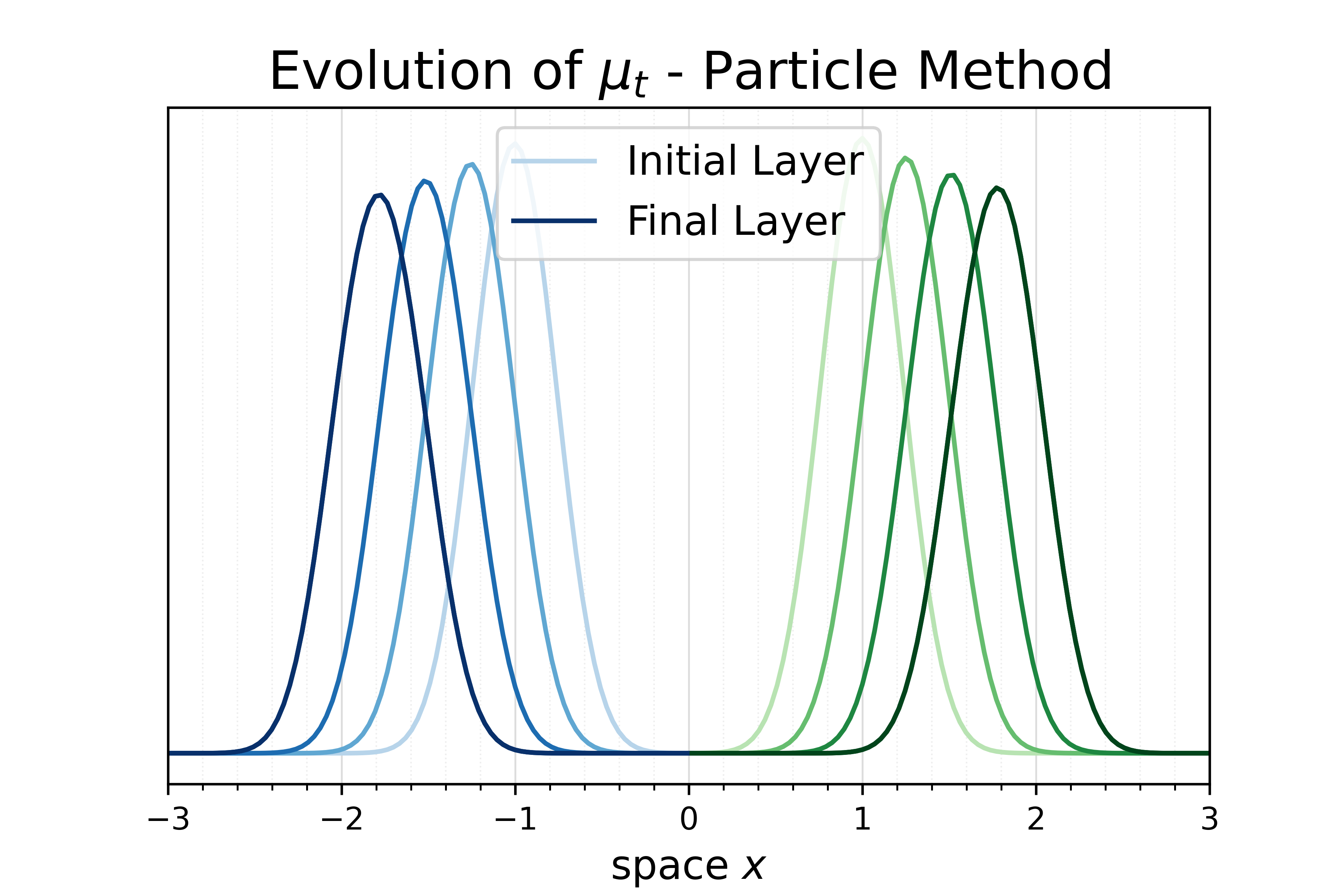}
\end{minipage}
\begin{minipage}[b]{0.5\linewidth} 
\includegraphics[scale = 0.47]{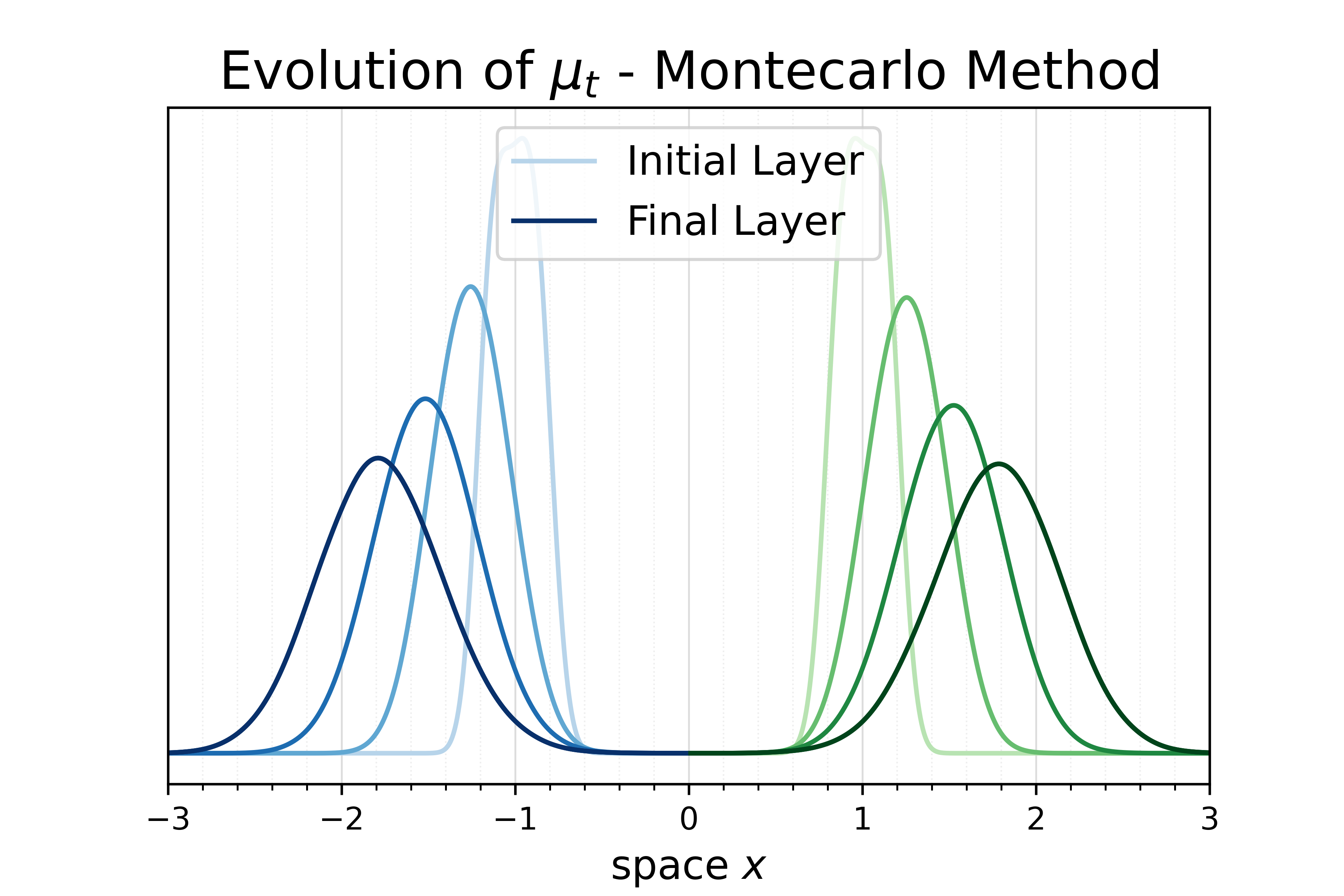}
\end{minipage}
\\
\begin{minipage}[b]{0.5\linewidth} 
\includegraphics[scale = 0.47]{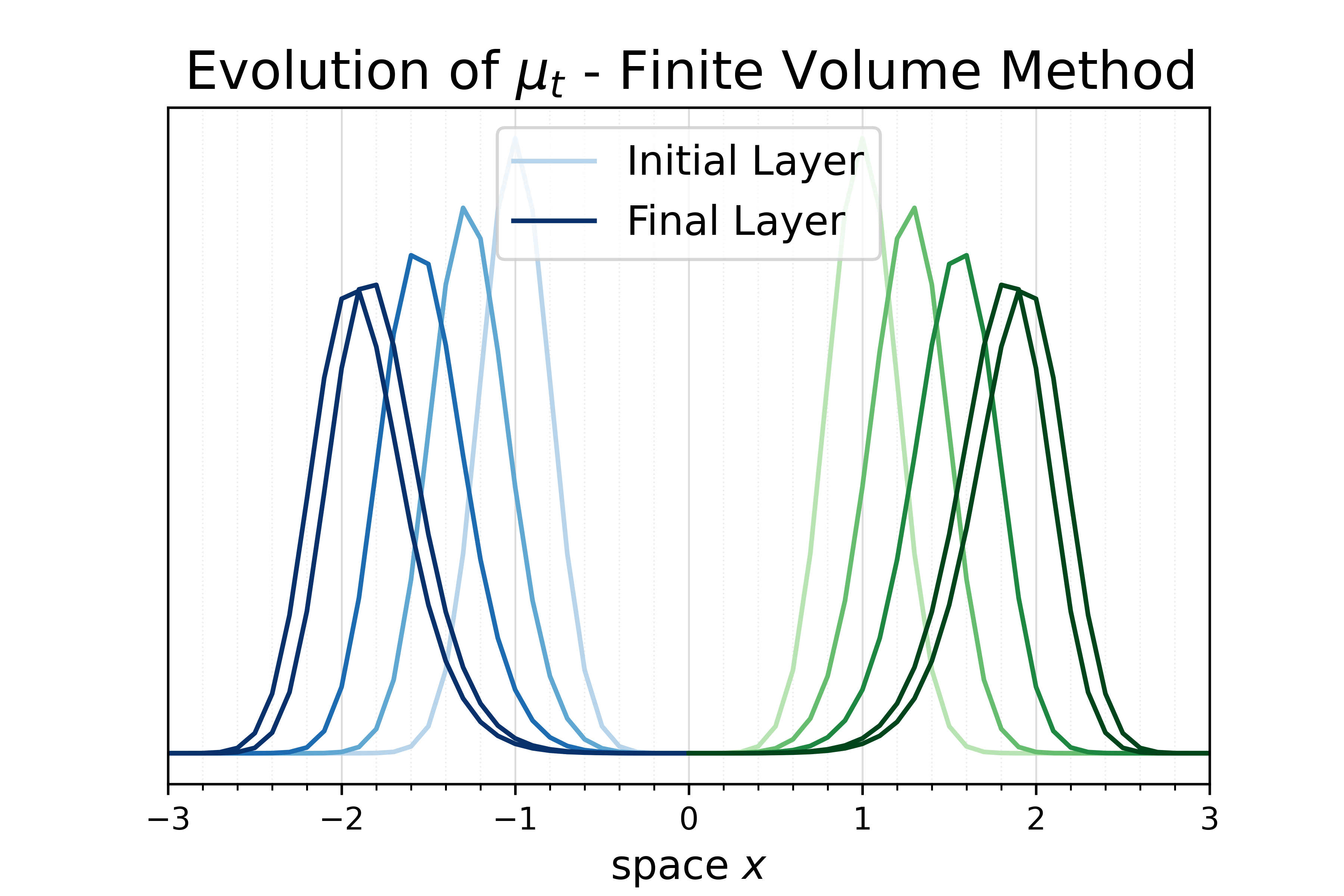}
\end{minipage}
\begin{minipage}[b]{0.5\linewidth} 
\includegraphics[scale = 0.47]{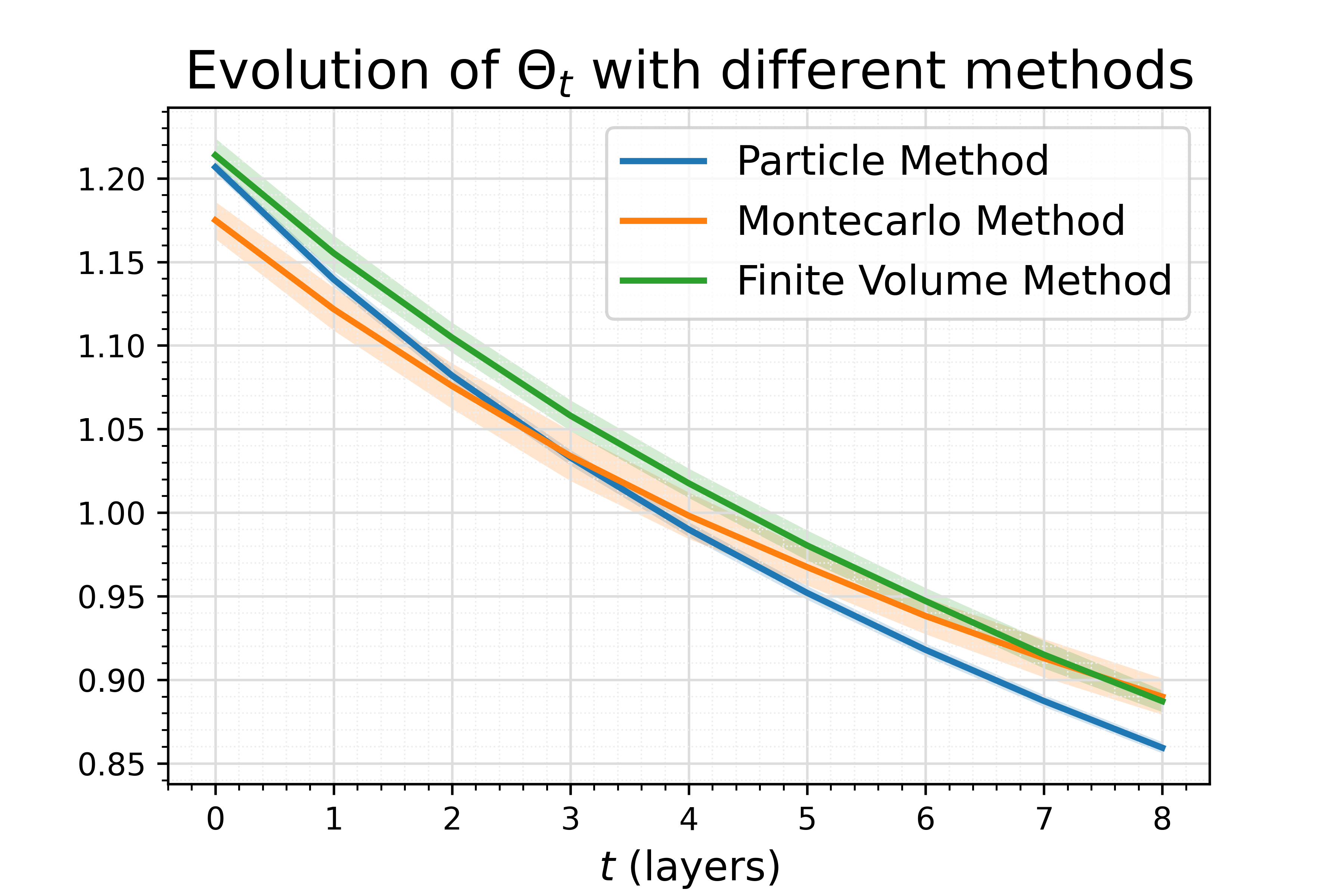}
\end{minipage}
\caption{Top Left: evolution of the estimated $\mu_t$ obtained with particle method; Top Right: evolution of the estimated $\mu_t$ obtained with Monte Carlo method; Bottom Left: evolution of the estimated $\mu_t$ obtained with finite volume method; Bottom Right: comparison of the optimal control $\theta_t$ obtained with the three different resolution methods of the forward.}
\label{fig:resolutions_methods}
\end{figure}

Let us now focus on the particle method and for that, analyze the statistical behavior of our algorithm.

\paragraph{Statistical behavior.}

The power of our mean-field maximum principle relies on the results presented in the previous paragraph regarding the independence of all parameters from the number of particles, but also on its ability to provide a strong quantitative generalization error \eqref{generr}. This means that, if we trained our network and obtained an optimal solution $\theta_t^*$, we have  the extra advantage of knowing through \eqref{generr} how well the latter will be able to perform on test data, i.e. when sampling new, unseen particles from $\mu_0$. Denoting by $J^N(\theta^*)$ the empirical error as in \eqref{PNCost}, the generalization error consists  in computing the same quantity, but with an empirical measure made by sampling new particles from $\mu_0$ that were not used for the training phase and, possibly, a significantly larger number of them. Similarly, we can define the accuracy as the empirical probability that the output of the network will be in a small ball around the corresponding label vectors, again for all the new particles that can be resampled from $\mu_0$. Figure \ref{fig:statistical} presents the expected {\it double descent curve} of the empirical and generalization error, and the corresponding increase of the accuracy. Since both generalization error and accuracy are measured on newly sampled test data,  we perform various samplings and calculations of these quantites and report in Figure \ref{fig:statistical} their mean values and their standard deviation in form of a ``cloud'' of the same color. The numerical results nicely confirm the theoretically predicted double descent phenomenon.
\begin{figure}[ht]
\centering
\includegraphics[scale = 0.47]{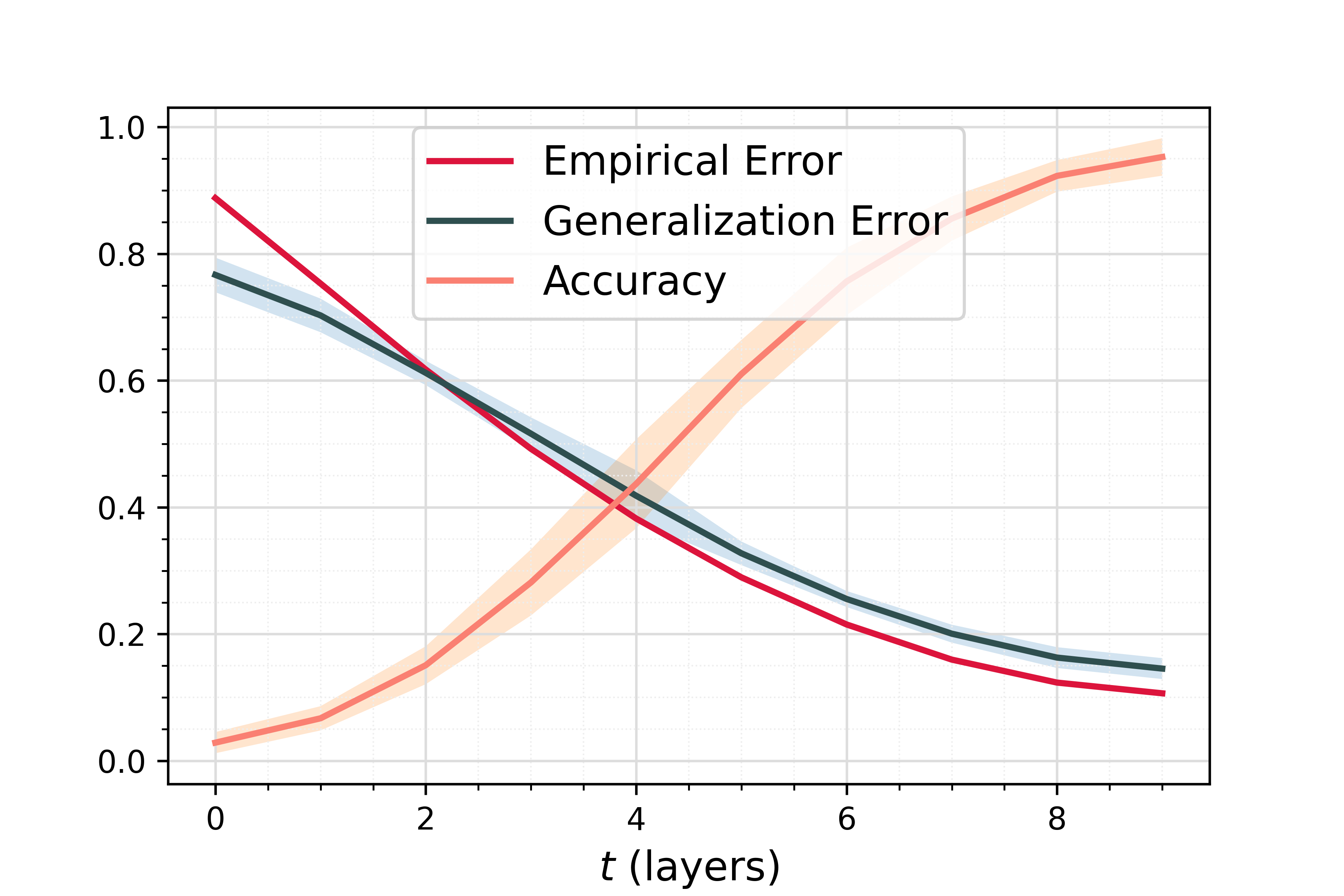}
\caption{Statistical behaviour of the algorithm resulting from the mean-field optimality conditions.}
\label{fig:statistical}
\end{figure}

Now that the resolution methods are clarified and the resulting algorithm is understood from a numerical and statistical perspective, we shift our focus towards expounding the influence of the parameters that are playing a role in our method, by first considering the number of outer iterations and then the interesting role of the regularization parameter $\lambda$ which acts as a learning rate. Finally, the necessary number of layers (i.e. the time discretization) may be examined in relation to the space discretization.

\paragraph{Contribution of the number of iterations of the shooting method.} 

In what follows, we test how many iterations of the shooting method are necessary to obtain a good result, starting first from the initial guess $\theta^0_t \equiv 0$, and then from the initial guess $\theta^0 \equiv 1$, which is closer to the optimal solution. In the case of zero initial guesses, our experiments show that after only one iteration of the shooting method, a reasonable result for $\theta$ is obtained, meaning that the parameter is constant in $t$ but manages to move the particles towards the location of the labels. At the second iteration of the shooting method, the newly learned parameter $\theta$ decreases in time and, after the third iteration, it remains stable to the values previously found, i.e. it converges to a control parameter that correctly moves the particles to the exact location of the labels. While in the case of initial guess close to the optimal solution, i.e., $\theta$ identically equal to one, already at the first iteration, the $\theta$ that is obtained decreases in time and stabilises to the appropriate values. Hence, for both cases, it is clear that it is not necessary to perform many iterations of the shooting method, even while starting from an initial guess $(\theta^0_t)_{t \in [0,T]}$ that is far away from the optimal solution. On the left of Figure \ref{fig:theta_over_time}, the $L^2$ distance between shooting method solutions, denoted by $\epsilon(k) = \big\| \theta^{k+1}_t- \theta^k_t \big\|_2$, is plotted for each $k= 0,...,$number of iterations, starting from different initial guesses $\theta^0_t$. It appears that independently of the initial guess, the distance between consecutive solutions goes to zero in a few iterations (which is also shown on the right of Figure \ref{fig:theta_over_time} where, after the second iteration, it becomes impossible to distinguish between consecutive solutions), with different velocities depending on the initial guess.

Moreover, it is interesting to notice that $\theta$ decreasing in time means that the particles at the beginning are moving faster in the direction of the labels and then when they are close enough, they slow down to precisely stop at the position of the corresponding label. The dynamics of the iterations is depicted in the plot on the right of Figure \ref{fig:theta_over_time}, in the one dimensional case where an initial bimodal Gaussian is fed to a network with layer forward map $\mathcal F(t, x, \theta_t) = \tanh( \theta_t x )$, and where the initial guess is $\theta^0 \equiv 1$.
\begin{figure}[ht]
\begin{minipage}[b]{0.46\linewidth} 
\includegraphics[scale = 0.47]{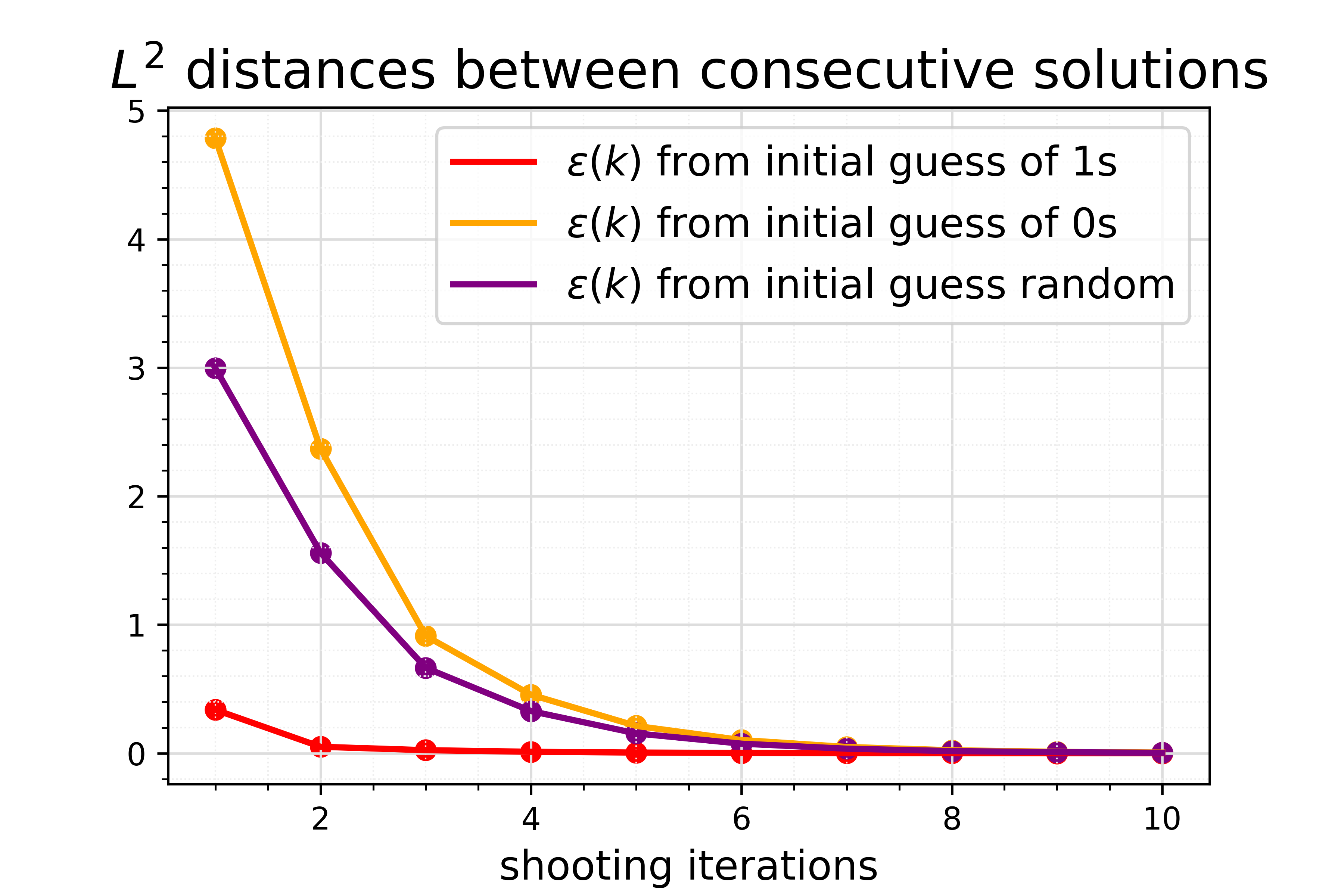}
\end{minipage}
\begin{minipage}[b]{0.46\linewidth} 
\includegraphics[scale = 0.465]{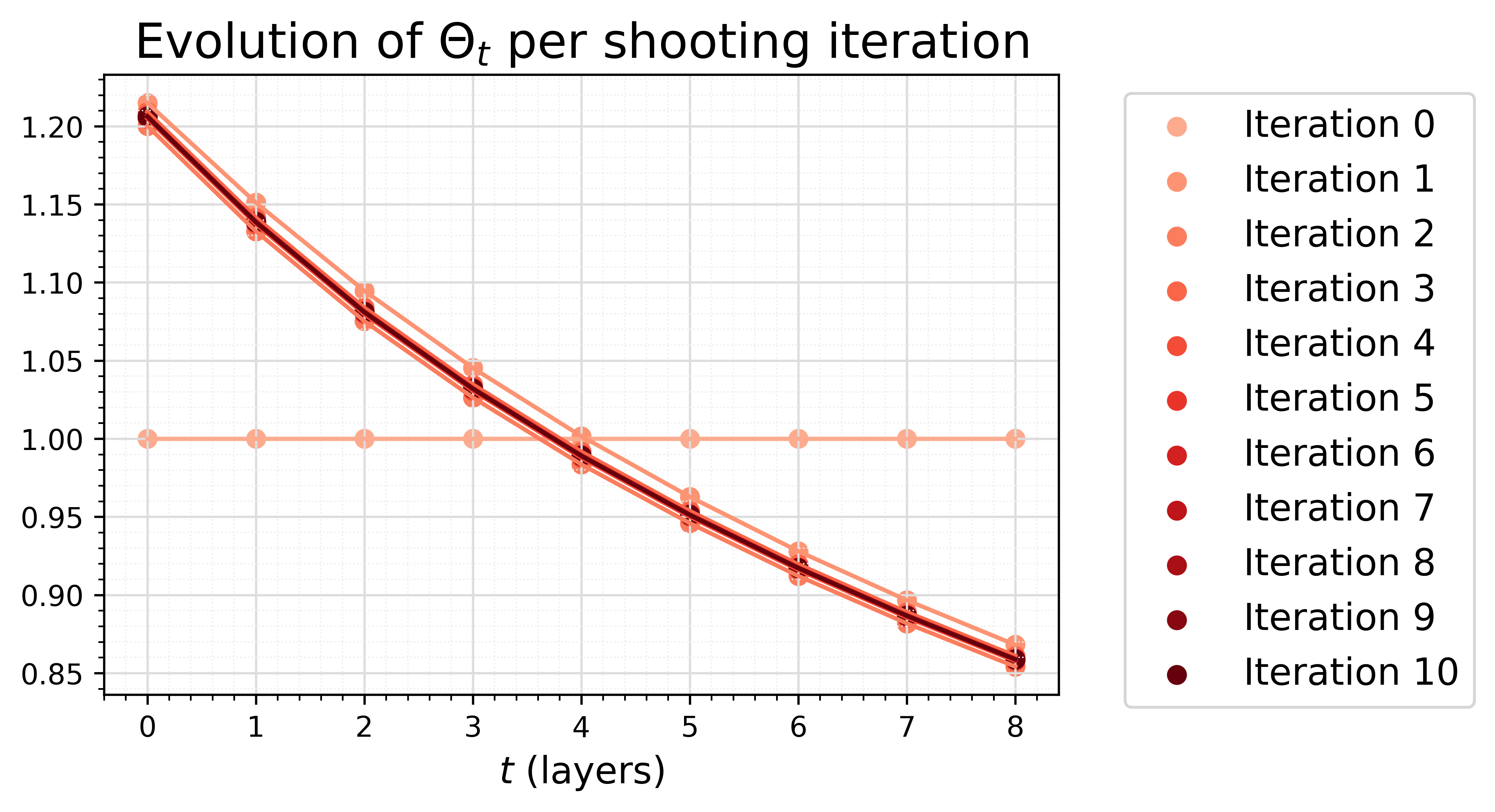}
\end{minipage}
\caption{Left: $L^2$ distance between successive solutions of the shooting method over the number of iterations, and starting from different initial guess, namely $\theta^0_t = 1$, $\theta^0_t = 0$, and $\theta^0_t = r_t$, $r_t \sim \mathcal U(0,1)$ for all $t$; Right: values of $\theta_t$ over time, starting from initial guess $\theta^0_t = 1$ for all $t$.}\label{fig:theta_over_time}
\end{figure}
\paragraph{On the effect of the regularization parameter $\lambda$.} 

A fundamental factor that has to be taken into consideration is that of the  impact of the regularization parameter $\lambda$, appearing in the fixed-point equation \eqref{eq3'} of the optimality conditions. The latter is a real positive number decided a priori, which determines the  competing influence of the regularization term in the loss function \eqref{Ecost}, and hence controls how large the $L^2$-norm of $\theta$ is allowed to be. In particular, since the layer forward map $\mathcal F$ depends on $\theta$, its norm highly influences the velocity flow of the particles in the forward equation. Hence, if the initial distribution $\mu_0$ of the particles is far away from the labels, $\lambda$ needs to be set to a small value -- e.g. smaller than $0.1$  in our examples --, to allow $\norm{\theta}_2$ to be large enough to reach the labels, otherwise the particles will not have enough speed to get to the correct location at time $T > 0$. However, always choosing a small $\lambda$ is not a good choice either, because that would destroy the convexity of the problem and lead, as we discuss below, to  numerical instabilities in the learning process. Indeed, our experiments show that small values of $\lambda$ may cause the mapping $f(\theta_t)$ defined in \eqref{update_discretization} to have many steep picks, which makes it impossible to use derivative-based methods such as Newton's algorithm to find its root. In case of exceedingly small $\lambda$, this can even lead to  functions $f(\theta_t)$ with multiple roots, which may cause the algorithm to lose stability and to oscillate between solutions, also reflecting the loss of convexity. That being said, this issue can be overcome at the price of increasing the total number of layers of the network, as evidenced by the discussion on the role of discretization parameters detailed hereinbelow.

\begin{figure}[ht]
\begin{minipage}[b]{0.5\linewidth} 
\includegraphics[scale = 0.47]{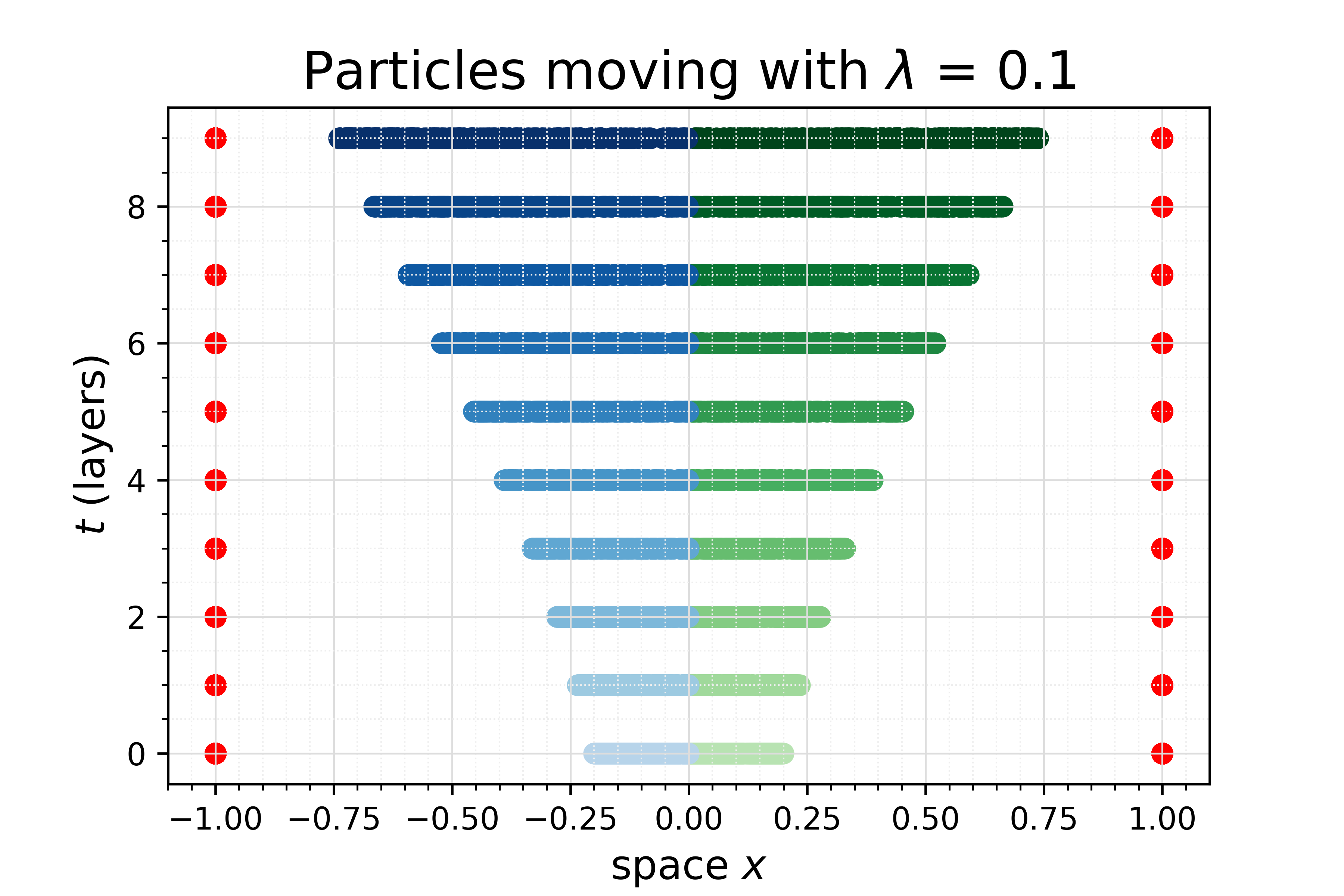}
\end{minipage}
\begin{minipage}[b]{0.5\linewidth} 
\includegraphics[scale = 0.47]{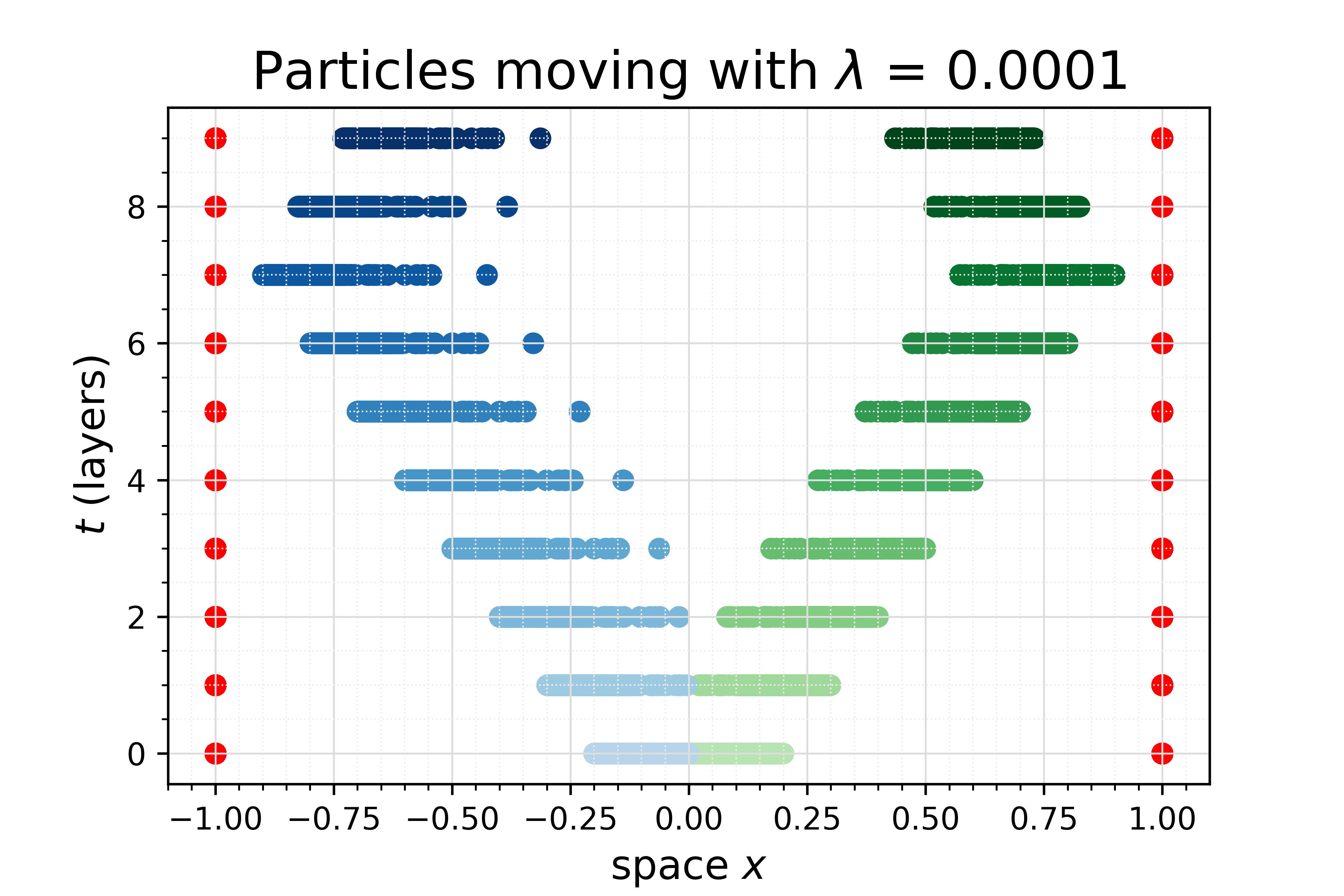}
\end{minipage}
\\
\begin{minipage}[b]{0.5\linewidth}
\hspace*{4cm}
\includegraphics[scale = 0.47]{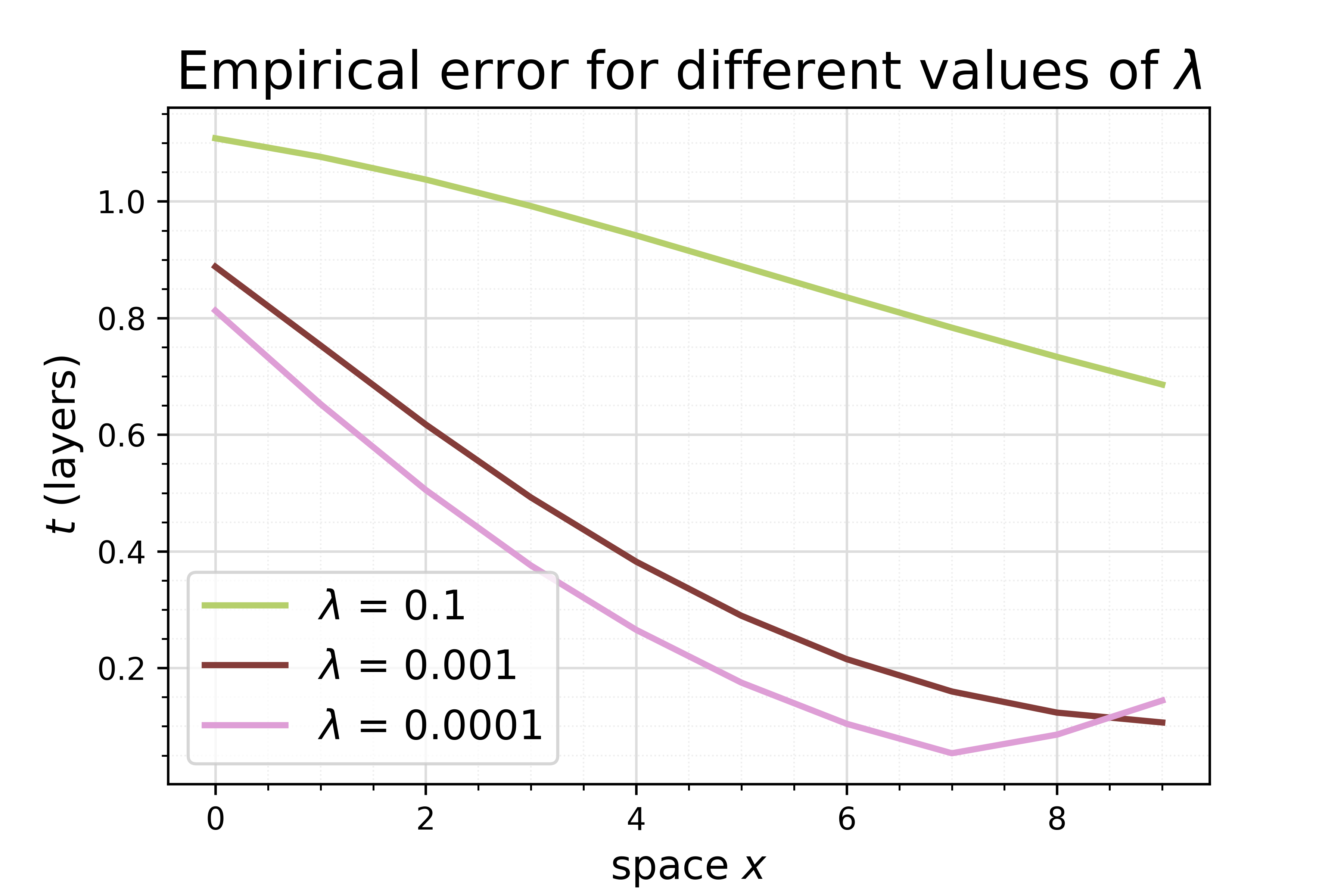}
\end{minipage}
\caption{Top Left: unimodal initial distribution, case with $\lambda=0.1$; Top Right: unimodal initial distribution, case with $\lambda=0.0001$ ; Bottom: Resulting empirical error for different values of the learning rate $\lambda$.}\label{fig:unimodal_problems}
\end{figure}

Let us now look at an instructive example, in which a unimodal monodimensional Gaussian centered  at the origin is fed to a neural network that has layer forward map without biases. In the left plot of Figure \ref{fig:unimodal_evolution}, $\lambda$ is set to $0.01$, leading to a correct solution. But in Figure \ref{fig:unimodal_problems}, we notice that if $\lambda$ is set to be too large, then $\norm{\theta}_2$ is not large enough to move the particles to the location of the labels and thus we obtain the behavior on the left of Figure \ref{fig:unimodal_problems} where the particles are moving in the correct direction but not fast enough to reach the label. On the contrary, if $\lambda$ is too small, the control $\theta_t$ obtained at every iteration of the shooting method leads to an unstable and oscillating behavior between the correct result and another solution, which is shown on the right of Figure \ref{fig:unimodal_problems}. In this case, the particles arrive too quickly to the labels, i.e., for $t < T$, due to the fact that small values of $\lambda$ allow for  large control magnitudes $\norm{\theta}_2$, which influences the velocity of the particles. At this point, the method should be able to learn a $\theta_{t+1}$ which stops the particles in order to remain at the position of the labels, but again the small value of $\lambda$ does not push easily $\theta_{t+1}$ to be zero and allows the norm of $\theta$ to remain large. As a result, the particles, instead of remaining in the location of the labels, start simply moving in the opposite direction. This behavior is not surprising as it is in accordance with Remark \eqref{remark_lambda}, for which $\lambda$ needs to be set to a large value, but the precise quantity that is needed depends on the initial distribution of $\mu_0$ and the domain $C_{\Gamma}$ in which the root can be found. Indeed, in the simpler case of a bimodal Gaussian initial distribution $\lambda$ does not have to be too small (recall that it was set to $0.1$ to produce the plot on the left of Figure \ref{fig:bimodal_evolution}), but in the more challenging case of a unimodal Gaussian initial distribution, its value has to be small enough to give the necessary velocity to the particles in order to let them split and reach the labels, e.g. $\lambda = 10^{-3}$ in the case on the left of Figure \ref{fig:unimodal_evolution}). Besides, these considerations still hold in case of activation function with biases. Indeed in this case, the parameter can be  split into two $\lambda = (\lambda_0, \lambda_1)$, set to different values in order to control separately the norm of $W$ and the one of $\tau$, which is fundamental when the Gaussian is centered in zero and the optimal $W$ should be greater than $1$, while the optimal $\tau$ should be zero.

\paragraph{Influence of the time and space discretization.} A first remark in connection with the role of $\lambda$ regards the number of layers of the neural network, hence the time discretization $\dt$ step. Figure \ref{fig:different_dt} shows an experiments in dimension 2: starting from the bimodal distribution and the same initial guess $\theta^0$, the shooting method is repeated $15$ times with $\lambda=0.1$ and $\dx=0.1$. The difference between the plots in Figure \ref{fig:different_dt} is that different numbers of layers are employed, i.e., $\dt = 0.2$ and $ \dt = 0.05$, respectively from left to right. Clearly, the case with $\dt=0.05$ is the one that works best, because if $\dt$ is too large, the particles do not have enough time to reach the labels (as in the case with $\dt=0.2$, i.e $5$ layers) or they reach them, but not completely (as in the case of $10$ layers, not depicted here). These issues can clearly be overcomed by using a smaller $\lambda$, but considering the difficulty in tuning $\lambda$, it is more convenient to increase the number of layers instead. This is consistent with the common technique in the deep learning community to increase the number of layers to obtain better results.

\begin{figure}[ht]
\begin{minipage}[b]{0.5\linewidth} 
\includegraphics[scale = 0.47]{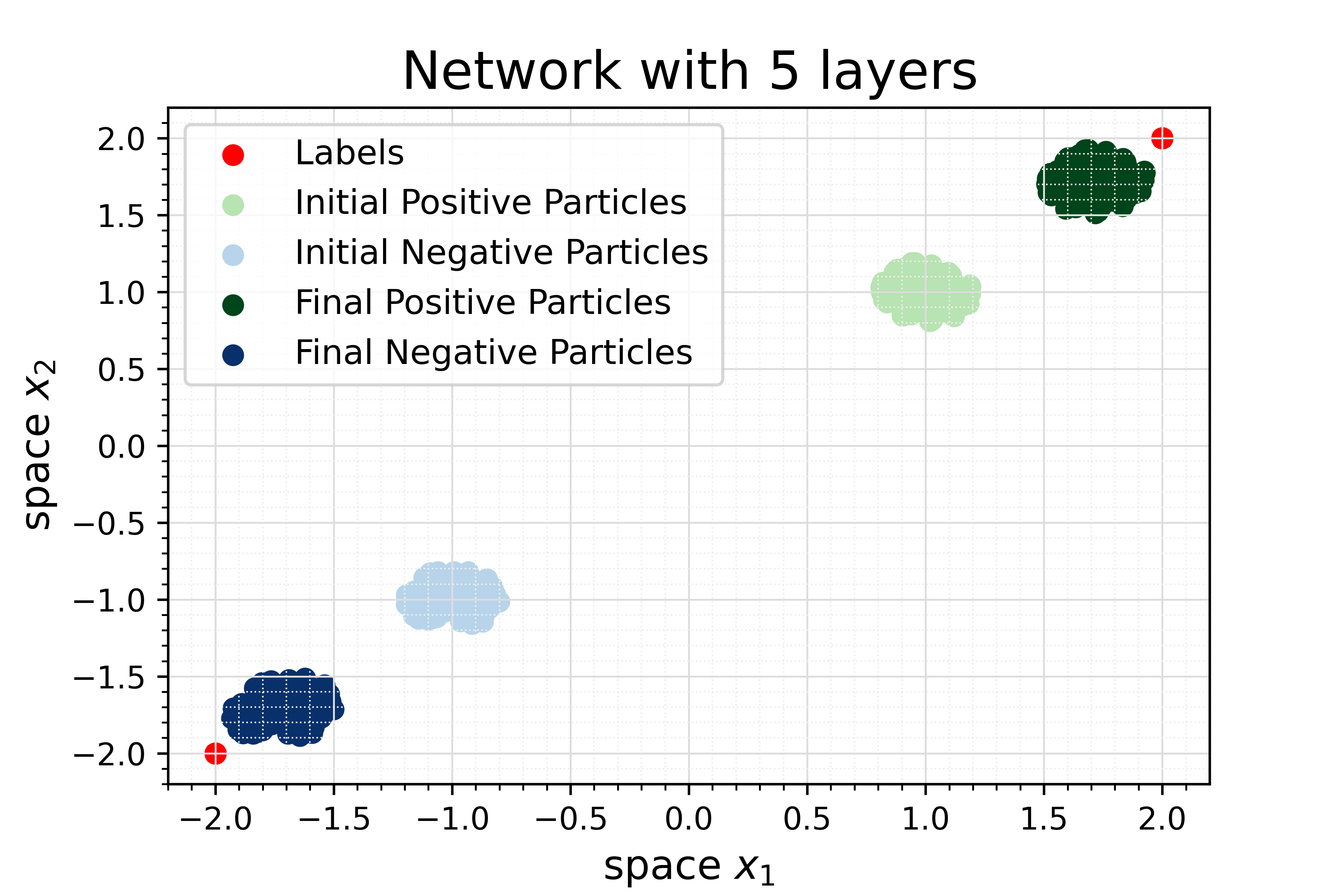}
\end{minipage}
\begin{minipage}[b]{0.5\linewidth} 
\includegraphics[scale = 0.47]{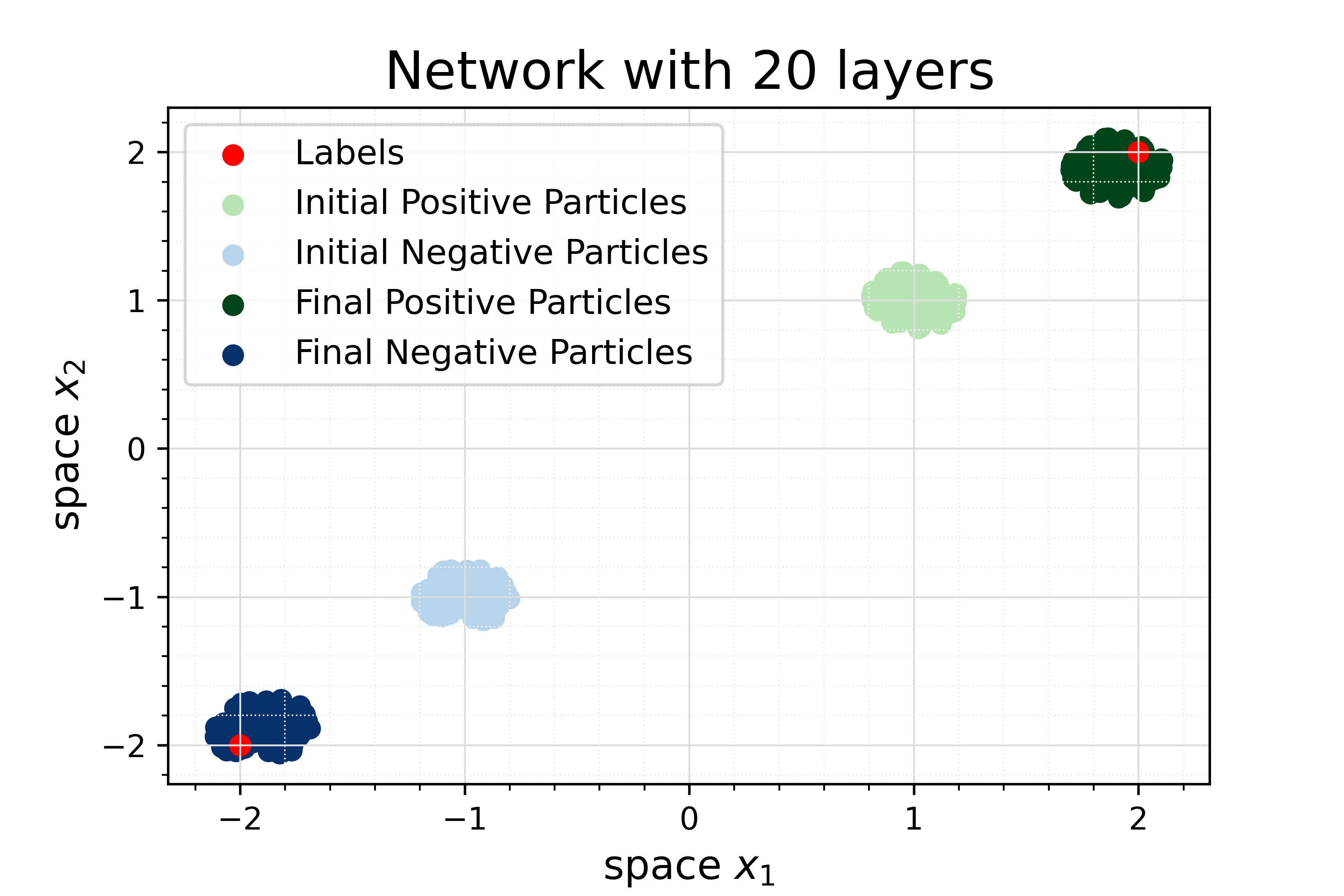}
\end{minipage}
\caption{Left: bimodal initial distribution in 2D with dt=0.2; Right: bimodal initial distribution in 2D with dt=0.05.}\label{fig:different_dt}
\end{figure}

Moreover, we need to keep in mind that the time step 
$\dt$ has to be chosen in accordance with the space step $\dx$ appearing in the backward equation as well, as the Courant number  needs to be kept below $1$ in order  for the CFL condition to be satisfied and to guarantee the convergence and stability of the numerical scheme. It is interesting to notice that in the case of unimodal distribution, increasing the space discretization to $\dx =\dy= 0.2$ is surprisingly beneficial. This is because the Courant number that needs to be set to a value between $0$ and $1$, but not too close to either of them, depends on the function $\mathcal F(t,x,\theta_t)$, and since all the particles $X^i_0$ are initially close to zero, this number tends to be too small. Hence, a better convergence rate is obtained when the space discretization is increased.

An implementation in Python of our algorithms, together with videos and code to reproduce our results, can be found at the following repository \textit{https://github.com/CristinaCipriani/Mean-fieldPMP-NeurODE-training}.


\addcontentsline{toc}{section}{Appendices}
\section*{Appendices}

In the following series of appendices, we recollect some auxiliary results appearing earlier in the paper, and detail the proofs of some intermediate steps in our previous arguments.

\setcounter{section}{0} 
\renewcommand{\thesubsection}{A}

\subsection{Well-posedness continuity equations and properties of characteristic flows}
\label{subsection:Flows}

\setcounter{thm}{0} \renewcommand{\thethm}{A.\arabic{thm}}
\setcounter{equation}{0} \renewcommand{\theequation}{A.\arabic{equation}}

\begin{proof}[Proof of Theorem \ref{thm:Wellposed}]
In what follows, we shall study qualitative properties of the ODEs
\begin{align}\label{ODE}
\frac{\rd X_t}{\dt}=\mathcal F(t,X_t,\theta_t) \qquad \text{and} \qquad \frac{\rd Y_t}{\dt}=0.
\end{align}
Since for any given $\theta\in L^2([0,T];\RR^{m})$ the velocity field $(t,x) \mapsto \mc{F}(t,x,\theta_t)$ satisfies the regularity and growth conditions of Assumption \ref{asum1}, it follows from standard results  that for any initial condition $(x_0,y_0)\in B(R)$, the above system has a unique solution $(X_t,Y_t)\in \Lip([0,T];\RR^{2d})$ on $[0,T]$. Moreover following e.g. \cite[Theorem A.2]{fornasier2014mean}, it holds that
\begin{equation}\label{support}
|X_t|\leq (R+C_{\MF}T)e^{C_{\MF}T} \quad \text{and} \qquad Y_t=y_0 \, ,
\end{equation}
for all $t \in [0,T]$. We consider the underlying characteristic flow between times $\tau,t\in[0,T]$, defined by 
\begin{equation}
\label{eq:Characteristics}
\BPhi_{(\tau,t)}^{\theta} : (x_{\tau},y_{\tau}) \in \R^{2d} \mapsto (X_t^{x_{\tau}},Y_t^{y_{\tau}}) \in \R^{2d},
\end{equation}
where $t \in [0,T] \mapsto (X^{x_0}_t,Y^{y_0}_t)$ is the unique solution of \eqref{ODE} starting from $(x_{\tau},y_{\tau}) \in \R^{2d}$ at time $\tau \in [0,T]$. Given an initial datum $\mu_0\in\mc{P}_c^a(\RR^{2d})$, we can use the characteristic flow to define the following curve of measures 
\begin{equation}
\mu_t:= \BPhi_{(0,t)}^{\theta} \sharp{\mu_0}\,,
\end{equation}
for all times $t \in [0,T]$, which equivalently means that
\begin{equation}
\int_{\RR^{2d}}\psi(t,x,y)\rd\mu_t(x,y)=\int_{\RR^{2d}} \psi(t,X_t^{x_0},Y_t^{y_0})\rd \mu_0(x_0,y_0)\,.
\end{equation}
for all $\psi\in \MC_b^1([0,T]\times \RR^{2d})$. It is well known that $\mu_t$ is a measure solution to the equation \eqref{PDE}. Indeed, using the change of variables formula for the push-forward measure, the chain rule, and once more the change of variables, one has
\begin{align}
\frac{\rd}{\dt}\int_{\RR^{2d}}\psi(t,x,y)\rd\mu_t(x,y)&=\int_{\RR^{2d}}\frac{\rd}{\dt}\psi(t,X_t^{x_0},Y_t^{y_0})\rd \mu_0(x_0,y_0)\nn\\
&=\int_{\RR^{2d}} \Big( \partial_t\psi(t,X_t^{x_0},Y_t^{y_0})+\nabla_x\psi(t,X_t^{x_0},Y_t^{y_0})\cdot \mathcal F(t,X_t^{x_0},\theta_t) \Big)\rd \mu_0(x_0,y_0)\nn\\
&=\int_{\RR^{2d}} \Big( \partial_t\psi(t,x,y)+\nabla_x\psi(t,x,y)\cdot \mathcal F(t,x,\theta_t) \Big)\rd\mu_t(x,y), 
\end{align}
and an integration with respect to the time variable leads to \eqref{eqweak}. Furthermore, it follows e.g. from \cite[Lemma 3.11]{canizo2011well} that for any $s,t\in[0,T]$, it holds
\begin{equation}
\label{Was}
W_1(\mu_t,\mu_s) = W_1 \big( \BPhi_{(0,t)}^{\theta}\sharp \mu_0,\BPhi_{(0,s)}^{\theta}\sharp\mu_0 \big) \leq \big\|\BPhi_{(0,t)}^{\theta} -\BPhi_{(0,s)}^{\theta}\big\|_{L^\infty(\supp(\mu_0))}\leq C|t-s|\,,
\end{equation}
due to the fact that
\begin{equation}
|\BPhi_{(0,t)}^{\theta}(x_0,y_0)-\BPhi_{(0,s)}^{\theta}(x_0,y_0)|=|(X_t^{x_0}-X_s^{x_0},0)|\leq C|t-s|,
\end{equation}
for all $(x_0,y_0) \in \supp(\mu_0)$, where $C$ depends only on $R$, $T$ and $C_{\MF}$. Thus, the curve $\mu$ is Lipschitz continuous with respect to $W_1$-metric, and it is such that $\supp(\mu_t) \in B(R_T)$ for all $t\in[0,T]$ as a consequence of \eqref{support}, where $R_T >0$ depends only on $R$, $T$ and $C_{\MF}$. 

Next we prove the stability estimate. For $i=1,2$, denote by $\mu^i$ be two measure solutions of \eqref{PDE} with initial data $\mu_0^i$. Introducing the notation $(X_t^i,Y_t^i) := \BPhi_{(0,t)}^{\theta}(x_0^i,y_0^i)$ for $t\in[0,T]$ and $(x^i_0,y^i_0) \in \supp(\mu^i_0)$, it holds that
\begin{align*}
\big| (X_t^1,Y_t^1)-(X_t^2,Y_t^2) \big| &= \left| \bigg( (x_0^1-x_0^2) + \int_0^t\mathcal F(s,X_s^1,\theta_s)-\mathcal F(s,X_s^2,\theta_s)\ds \, , \, y_0^1-y_0^2 \bigg) \right|\\
&\leq  |(x_0^1-x_0^2,y_0^1-y_0^2)|+\int_0^t\left|\mathcal F(s,X_s^1,\theta_s)-\mathcal F(s,X_s^2,\theta_s)\right|\ds\\
&\leq |(x_0^1,y_0^1)-(x_0^2,y_0^2)|+\int_0^tL_{\MF} (1+|\theta_s|) |X_s^1-X_s^2|\ds, 
\end{align*}
which by applying Gronwall's Lemma then leads to 
\begin{equation}
\label{Gron}
\big| (X_t^1,Y_t^1)-(X_t^2,Y_t^2) \big| \leq \big| (x_0^1,y_0^1)-(x_0^2,y_0^2) \big| e^{\int_0^tL_{\MF} (1+|\theta_s|) \ds}= |(x_0^1,y_0^1)-(x_0^2,y_0^2)|e^{L_{\MF,T,\| \theta \|_1}}, 
\end{equation}
for all times $t \in [0,T]$. This provides us with the following Lipschitz estimate  
\begin{equation}
|\BPhi_{(0,t)}^{\theta}(x_0^1,y_0^1) - \BPhi_{(0,t)}^{\theta}(x_0^2,y_0^2)|\leq  L_{\mc{T}}|(x_0^1,y_0^1)-(x_0^2,y_0^2)|,
\end{equation}
for all times $t \in [0,T]$, where $L_{\mc{T}}:=e^{L_{\MF,T,\| \theta \|_1}}$. Given an optimal transport plan $\pi_0$ between $\mu_0^1$ and $\mu_0^2$, one can check that the measure $\pi :=(\BPhi_{(0,t)}^{\theta} \times \BPhi_{(0,t)}^{\theta})\sharp \pi$ has marginals $\BPhi_{(0,t)}^{\theta} \sharp\mu_0^1$ and $\BPhi_{(0,t)}^{\theta} \sharp\mu_0^2$. Whence, it holds 
\begin{align*}
W_1 \Big( \BPhi_{(0,t)}^{\theta} \sharp\mu_0^1,\BPhi_{(0,t)}^{\theta} \sharp \mu_0^2 \Big) & \leq \int_{\RR^{2d}\times\RR^{2d}}|x-y|\rd \gamma(x,y) \\
& = \int_{\RR^{2d}\times\RR^{2d}} \big| \BPhi_{(0,t)}^{\theta}(x^1_0,y^1_0) - \BPhi_{(0,t)}^{\theta}(x^2_0,y^2_0) \big| \rd \pi(x^1_0,y^1_0,x^2_0,y^2_0) \\
& \leq L_{\mc T}\int_{\RR^{2d}\times\RR^{2d}} \big| (x^1_0,y^1_0) - (x^2_0,y^2_0) \big|\rd \pi(x^1_0,y^1_0,x^2_0,y^2_0) \\
& =L_{\mc T} \, W_1(\mu_0^1,\mu_0^2)\,,
\end{align*}
which leads to 
\begin{align}
W_1(\mu_t^1,\mu_t^2)=W_1(\BPhi_{(0,t)}^{\theta} \sharp \mu_0^1, \BPhi_{(0,t)}^{\theta} \sharp\mu_0^2)\leq L_{\mc{T}}W_1(\mu_0^1,\mu_0^2), 
\end{align}
for all times $t \in [0,T]$, and completes the proof of Theorem \ref{thm:Exist}.
\end{proof}

\begin{proof}[Proof of Proposition \ref{prop2}]
We shall use the standard characteristic method with backward propagation. For any terminal condition $(X_T,Y_T)=(x,y)\in B(R_T)$, we know thanks to  the classical Cauchy-Lipschitz theory that the ODEs
\begin{align}
\label{ODE1}
\frac{\rd X_t}{\dt}=\mathcal F(t,X_t,\theta_t) \qquad \text{and} \qquad \frac{\rd Y_t}{\dt}=0,
\end{align}
admit a unique solution $t \in [0,T] \mapsto (X_t,Y_t) := \BPhi_{(T,t)}^{\theta}(x,y) \in \R^{2d}$ which can be written explicitly as 
\begin{equation}
\label{eq:BackwardFlow}
\BPhi_{(T,t)}^{\theta}(x,y) = \bigg( x - \int_t^T\mathcal F(s,X_s,\theta_s)\ds \, , \, y \bigg). 
\end{equation}
for all $(x,y)\in B(R_T)$. Moreover, one has  that
\begin{equation*}
\big| \BPhi_{(T,t)}^{\theta}(x,y) \big| \leq (R_T+C_{\MF}T)e^{C_{\MF}T}+R_T
\end{equation*}
by Gronwall's inequality as in \eqref{support}, which equivalently means that $\BPhi_{(T,t)}^{\theta}(B(R_T)) \subset B(R_T)$ with $R_T' := R + (R+C_{\MF}T)e^{C_{\MF}T}$. Furthermore under Assumptions \ref{asum1} and \ref{asum2}, the functions $\BPhi_{(T,t)}^{\theta} :~\RR^{2d}\to\RR^{2d}$ are $\MC^2$ diffeomorphisms for any $t\in[0,T]$, and the application $(t,x,y)\mapsto \BPhi_{(T,t)}^{\theta}(x,y) \in \R^{2d}$ is locally Lipschitz. 

Building on these insights, we can construct solutions of \eqref{eq2} via the standard characteristic method, by setting
\begin{equation}
\psi^\theta(t,x,y):=\psi_T \big( \BPhi_{(T,t)}^{\theta}(x,y) \big),
\end{equation}
for all $(t,x,y)\in [0,T] \RR^{2d}$, where $\psi_T\in \MC_c^2(\RR^{2d})$ satisfies \eqref{tem}. This implies that in particular that 
\begin{equation*}
\psi^\theta \big( t,\BPhi_{(T,t)}^{\theta}(x,y) \big)=\psi_T(x,y), 
\end{equation*}
for all times $t \in [0,T]$, from whence we can deduce 
\begin{align*}
0 & =\tfrac{\rd}{\dt}\psi^\theta(t,X_t,\overline Y_t)\\
& =\partial_t\psi^\theta(t,\overline X_t,Y_t)+\nabla_x\psi^\theta(t,X_t,Y_t)\cdot \tfrac{\rd X_t}{\dt} \\ 
& =  \Big( \partial_t\psi^\theta + \nabla_x\psi^\theta\cdot \MF \Big) \big( t, \BPhi_{(0,t)}^{\theta}(x,y) \big).
\end{align*}
 for any $t \in [0,T)$ and $(x,y)\in \RR^{2d}$. Since $\supp(\psi_T)=B(R_T)$, one has that
\begin{equation*}
\supp(\psi^{\theta}(t)) = \BPhi_{(T,t)}^{\theta}(B(R_T))\subset B(R_T')
\end{equation*}
for all times $t\in[0,T]$. Thus, we have constructed a function $\psi^\theta(t,x,y) = \psi_T \big( \BPhi_{(T,t)}^{\theta}(x,y) \big)$  of class $\mc{C}^1([0,T];\MC_c^2(\RR^{2d}))$ satisfying \eqref{eq2}. 

At this stage by considering the analytical expression \eqref{eq:BackwardFlow}, it follows from arguments similar to those leading to \eqref{Gron} that 
\begin{equation*}
\Big| \BPhi_{(T,t)}^{\theta}(x_1,y_1)-\BPhi_{(T,t)}^{\theta}(x_2,y_2) \Big|\leq (|x_1-x_2|+|y_1-y_2|)e^{L_{\MF,T,C_{\Gamma}}T}, 
\end{equation*} 
which combined Assumption \ref{asum2}-$(i)$, according to \cite[Lemma 2.3]{ying2006phase} further implies that
\begin{equation}
\big\| \BPhi_{(T,t)}^{\theta} \big\|_{\MC^2(\BPhi_{(T,t)}^{\theta}(B(R_T)))} \leq C(R_T',T,C_{\Gamma},C_{\MF},L_{\MF,T,C_{\Gamma}})\,.
\end{equation}
Thus we have for all $t\in[0,T]$
\begin{equation}
\label{psi_t_norm}
\begin{aligned}
\big\| \psi_t^\theta \big\|_{\MC_c^2(\RR^{2d})} = \big\| \psi_t^\theta \big\|_{\MC^2(\BPhi_{(T,t)}^{\theta}(B(R_T)))} & = \big\| \psi_T(\BPhi_{(T,t)}^{\theta}) \big\|_{\MC^2( \BPhi_{(T,t)}^{\theta}(B(R_T)))} \\
& \leq  C\left( \big\| \BPhi_{(T,t)}^{\theta} \big\|_{\MC^2(\BPhi_{(T,t)}^{\theta}(B(R_T)))}\right)\|\psi_T\|_{\MC^2(B(R_T))}, 
\end{aligned}   
\end{equation}
which concludes the proof of \eqref{rad}. 
\end{proof}

We now end this first appendix section by detailing the proof of Lemma \ref{lem:AdjointFlow}.

\begin{proof}[Proof of Lemma \ref{lem:AdjointFlow}]
By construction of the semigroups $(\Phi_{(\tau,t)}^{\theta})_{\tau,t \in [0,T]}$, it holds for all $(t,x) \in [0,T] \times \R^d$ that
\begin{equation}
\label{eq:FlowEq1}
\Phi_{(t,T)}^{\theta} \circ \Phi_{(T,t)}^{\theta}(x) = x,
\end{equation}
where ``$\circ$'' stands for the standard composition operation between functions. Thus by differentiating with respect to $x \in \R^d$ in \eqref{eq:FlowEq1}, we obtain 
\begin{equation*}
\nabla_x \Phi_{(t,T)}^{\theta}(\Phi_{(T,t)}^{\theta}(x)) \nabla_x \Phi_{(T,t)}^{\theta}(x) = \operatorname{Id}, 
\end{equation*}
for every $y \in \R^d$. Thus, recalling that $\nabla_x \Phi_{(T,t)}^{\theta}(x)$ is invertible by construction, one further has
\begin{equation}
\label{eq:FlowEq2}
\nabla_x \Phi_{(t,T)}^{\theta} (\Phi_{(T,t)}^{\theta}(x)) = \nabla_x \Phi_{(T,t)}^{\theta}(x)^{-1}, 
\end{equation}
for every $(t,x) \in [0,T] \times \R^d$. Differentiating with respect to $t \in [0,T]$ in \eqref{eq:FlowEq2} while recalling the ODE characterization derived in \eqref{eq:linearizedFlow} for $t \in [0,T] \mapsto \nabla_x \Phi_{(T,t)}^{\theta}(x)$ then yields
\begin{equation*}
\begin{aligned}
\partial_t \Big( \nabla_x \Phi_{(t,T)}^{\theta} (\Phi_{(T,t)}^{\theta}(x)) \Big) & = - \nabla_x \Phi_{(T,t)}^{\theta}(x)^{-1} \partial_t \Big( \nabla_x \Phi_{(T,t)}^{\theta}(x) \Big) \nabla_x \Phi_{(T,t)}^{\theta}(x)^{-1} \\
& = -\nabla_x \Phi_{(T,t)}^{\theta}(x)^{-1} \nabla_x \Fcal \big(t,\Phi_{(T,t)}^{\theta}(x),\theta_t \big) \\
& = -\nabla_x \Phi_{(t,T)}^{\theta} (\Phi_{(T,t)}^{\theta}(x)) \nabla_x \Fcal \big(t,\Phi_{(T,t)}^{\theta}(x),\theta_t \big), 
\end{aligned}
\end{equation*}
where we used the classical characterization of the differential of the inverse mapping over matrices. Taking the transpose in the previous expression while using the fact that the process of adjoining a matrix is linear, we can conclude that 
\begin{equation*}
\left\{
\begin{aligned}
\partial_t \Big( \nabla_x \Phi_{(t,T)}^{\theta} (\Phi_{(T,t)}^{\theta}(x))^{\top} \Big) & = - \nabla_x \Fcal \big(t,\Phi_{(T,t)}^{\theta}(x),\theta_t \big)^{\top} \nabla_x \Phi_{(t,T)}^{\theta} (\Phi_{(T,t)}^{\theta}(x))^{\top}, \\
\nabla_x \Phi_{(T,T)}^{\theta} (\Phi_{(T,T)}^{\theta}(x))^{\top} & = \operatorname{Id}, 
\end{aligned}
\right.
\end{equation*}
which ends the proof of our claim. 
\end{proof}


\setcounter{section}{0} 
\renewcommand{\thesubsection}{B}

\subsection{Regularity of ODE flows with respect to the control variables}
\label{subsection:ControlFlow}

\setcounter{thm}{0} \renewcommand{\thethm}{B.\arabic{thm}}
\setcounter{equation}{0} \renewcommand{\theequation}{B.\arabic{equation}}

In this second Appendix section, we recollect somewhat elementary results concerning the regularity of the flows of diffeomorphisms $(\Phi^{\theta}_{(0,t)})_{t \in [0,T]} \subset \Ccal(\R^d,\R^d)$ defined in \eqref{eq:ThetaFlow} with respect to the control variable $\theta \in L^2([0,T],\R^m)$.

\begin{proposition}[Lipschitz and supremum bound for controlled flows]
\label{prop:BoundFlow}
For any given $T>0$, suppose that $\Fcal$ satisfies Assumptions \ref{asum1} and \ref{asum2}. Then for every $R > 0$ and any pair of control signals $\theta_1,\theta_2 \in L^2([0,T],\R^m)$, there exists a constant $C(T,R,\| \theta^1 \|) > 0$ such that 
\begin{equation*}
\sup_{t \in [0,T]} \big\| \Phi^{\theta^1}_{(0,t)} \big\|_{\Ccal(B(R))} \leq C(T,R,\| \theta^1 \|_1)
\end{equation*}
and
\begin{equation*}
\sup_{t \in [0,T]} \big\| \Phi^{\theta^1}_{(0,t)} -\Phi^{\theta^2}_{(0,t)} \big\|_{\Ccal(B(R))} \leq C(T,R,\| \theta^1 \|_1) \| \theta^1 - \theta^2 \|_2.
\end{equation*}
\end{proposition}

\begin{proof}
These estimates follows from our quantitative regularity assumptions together with a standard application of Gr\"onwall's lemma.  
\end{proof}

\begin{proposition}[Regularity of the flow with respect to the control variable] 
\label{prop:DiffFlow}
For any given $T>0$, suppose that $\Fcal$ satisfies Assumptions \ref{asum1} and \ref{asum2}. Then for every $\theta,\vartheta \in L^2([0,T],\R^m)$, the following Taylor expansion 
\begin{equation}
\label{eq:GDiffFlow}
\Phi^{\theta + \varepsilon \vartheta}_{(0,t)}(x) = \Phi^{\theta}_{(0,t)}(x) + \varepsilon \INTSeg{\Rcal^{\theta}_{(s,t)}(x) \nabla_{\theta} \Fcal \big(t,\Phi_{(0,s)}^{\theta}(x),\theta_s \big) \vartheta_s}{s}{0}{t} + o_{\theta}(\varepsilon), 
\end{equation}
holds in $\Ccal([0,T] \times B(R),\R^{2d})$, where for each $(\tau,x) \in [0,T] \times \in \R^d$ the resolvent map $t \in [0,T] \mapsto \Rcal_{(\tau,t)}^{\theta}(\cdot) \in \Ccal^1(\R^d;\R^{d \times d})$ is the unique solution of the linearized Cauchy problem
\begin{equation}
\label{eq:Resolvent}
\left\{
\begin{aligned}
\partial_t \Rcal_{(\tau,t)}^{\theta}(x) & = \nabla_x \Fcal \big( t , \Phi_{(0,t)}^{\theta}(x) , \theta_t \big) \Rcal_{(\tau,t)}^{\theta}(x), \\
\Rcal_{(\tau,\tau)}^{\theta}(x) &  = \operatorname{Id}.
\end{aligned}
\right.
\end{equation}
Moreover, for any $\theta^1,\theta^2 \in L^2([0,T],\R^m)$, there exists a constant $C'(T,R,\| \theta^1 \|_1) > 0$ such that
\begin{equation}
\label{eq:ResolventEst1}
\sup_{t \in [0,T]} \big\| \Rcal^{\theta^1}_{(0,t)} \big\|_{\Ccal(B(R),\R^{d \times d})} \leq C'(T,R,\| \theta^1 \|_1)
\end{equation}
and 
\begin{equation}
\label{eq:ResolventEst2}
\sup_{t \in [0,T]} \big\| \Rcal^{\theta^1}_{(0,t)} - \Rcal^{\theta^2}_{(0,t)} \big\|_{\Ccal(B(R),\R^{d \times d})} \leq C'(T,R,\| \theta^1 \|_1) \| \theta^1 - \theta^2 \|_2.
\end{equation}
In particular, the map $\theta \in L^2([0,T],\R^m) \mapsto \Phi^{\theta} \in C^0([0,T] \times B(R),\R^{2d})$ is Fr\'echet-differentiable. 
\end{proposition}

\begin{proof}
By reproducing the parametrised fixed-point argument detailed in \cite[Theorem 2.3.1]{BressanPiccoli}, one can prove that the following Taylor expansion 
\begin{equation}
\label{eq:ProofGdiff1}
\Phi^{\theta + \varepsilon \vartheta}_{(0,t)}(x) = \Phi^{\theta}_{(0,t)}(x) + \varepsilon  \Psi^{\theta,\vartheta}_{(0,t)}(x) + o_{\theta}(\varepsilon)
\end{equation}
holds for all $(t,x) \in [0,T] \times B(R)$ and each $\varepsilon > 0$, where the map $t \in [0,T] \mapsto \Psi^{\theta,\vartheta}_{(0,t)}(x) \in \R^d$ is the unique solution of the linearized Cauchy problem
\begin{equation}
\label{eq:ProofGdiff2}
\left\{
\begin{aligned}
\partial_t \Psi^{\theta,\vartheta}_{(0,t)}(x) & = \nabla_x \Fcal \big( t , \Phi_{(0,t)}^{\theta}(x) , \theta_t \big) \Psi^{\theta,\vartheta}_{(0,t)}(x) + \nabla_{\theta} \Fcal \big( t , \Phi_{(0,t)}^{\theta}(x) , \theta_t \big) \vartheta_t, \\
\Psi^{\theta,\vartheta}_{(0,0)}(x) & = 0.
\end{aligned}
\right.
\end{equation}
By a simple application of the constant variation formula (see e.g. \cite[Theorem 2.2.3]{BressanPiccoli}), it can be shown that it can in fact be expressed as 
\begin{equation*}
\Psi^{\theta,\vartheta}_{(0,t)}(x) = \INTSeg{\Rcal^{\theta}_{(s,t)}(x) \nabla_{\theta} \Fcal \big( s , \Phi^{\theta}_{(0,s)}(x), \theta_s \big) \vartheta_s}{s}{0}{t},
\end{equation*}
for all times $t \in [0,T]$, where the resolvent map $t \in [0,T] \mapsto \Rcal_{(\tau,t)}^{\theta}(x) \in \R^{d \times d}$ is defined as in \eqref{eq:Resolvent}. The regularity bounds displayed in \eqref{eq:ResolventEst1}-\eqref{eq:ResolventEst2} easily follow by combining the regularity hypotheses of Assumption \ref{asum1} and \ref{asum2} with the arguments detailed in \cite[Theorem 2.2.4]{BressanPiccoli}.
\end{proof}


\setcounter{section}{0} 
\renewcommand{\thesubsection}{C}

\subsection{Proof of Theorem \ref{thmla}}
\label{subsection:Lagrange}

\setcounter{thm}{0} \renewcommand{\thethm}{C.\arabic{thm}}
\setcounter{equation}{0} \renewcommand{\theequation}{C.\arabic{equation}}

In this third appendix section, we provide a proof of the abstract Lagrange multiplier rule stated in Theorem \ref{thmla}. 

\hspace{-0.75cm} $\bullet$ \textbf{Step 1.} We first want to show that 
\begin{equation}
G'(x^\ast)h=0 \qquad \mbox{ implies } \qquad DJ(x^\ast)h=0\, ,
\end{equation}
for all $h\in \overline {X}_E$. To this end, let $h\in \overline {X}_E$ be given such that $G'(x^*)h=0$. Here $DJ(x^\ast)$ is the multivalued $F$-differential of $J$ at $x^*$ as in Definition \ref{def:Fdiff}. Consider the operator
\begin{equation}
\Psi(\varepsilon,u):=\overline G(x^\ast+\varepsilon h+u)\,,
\end{equation}
where $(\varepsilon,u)$ is in some neighborhood of $(0,0)$ in $\RR\times \overline {X}_E$, and $\overline G$ is the unique extension of $G$ to $\overline E$. Indeed, for any $h,u\in \overline {X}_E$, there exists sequences $(h^n)_{n \in \NN},(u^n)_{n \in \NN} \subset X_E$ such that $h^n\to h$ and $u^n\to u$. According to the assumption it necessarily holds that $(x^*+\varepsilon h^n+u^n) \in x^*+X_E\subset  E$, so one can uniquely define
\begin{equation}
\Psi(\varepsilon,u)=\overline G(x^\ast+\varepsilon h+u):=\lim_{n\to\infty} G(x^\ast+\varepsilon h^n+u^n)\,.
\end{equation} 
In the sequel we will not differentiate $G$ from $\overline G$.
 
Note that if $x^*$ solves \eqref{eqopt}, one has
\begin{equation}
\Psi(0,0)=G(x^\ast)=0\,.
\end{equation}
By the definition of $F$-derivatives, we note that
\begin{equation}
\lim_{y\to 0}\frac{\norm{\Psi(0,y)-\Psi(0,0)-G'(x^\ast)y}_Y}{\norm{y}_X}=\lim_{y\to 0}\frac{\norm{G(x^*+y)-G(x^*)-G'(x^\ast)y}_Y}{\norm{y}_X}=0.
\end{equation}
This means that $G'(x^*)\in D\Psi_u(0,0)$. Thus there exists some $\Psi'_u(0,0)\in D\Psi_u(0,0)$ such that 
\begin{equation}
\Psi'_u(0,0)=G'(x^\ast), 
\end{equation}
Moreover $\Psi'_u(0,0)$ is surjective on $\overline {X}_E\rightarrow Y$, since $G'(x^\ast)$ is  surjective on $\overline {X}_E\rightarrow Y$. 

\smallskip

\hspace{-0.7cm} $\circ$ \textit{Step 1.1.} From above, we know that $\Psi'_u(0,0)$ is surjective on $\overline {X}_E\rightarrow Y$. Thus, there exists a number $\kappa>0$ such that, for each $y\in Y$, there is a point $\omega(y)\in \overline {X}_E\subset X$ satisfying 
\begin{equation}
\Psi'_u(0,0)\omega(y)=y \qquad \mbox{and} \qquad \|\omega(y)\|_X\leq \kappa\|y\|_Y\,,
\end{equation}
where  the second inequality follows from Banach's continuous inverse theorem. We define 
\begin{equation}
f(\varepsilon,u):=\Psi'_u(0,0)u-\Psi(\varepsilon,u)\,.
\end{equation}
Let $\varepsilon\leq \rho$ and $\|u\|_X,\|v\|_X\leq r$, and observe that for some $f_u'(\varepsilon,u)\in Df_u(\varepsilon,u)$, it holds
\begin{equation}
f'_u(\varepsilon,u)=\Psi'_u(0,0)-\Psi'_u(\varepsilon,u)\,.
\end{equation}
Since $f'_u(\varepsilon,u)$ is continuous at $(0,0)$ and $f'_u(0,0)=0$, Taylor's theorem implies that
\begin{equation}\label{94}
\|f(\varepsilon,u)-f(\varepsilon,v)\|\leq \sup\limits_{0\leq\tau\leq 1}\|f'_u(\varepsilon,u+\tau(v-u))\|\|u-v\|_X=o(1)\|\|u-v\|_X \, ,
\end{equation}
as $\rho,r\to0$. In addition since $f(0,0)=0$ and $f$ is continuous at $(0,0)$, we also get
\begin{align}\label{95}
\|f(\varepsilon,u)\|_Y\leq \|f(\varepsilon,u)-f(\varepsilon,0)\|_Y+\|f(\varepsilon,0)\|_Y\leq o(1)\|u\|_X+\|f(\varepsilon,0)\|_Y \, ,
\end{align}
as $\rho,r\to0$. For a given $\varepsilon\in\RR_+$ with $\varepsilon<\rho$, we consider following iterative method
\begin{equation}\label{96}
\Psi'_u(0,0)u_{m+1}=f(\varepsilon,u_m),\quad m=0,1,2,\cdots\,,
\end{equation}
where $u_0=0$ and $u_{m+1}=\omega(f(\varepsilon,u_m))$. Since $\|u_{m+1}\|_X\leq \kappa\|f(\varepsilon,u_m)\|_Y$, it follows from \eqref{94} and \eqref{95} that for sufficiently small $\rho$ and $r$, one has
\begin{equation}
\|u_m\|_X\leq o(1)r+o(1),~\rho\to 0\quad\mbox{ and }\quad \|u_{m+2}-u_{m+1}\|_X\leq \frac{1}{2} \|u_{m+1}-u_{m}\|_X\mbox{ for all }m=0,1,\cdots\,,
\end{equation}
which means that $\{u_m\}_{m\geq 0}$ is a Cauchy sequence in the Banach space $\overline {X}_E$, and hence there exists some $u\in \overline {X}_E $ such that
\begin{equation}
u_m\to u\mbox{ as }m\to\infty\,.
\end{equation}
Moreover we have that $\|u\|_X\leq r$ and $\Psi'_u(0,0)u=f(\varepsilon,u)$ because of \eqref{96}, and thus $\Psi(\varepsilon,u)=0$. Lastly, we let $m\to \infty$ in
\begin{equation}
\|u_{m+2}\|_X\leq \kappa\|f(\varepsilon,u_{m+1})\|_Y=\kappa\norm{\Psi'_u(0,0)u_{m+1}-\Psi(\varepsilon,u_{m+1})}_Y\,,
\end{equation}
then it follows that $\|u\|_X\leq \kappa\|\Psi'_u(0,0)u\|_Y$.

\hspace{-0.75cm} $\circ$ \textit{Step 1.2.}  It follows from Step 1.1 above that there exists numbers $\rho>0$ and $r>0$ such that for any $\varepsilon\in\RR_+$ and $\varepsilon\leq \rho$, there exists $u(\varepsilon)\in  \overline {X}_E$ with $\norm{u(\varepsilon)}_X\leq r$ such that
\begin{equation}
\Psi(\varepsilon,u(\varepsilon))=G(x^\ast+\varepsilon h+u(\varepsilon))=0
\end{equation}
and
\begin{equation}\label{100}
\norm{u(\varepsilon)}_X\leq \kappa\norm{\Psi'_u(0,0)u(\varepsilon)}_Y=\kappa\norm{G'(x^\ast)u(\varepsilon)}_Y
\end{equation}
along with $\norm{u(\varepsilon)}_X\to 0$ as $\varepsilon\to 0$.

By the definition of $F$- derivative, one has
\begin{equation}
G(x^\ast+k)-G(x^\ast)-G'(x^\ast)k=o(\norm{k}_X),\quad k\to 0\,.
\end{equation}
Let $k=\varepsilon h+u(\varepsilon)$, we have
\begin{equation}
G(x^\ast+\varepsilon h+u(\varepsilon))-G(x^\ast)-\varepsilon G'(x^\ast)h-G'(x^\ast)u(\varepsilon)=o(\norm{\varepsilon h+u(\varepsilon)}_X),\quad \varepsilon\to 0\,.
\end{equation}
Therefore
\begin{equation}
G'(x^\ast)u(\varepsilon)=o(1)\norm{\varepsilon h+u(\varepsilon)}_X,\quad \varepsilon\to 0\,.
\end{equation}
By \eqref{100}, we obtain $\norm{u(\varepsilon)}_X\leq o(1)\norm{\varepsilon h+u(\varepsilon)}_X$, which is
\begin{equation}
\norm{u(\varepsilon)}=o(\varepsilon),\quad \varepsilon\to 0\,.
\end{equation}
Since $x^\ast$ is the minimizer of $J$, one has
\begin{equation}
J(x^\ast+\varepsilon h+u(\varepsilon))\geq J(x^\ast)\,,
\end{equation}
which yields
\begin{equation}
DJ(x^\ast)(\varepsilon h+u(\varepsilon))+o(\norm{\varepsilon h+u(\varepsilon)}_X)\geq 0,\quad \varepsilon\to 0\,.
\end{equation}
Dividing by $\varepsilon$ and letting $\varepsilon \to \pm0$, one has $DJ(x^\ast)h\geq 0$ and $DJ(x^\ast)h\leq 0$. In other words 
\begin{equation}
DJ(x^\ast)h=0\,.
\end{equation}

\medskip

\hspace{-0.75cm} $\bullet$ \textit{Step 2.}
In Step 1 we have proven that if $G'(x^\ast)h=0$ for some $h\in  \overline {X}_E$, then $DJ(x^\ast)h=0$. This can be written in the more compact operator form
\begin{equation}
DJ(x^\ast)\subset [\mc{N}(G'(x^\ast))]^\perp= \Big\{x'\in  \overline {X}_E' ~|~ \la x', h\ra=0 \mbox{ for all } h\in \mc{N}(G'(x^\ast))\subset  \overline {X}_E \Big\}\,.
\end{equation}
Then, it follows from the closed range theorem in Banach spaces that
\begin{equation}
[ \mc{N}(G'(x^\ast))]^\perp = \mc{R}(G'(x^\ast)^\top)\,.
\end{equation}
which implies that 
\begin{equation*}
DJ(x^\ast)\subset  \mc{N}(G'(x^\ast))^\perp= \mc{R}(G'(x^\ast)^\top)\,.
\end{equation*}
Therefore, there exists a covector $p^*\in Y'$ such that $J'(x^\ast)=G'(x^\ast)^\top p^*$ for any $J'(x^\ast)\in DJ(x^\ast)$. In other words
\begin{equation}
\la J'(x^\ast), z\ra =\la  G'(x^\ast)^\top p^*,z   \ra=\la p^*, G'(x^\ast)z\ra\quad \mbox{ for all }z\in  \overline {X}_E\,,
\end{equation}
which completes the proof of Theorem \ref{thmla}.

\section*{Acknowledgments}

C.C., H.H., and M.F. acknowledge the support of the DFG Project "Identification of Energies from Observation of Evolutions" and the DFG SPP 1962 "Non-smooth and Complementarity-based Distributed Parameter Systems: Simulation and Hierarchical Optimization". C.C. and M.F. acknowledge also the partial support of the project  
``Online Firestorms And Resentment Propagation On Social Media: Dynamics, Predictability and Mitigation'' of the TUM Institute for Ethics in Artificial Intelligence.


{\small 
\bibliographystyle{plain}
\bibliography{references,control,biblioflock}
}

\end{document}